\documentclass[letterpaper,10pt]{article}

\usepackage[top=1in, bottom=1in, left=1in, right=1in]{geometry}
\usepackage{latexsym}

\usepackage{amsmath}
\usepackage{amssymb}
\usepackage{amsthm}
\usepackage{epsfig}
\usepackage{xcolor}
\usepackage{graphicx}
\usepackage{bm}
\usepackage{hyperref}
\usepackage{xifthen}
\usepackage{wrapfig}

\usepackage{amsmath,amsfonts,amsthm,amssymb}
\usepackage{algorithmic}
\usepackage{algorithm}
\usepackage{array}
\usepackage[caption=false,font=normalsize,labelfont=sf,textfont=sf]{subfig}
\usepackage{textcomp}
\usepackage{stfloats}
\usepackage{url}
\usepackage{verbatim}
\usepackage{graphicx}

\def\shownotes{1}  %set 1 to show author notes
\ifnum\shownotes=1
\newcommand{\authnote}[2]{{\scriptsize $\ll$\textsf{#1 notes: #2}$\gg$}}
\else
\newcommand{\authnote}[2]{}
\fi

\newtheorem{theorem}{Theorem}

\newtheorem{lemma}{Lemma}
\newtheorem{proposition}{Proposition}
\newtheorem{definition}{Definition}
\newtheorem{assumption}{Assumption}

\newtheorem{claim}{Claim}

\newtheorem{formalism}{Formalism}
\newtheorem{conjecture}{Conjecture}

\usepackage{hyperref}
\hypersetup{
  colorlinks,
  citecolor=blue!70!black,
  linkcolor=blue!70!black,
  urlcolor=blue!70!black
}
\usepackage{cleveref}  
    
\renewcommand{\P}{\mathbb{P}}
\newcommand{\E}{\mathbb{E}}
\newcommand{\cN}{\mathcal{N}}
\newcommand{\tbeps}{\boldsymbol{\tilde\eps}}
\newcommand{\R}{\mathbb{R}}
\newcommand{\C}{\mathbb{C}}
\newcommand{\N}{\mathbb{N}}
\newcommand{\rvline}{\hspace*{-\arraycolsep}\vline\hspace*{-\arraycolsep}}

\newcommand{\<}{\langle}
\renewcommand{\>}{\rangle}
\newcommand{\tr}{\text{tr}}
\newcommand{\op}{{\rm op}}
\newcommand{\ones}{\bm{1}}
\newcommand{\what}{\widehat}
\def\sT{{\mathsf T}}
\def\bzero{{\boldsymbol 0}}
\def\simiid{\sim_{i.i.d.}}
\def\sign{{\rm sign}}
\def\de{{\rm d}}
\def\bd{{\boldsymbol d}}

\def\br{{\boldsymbol r}}

\def\bu{{\boldsymbol u}}

\def\bx{{\boldsymbol x}}
\def\by{{\boldsymbol y}}
\def\bz{{\boldsymbol z}}

\def\bC{{\boldsymbol C}}

\def\bG{{\boldsymbol G}}

\def\bQ{{\boldsymbol Q}}

\def\bT{{\boldsymbol T}}

\def\bX{{\boldsymbol X}}
\def\bY{{\boldsymbol Y}}
\def\bZ{{\boldsymbol Z}}

\def\beps{{\boldsymbol \eps}}
\def\bepsilon{{\boldsymbol \epsilon}}
\def\bmu{{\boldsymbol \mu}}
\def\bxi{{\boldsymbol \xi}}
\def\btheta{{\boldsymbol \theta}}
\def\bbeta{{\boldsymbol \beta}}

\def\bbeta{{\boldsymbol \beta}}

\def\tp{{\tilde p}}

\def\cB{{\mathcal B}}

\def\cD{{\mathcal D}}
\def\cE{{\mathcal E}}

\def\cH{{\mathcal H}}

\def\cL{{\mathcal L}}
\def\cM{{\mathcal M}}
\def\cN{{\mathcal N}}

\def\cP{{\mathcal P}}

\def\cS{{\mathcal S}}

\def\cX{{\mathcal X}}
\def\cY{{\mathcal Y}}

\newcommand{\eps}{\varepsilon} 

\newcommand{\la}{\langle}
\newcommand{\ra}{\rangle}

\newcommand{\txb}{\tilde\xb}

\def\id{{\mathbf I}}

\def\ones{{\mathsf 1}}

\def\bT{{\boldsymbol T}}
\def\cB{{\mathcal B}}
\def\MI{{\rm MI}}

\def\xb{{\mathbf x}}

\def\Yb{{\mathbf Y}}
\def\Xb{{\mathbf X}}

\def\Null{{\rm null}}
\def\ridge{{\rm ridge}}

\DeclareMathOperator*{\argmin}{arg\,min}
\DeclareMathOperator*{\argmax}{arg\,max}

\def\FDP{{\tt FDP}}
\def\TPP{{\tt TPP}}
\def\FD{{\tt FD}}
\def\TP{{\tt TP}}
\def\Rej{{\tt R}}

\def\TD{{\tt TD}}

\def\TPR{{\tt TPR}}
\def\mFDR{{\tt mFDR}}

\def\mTPR{{\tt mTPR}}
\def\BFDR{{\tt BFDR}}
\def\BTPR{{\tt BTPR}}

\def\FDR{{\tt FDR}}

\def\TPoP{{{\tt{TPoP}}}}
\def\CPoP{{{\tt {CPoP}}}}
\def\CRT{{{\tt{CRT}}}}
\def\dCRT{{{\tt{dCRT}}}}
\def\BH{{{\tt{BH}}}}
\def\eBH{{{\tt{eBH}}}}
\def\PoPCe{{{\tt{PoPCe}}}}
\def\PoEdCe{{{\tt{PoEdCe}}}}
\def\EPoEdCe{{{\tt{EPoEdCe}}}}

\def\ProbFam{{\mathcal P}}
\def\PostProb{{\mathcal P}}

\def\Xb{{\mathbf{X}}}

\newcommand{\indep}{\perp \!\!\! \perp}

\newcommand{\red}[1]{\textcolor{red}{#1}}
\newcommand{\blue}[1]{\textcolor{blue}{#1}}

\begin{document}

\title{Near-optimal multiple testing in Bayesian linear models with finite-sample FDR control}

\makeatletter
\def\blfootnote{\gdef\@thefnmark{}\@footnotetext}
\makeatother

\makeatletter
\def\@fnsymbol#1{\ensuremath{\ifcase#1\or \dagger\or %\dagger\or \ddagger\or
   \dagger\or \ddagger\or \|\or *\or **
   \or \ddagger\ddagger \else\@ctrerr\fi}}
\makeatother

\author{Taejoo Ahn\footnote{Department of Statistics, University of California, Berkeley. Email: taejoo\_ahn@berkeley.edu} \and Licong Lin\footnote{Department of Statistics, University of California, Berkeley. Email: liconglin@berkeley.edu} \and  Song Mei\thanks{Department of Statistics and EECS, University of California, Berkeley. Email: songmei@berkeley.edu}
}

        % <-this % stops a space
%\thanks{This paper was produced by the IEEE Publication Technology Group. They are in Piscataway, NJ.}% <-this % stops a space
%\thanks{Manuscript received April 19, 2021; revised August 16, 2021.}}



\maketitle

\begin{abstract}
In high dimensional variable selection problems, statisticians often seek to design multiple testing procedures that control the False Discovery Rate (\FDR), while concurrently identifying a greater number of relevant variables. Model-X methods, such as Knockoffs and conditional randomization tests, achieve the primary goal of finite-sample {\FDR} control, assuming a known distribution of covariates. However, whether these methods can also achieve the secondary goal of maximizing discoveries remains uncertain. In fact, designing procedures to discover more relevant variables with finite-sample {\FDR} control is a largely open question, even within the arguably simplest linear models.  

In this paper, we develop near-optimal multiple testing procedures for high dimensional Bayesian linear models with isotropic covariates. We introduce Model-X procedures that provably control the frequentist {\FDR} from finite samples, even when the model is misspecified, and conjecturally achieve near-optimal power when the data follow the Bayesian linear model. Our proposed procedure, {\PoEdCe}, incorporates three key ingredients: {\tt Po}sterior {\tt E}xpectation, {\tt d}istilled {\tt C}onditional randomization test (\dCRT), and the Benjamini-Hochberg procedure with {\tt e}-values (\eBH). The optimality conjecture of {\PoEdCe} is based on a heuristic calculation of its asymptotic true positive proportion (\TPP) and false discovery proportion (\FDP), which is supported by methods from statistical physics as well as extensive numerical simulations. Our result establishes the Bayesian linear model as a benchmark for comparing the power of various multiple testing procedures.

%Furthermore, when the prior is unknown, we show that an empirical Bayes variant of {\PoEdCe} still has finite-sample {\FDR} control and achieves near-optimal power. \sm{Modify} %\tj{font of dCRT, PoEdCe, eBH etcs are making lowercases to be ther upper case}   % \lc{want to change the font of FDR, POEdCe, etc. or not?}

% In this paper, we consider three notions of optimality in covariate selection problems, and derive the optimal Bayesian multiple testing procedures in linear models. We then propose a Model-X multiple testing procedure, PoEdCe, which provably controls the frequentist FDR from finite samples even under model misspecification, and conjecturally achieves near-optimal power when the data follows the Bayesian linear model with isotropic covariates and known prior. PoEdCe has three important ingredients: \textbf{Po}sterior \textbf{E}xpectation, \textbf{d}istilled \textbf{C}onditional randomization test (dCRT), and the Benjamini-Hochberg procedure with \textbf{e}-values (eBH). The optimality conjecture of PoEdCe is based on a heuristic calculation of its asymptotic true positive proportion (TPP) and false discovery proportion (FDP), which is supported by methods from statistical physics as well as extensive numerical simulations. Finally, when the prior is unknown, we numerically show that an empirical Bayes variant of PoEdCe still achieves near-optimal power with finite-sample FDR control. 

 \blfootnote{Code for our experiments is available at~\url{https://github.com/taejoo-ahn/FDR_Bayes_figures}.}
 
\end{abstract}

%\begin{IEEEkeywords}
%KEYWORDS
%\end{IEEEkeywords}

\section{Introduction}

High dimensional variable selection problems arise pervasively in a broad range of scientific domains, including genetics, healthcare, economics, and political science. These problems are often framed by statisticians within a multiple testing context, where the objectives are twofold: to control the false discovery rate (\FDR) and to identify as many relevant variables as possible. The first goal of {\FDR} control is usually much more crucial and expected to be achieved under much weaker model assumptions compared to the goal of variable discovery. To uncover more relevant variables, statisticians typically employ strong model assumptions that incorporate prior knowledge of the scientific domain. However, {\FDR} control, often having more significant consequences or risks in scientific applications, is desired even if the model and the prior are misspecified. The contrasting model assumptions needed for these two objectives pose a significant challenge in high-dimensional variable selection tasks.

Considerable previous work has focused on controlling the frequentist {\FDR} in variable selection problems. Among these, Model-X methods such as Knockoffs and conditional randomization tests \cite{barber2015controlling,candes2018panning,liu2022fast}, have proven successful in controlling {\FDR} from finite samples under mild model assumptions. These methods, more specifically, presume a known covariate distribution but allow any correlation between the response and the covariates, employing resampling techniques to convert any base statistics into test statistics that control finite-sample {\FDR}. Recent work has also studied the power of Model-X procedures when combined with LASSO-based statistics \cite{weinstein2017power,weinstein2020power,liu2019power,wang2022high}. However, how to leverage Model-X methods to establish procedures with \textit{optimal power} in specific models remains largely unanswered.

An alternative line of research has sought to derive optimal {\FDR} control procedures using the Bayesian approach. The central quantity in these Bayesian methods is the local false discovery rate (local fdr) \cite{efron2001empirical,efron2005local}, which is the posterior probability of the null hypothesis being applicable. Prior work \cite{muller2004optimal,muller2006fdr, sun2007oracle, xie2011optimal} has demonstrated that truncating local fdrs results in the most powerful procedure among those controlling Bayesian {\FDR}. However, the Bayesian {\FDR} control of these methods relies heavily on the assumption that the model and the prior are correctly specified. These procedures could potentially lose {\FDR} control whenever the prior is incorrect or the model is misspecified.

% Another line of research in Bayesian school address the problem of optimality. It is shown that under Bayesian multiple hypothesis problem setup, truncating the local fdr, which is the posterior probability of j'th hypothesis being null, gives the most powerful testing procedures under given FDR control (e.g., \cite{muller2004optimal,muller2006fdr, sun2007oracle, xie2011optimal}). However, such a procedure strongly relies on its model assumptions, and might lose control of FDR level when our prior information is wrong or the model is misspecified. 
 
A natural idea for circumventing the restrictions of both frequentist and Bayesian methods involves developing a procedure that integrates these two approaches. More specifically, one might use local fdrs as base statistics and wrap them using Model-X methodologies. Moreover, the Bayesian linear model, which is applicable to a broad range of scientific problems, presents itself as arguably the most suitable model for exploring such an idea. This prompts us to pose the following question:
\begin{quote}
 \textit{Is there a procedure that controls {\FDR} from finite samples and achieves near-optimal power under well-specified Bayesian linear models?  }
\end{quote}

% Given these frequentist and Bayesian literature, combining these two approaches seems to be a natural idea of simultaneously achieving frequentist {\FDR} control and optimizing for the Bayesian power. That is, using local fdrs as base statistics and wrapping them by Model-X procedures. The argurably most natural and simplest model for investigating such an idea is the Bayesian linear model, which is also widely applicable in various scientific problems. This motivates us to ask the following question:
% \begin{quote}
%  \textit{Is there a procedure that wraps local fdrs by a Model-X procedure, which controls {\FDR} from finite samples, and achieves near Bayes-optimal power under Bayesian linear models? }
% \end{quote}
% Prior works \cite{weinstein2017power,weinstein2020power} have shown that Knockoffs will lose power comparing to oracle procedures: wrapping the LASSO statistics (or LASSO coefficient-difference statistics) by the Knockoff procedure gives less asymptotic power. This rule out the vanilla Knockoff procedure 

In this paper, we investigate this question and suggest an affirmative answer through the introduction of two procedures, {\PoPCe} (pronounced as ``pop-see'') and {\PoEdCe} (pronounced as ``pod-see''). These procedures utilize local fdr and posterior expectation as the base statistics, apply the conditional randomization test (\CRT) \cite{candes2018panning} to compute the p-value of each hypothesis, and then implement the {\eBH} procedure \cite{wang2022false} on the obtained p-values. We show that {\PoPCe} and {\PoEdCe} always control finite-sample {\FDR} and conjecturally achieve near-optimal power under Bayesian linear models with a known prior distribution. The conjectured result is supported by methods from statistical physics as well as extensive numerical simulations. This result establishes the Bayesian linear model as a benchmark for comparing the power of various multiple testing procedures.

% In this paper, we give an affirmative answer to this question. We propose the {\PoPCe} and {\PoEdCe} procedures which asymptotically obtain maximal number of true discoveries when the model is correctly specified while controlling the finite-sample frequentist FDR even under model misspecification under Model-X settings. We focus on Bayesian linear model with i.i.d. Gaussian design. Our testing procedure uses local fdr and posterior mean as a base statistics, and apply distilled conditional randomization test to calculate p-values for each hypothesis, and then apply eBH procedure on obtained p-values. \sm{I will rewrite these three paragraphs. }
 
% In spite of an extensive literature on FDR-controlling methods for multiple testing problems, the question regarding the optimality of a FDR-controlling  method is not well understood. For instance, under the Model-X setting, the Knockoff method \cite{candes2018panning} is known to control the FDR in general. However, it is not known that whether the Knockoff filter with a specific test statistics can achieve the best performance, meaning that it reject most non-null hypotheses while controlling the FDR under a prespecified level.  In this work, we calculate the statistical limit of multiple testing methods and propose a new method which controls FDR under the Model-X setting and has  optimal power under Bayesian linear models.

\subsection{Model setup}\label{sec:model-setup}

%While we would like to discover as many true variables as possible, we do not want to include too many variables that does not contribute to $Y$. 

%Namely, we consider in total $d$ null hypotheses, where the $j$-th null hypothesis gives 
%\[
%H_{0j}: Y \perp X_j \vert \bX_{-j}, ~~~~~~~~ j = 1, 2, \ldots, d. 
%\]
%We denote $\mathcal{H}_0 \subseteq [d]$ to be the set of null hypotheses and $S=[d]/\mathcal{H}_0$ to be the set of non-null hypotheses. 

Suppose we observe independent and identically distributed (i.i.d.) samples $\{ (\bx_i, y_i) \}_{ i \in [n]} \subseteq \R^d \times \cY$ from the joint distribution $P \in \ProbFam(\R^d \times \cY)$ over variables $\bX = (X_1, X_2, \ldots, X_d)$ and $Y$. Throughout the paper, we represent the response vector as $\Yb = (y_1, \ldots, y_n)^\sT \in \R^n$ and the covariate matrix as $\Xb = (\bx_1, \ldots, \bx_n)^\sT = (\xb_1, \ldots, \xb_d) \in \R^{n \times d}$, where $\{ \bx_i\}_{i \in [n]} \subseteq \R^d$ and $\{ \xb_j \}_{j \in [d]} \subseteq \R^n$. In the case when the response $Y$ only depends on a small subset of the covariates, our goal is to identify this subset from samples. This problem can be cast into the multiple testing framework. Namely, a variable $X_j$ is said to be ``null'' if and only if under the joint distribution $P \in \ProbFam(\R^d \times \cY)$, $Y$ is independent of $X_j$ when conditioned on the other variables $\bX_{-j} = \{ X_1, \ldots, X_d \} \setminus \{ X_j \}$. The subset of null variables is denoted by $\cH_0 = \cH_0(P) \subseteq [d]$.  Conversely, a variable $X_j$ is called "nonnull" if $j \not\in \cH_0$, with the subset of nonnull variables denoted by $S = S(P) = [d] \setminus \cH_0(P)$.

A test statistics $\bT = ( T_j )_{j \in [d]}$ is a vector of possibly random functions $T_j: \Omega \to \{ 0, 1 \}$ where $\Omega = (\R^{d} \times \cY)^n$. We say that the $j$-th null hypothesis is rejected if and only if $T_j(\cD) = 1$. Given the dataset $\cD = (\Xb, \Yb)$ and the joint distribution $P \in \ProbFam(\R^d \times \cY)$, the number of discoveries $\Rej$, the number of false discoveries $\FD$, and the number of true discoveries $\TD$ of the test statistics $\bT$ are respectively defined as:
\begin{equation}\label{eqn:FD_TD}
\begin{aligned}
\Rej(\bT; \cD) \equiv&~ \#\{ j \in [d]: T_j(\cD) = 1 \}, \\
\FD(\bT; \cD, P) \equiv&~ \#\{ j \in \cH_0(P) : T_j(\cD) = 1 \}, \\
\TD(\bT; \cD, P) \equiv&~ \#\{ j \in S(P): T_j(\cD) = 1 \}.
\end{aligned}
\end{equation}
Furthermore, we define the \textit{false discovery proportion} (\FDP) and the \textit{true positve proportion} (\TPP) of the test statistics $\bT$ as
\begin{align}\label{eqn:FDP_TPP}
\FDP(\bT; \cD, P) \equiv \frac{ \FD(\bT; \cD, P) }{ \Rej(\bT; \cD) \vee 1},~~~~~~~~ \TPP(\bT; \cD, P) \equiv \frac{ \TD(\bT; \cD, P) }{ |S(P) |\vee 1 }. 
\end{align}
In the above expressions, we explicitly detail the dependence on  $\cD$ and $P$ for clarity. We will abbreviate as $\FDP(\bT) = \FDP(\bT; \cD, P)$ when the dataset $\cD$ and the distribution $P$ are unambiguous from the context (similar abbreviations apply to $\TPP$, $\Rej$, $\FD$, and $\TD$).  %We will write in short $\FDP(\bT) = \FDP(\bT; \cD, P)$ whenever the dataset $\cD$ and the distribution $P$ do not have ambiguity from the context (and we write in short similarly for $\TPP$, $\Rej$, $\FD$, and $\TD$). 

We next define the usual frequentist {\FDR} and {\TPR} \cite{benjamini1995controlling}, which is the expectation of {\FDP} and {\TPP} over the observed samples following the distribution $\{ (\bx_i, y_i) \}_{i \in [n]}\sim_{i.i.d.} P$, and over the randomness of the test statistics $\bT$, 
\begin{align}\label{eqn:frequentist-fdr}
\FDR(\bT, P) \equiv~  \E_{\cD \sim P, \bT} \Big[ \frac{\FD(\bT; \cD, P)}{\Rej(\bT; \cD) \vee 1} \Big], ~~~~~~ \TPR(\bT, P) \equiv~ \E_{\cD \sim P, \bT}\Big[ \frac{ \TD(\bT; \cD, P) }{ |S(P) |\vee 1 } \Big].
\end{align}
A test statistics $\bT$ is said to control the frequentist {\FDR} at level $\alpha$ over a collection of distributions $\cM$, if $\FDR(\bT; P) \le \alpha$ for any distribution $P \in \cM$.

Additionally, statisticians often consider Bayesian variants of {\FDR} and {\TPR}. Given a family of distributions $\{ P_{\bZ \vert \bbeta_0} \}_{\bbeta_0 \in \cB} \subseteq \ProbFam(\R^d \times \cY)$ parameterized by $\bbeta_0 \in \cB$, and given a prior $\bbeta_0 \sim \Pi \in \ProbFam(\cB)$, we can define the Bayesian false discovery rate (\BFDR) and the Bayesian true positive rate (\BTPR) \cite{muller2004optimal, muller2006fdr, sun2007oracle, storey2007optimal} of the test statistics $\bT$ as
\begin{equation}\label{eqn:def_BFDR}
\begin{aligned}
\BFDR(\bT, \Pi) \equiv&~ \E_{\bbeta_0 \sim \Pi} \Big\{ \E_{\cD \sim P_{\bZ \vert \bbeta_0}, \bT} \Big[ \frac{\FD(\bT; \cD, P_{\bZ \vert \bbeta_0})}{\Rej(\bT; \cD) \vee 1} \Big] \Big\},\\ \BTPR(\bT, \Pi) \equiv&~ \E_{\bbeta_0 \sim \Pi} \Big\{ \E_{\cD \sim P_{\bZ \vert \bbeta_0}, \bT} \Big[ \frac{\TD(\bT; \cD, P_{\bZ \vert \bbeta_0})}{\vert S(P_{\bZ \vert \bbeta_0}) \vert \vee 1} \Big] \Big\}.
\end{aligned}
\end{equation}
In the literature, people have also considered the \textit{marginal false discovery rate} (\mFDR) and \textit{marginal true positive rate} (\mTPR) defined as (with the convention that $0 / 0 = 0$)
\begin{equation}\label{eqn:def_mFDR_mTPR}
\begin{aligned}
\mFDR(\bT, \Pi) \equiv&~ \frac{\E_{\bbeta_0 \sim \Pi}\{\E_{\cD \sim P_{\bZ \vert \bbeta_0}, \bT}[\FD(\bT; \cD, P_{\bZ \vert \bbeta_0})]\}}{\E_{\bbeta_0 \sim \Pi}\{\E_{\cD \sim P_{\bZ \vert \bbeta_0}, \bT}[\Rej(\bT; \cD)]\}}, \\ 
\mTPR(\bT, \Pi) \equiv&~ \frac{\E_{\bbeta_0 \sim \Pi}\{\E_{\cD \sim P_{\bZ \vert \bbeta_0}, \bT}[\TP(\bT; \cD, P_{\bZ \vert \bbeta_0})]\}}{\E_{\bbeta_0 \sim \Pi}\{\vert S(P_{\bZ \vert \bbeta_0}) \vert \}}. 
\end{aligned}
\end{equation}
In high dimensional statistical models, due to the concentration phenomenon, different average notions of false discovery proportion (and true positive proportion) often approximate each other asymptotically. In simpler terms, in high dimensions, we often observe $\FDP \approx \FDR \approx \BFDR \approx \mFDR$ and $\TPP \approx \TPR \approx \BTPR \approx \mTPR$ \cite{genovese2002operating}. %Consequently, any of $\TPP, \TPR, \BTPR, \mTPR$ can serve as the measure of power. %In this paper, we will use {\mTPR} as the measure of power for theoretical simplicity. 

This paper will focus on the Bayesian linear model with isotropic Gaussian covariates. The same model was previously investigated in \cite{weinstein2017power,weinstein2020power} to analyze the efficacy of the Knockoff approach when applied to LASSO-based statistics. In the Bayesian linear model, the response $Y \in \R$ and the covariates $\bX \in \R^d$ form a linear relationship $Y = \< \bX, \bbeta_0\> + \eps$, where  the covariates $\bX$, the parameters $\bbeta_0 = (\beta_{0, j})_{j \in [d]} \in \R^d$, and the random noise $\eps \in \R$ are mutually independent. We assume that $\bX \sim \cN(0, (1/n) \id_d)$, and assume that the prior of $\bbeta_0$ has i.i.d. coordinates $(\beta_{0, j})_{j \in [d]} \sim_{i.i.d.} \Pi$, where $\Pi \in \ProbFam(\R)$ is a distribution supported on the real line. We further assume that $\Pi$ has a point mass at $0$, and the weight of this point mass gives $\pi_0 \in (0, 1)$, which is roughly the proportion of parameters $\{ \beta_{0,j} \}_{j \in [d]}$ that equal $0$. It should be noted that in linear models, under certain mild conditions (such as one cannot perfectly predict any $X_j$ from knowledge of $\bX{-j}$, \cite[Proposition 2.2]{candes2018panning}), the $j$-th null hypothesis $H_{0j}: Y \perp X_j \vert \bX_{- j},\bbeta_0$ can be demonstrated to be equivalent to the hypothesis $H_{0j}: \beta_{0, j} = 0$. A more comprehensive description of the Bayesian linear model is provided in Assumption \ref{ass:Bayesian_linear_model}. 
%Note that in this linear model, if we assume the covariance matrix of the marginal distribution of the covariates $P_\bX$ is nonsingular (which means that the feature variable $\bX$ is not autocorrelated \sm{Check}), the $j$-th null hypothesis $H_{0j}: Y \perp X_j \vert \bX_{- j},\bbeta$ can be shown to be equivalent to the hypothesis $H_{0j}: \beta_j = 0$. 

\subsection{Frequentist optimality in the Model-X setting and Bayesian optimality}\label{sec:freq}

%Here $\ProbFam$ will be the set of possible distributions of $(\bX,\by)$, and for any $\Pi\in\ProbFam(\cB)$, it contains $\{ P_{\bZ \vert \bbeta}: \bbeta \in {\rm supp}\{ \Pi \} \}$, the set of parametrized distribution $P_{\bZ \vert \bbeta}$.

We now focus on the Model-X setting \cite{candes2018panning, liu2022fast}, in which $P \vert_\bX$, the marginal distribution of the covariates of $P \in \ProbFam(\R^d \times \cY)$, is known to the statistician. For a distribution over covariates $P_\bX \in \ProbFam(\R^d)$, we let 
\begin{equation}\label{eqn:M-P-X}
\cM(P_\bX) \equiv \{ P \in \ProbFam(\R^d \times \cY): P \vert_\bX = P_\bX\}
\end{equation} 
represent the ensemble of probability distributions where the marginal distribution of the covariates aligns with $P_\bX$. 
Notice that our goal is to propose test statistics that control the frequentist {\FDR} from finite samples and ensure asymptotic optimality given the model is correctly specified. This gives rise to the following optimality notion. 
%In this setting, the test can take $P_x$ as the input, and hence we consider the set of test statistics
%\[
%\cT = \{ \bT : \Omega^n \times \ProbFam(\cX) \to \{ 0, 1\}^d \}. 
%\]
%For any $\bT \in \cT$ and $P_x \in \ProbFam(\cX)$, we denote $\bT \circ P_x$ as a mapping from $\Omega^n$ to $\{0, 1\}^d$ where for any dataset $\cD \in \Omega^n$, we have $\bT \circ P_x (\cD) = \bT(\cD, P_x)$. 

\begin{definition}[Optimal test with finite-sample {\FDR} control]\label{def:optimal-frequentist-fdr}
Let $\{ P_{\bZ \vert \bbeta_0} \}_{\bbeta_0 \in \cB} \subseteq \ProbFam(\R^d \times \cY)$ be a family of distributions and let $\Pi \subseteq \ProbFam(\cB)$ be a prior. Assume that $\{ P_{\bZ \vert \bbeta_0} \}_{\bbeta_0 \in \cB} \subseteq \cM(P_\bX)$ for a distribution $P_\bX \in \cP(\R^d)$. For some scalars $\alpha \in (0, 1)$ and $\eps > 0$, we say a test statistics $\bT_\star: \Omega \to \{ 0, 1\}^d$ is $(\alpha, \cM(P_\bX), \Pi, \eps)$-optimal with frequentist {\FDR} control, if its frequentist {\FDR} is controlled at level $\alpha$ over the collection of distributions $\cM(P_\bX)$, 
\begin{align}\label{eq:fdr_control_level_alpha}
\sup_{P \in \cM(P_\bX)} \FDR(\bT_\star, P) \le \alpha,
\end{align}
and its {\mTPR} with prior $\Pi$ is nearly maximized across all tests that control frequentist {\FDR} at level $\alpha$, 
\begin{align}
\mTPR(\bT_\star , \Pi) \ge \sup_{\bT: \Omega \to \{0, 1\}^d} \Big\{ \mTPR(\bT , \Pi): \sup_{P \in \cM(P_\bX)} \FDR(\bT , P) \le \alpha \Big\} - \eps.
\end{align}
\end{definition}

%\begin{definition}[Set $\BTPE$-{\FDR} optimality]
%Let $\mcP$ be a set of probability distribution on $\Omega = \cX \times \cY$. Let $\mcPi$ be a set of prior distribution on $\cB$, and let $\eps: \mcPi \to \R_{\ge 0}$. For a procedure $\bT_\star \in \cT$, we call it $(\alpha, \mcP, \mcPi, \eps)$-$\BTPE$-{\FDR} optimal if 
%\[
%\sup_{P \in \mcP} \FDR(\bT_\star \circ P_x, P) \le \alpha
%\]
%and for any $\Pi \in \mcPi$, we have
%\[
%\BTPE(\bT_\star \circ P_x, \Pi) \ge \max_{\bT \in \cT}\Big\{ \BTPE(\bT \circ P_x, \Pi): \sup_{P \in \mcP} \FDR(\bT \circ P_x, P) \le \alpha \Big\} - d \cdot \eps(\Pi).
%\]
%\end{definition}

%\subsection{Relation between FDR optimality and BFDR optimality.}

%The following two definitions are two different optimalities in the Bayesian sense. 
%\begin{definition}[{\mFDR} optimality]
%Assume Assumption \ref{bayesian_data_dist} and let $\Pi \subseteq \ProbFam(\cB)$ be the prior distribution in it. For a procedure $\bT_\star: \Omega \to \{ 0, 1\}^d$, we call it is $(\alpha, \Pi, \eps)$-{\mFDR} optimal if 
%\[
%\mFDR(\bT_\star, \Pi) \le \alpha
%\]
%and 
%\[
%\mTPR(\bT_\star, \Pi) \ge \max_{\bT: \Omega \to \{0, 1\}^d } \Big\{ \mTPR(\bT, \Pi): \mFDR(\bT, \Pi) \le \alpha \Big\} -  \eps.
%\]
%We say a procedure is $(\alpha, \Pi)$-{\mFDR} optimal if it is $(\alpha, \Pi, 0)$-{\mFDR} optimal. 
%\end{definition}
%

In Definition \ref{def:optimal-frequentist-fdr}, an alternative to employing {\mTPR} as the power measure is to use {\BTPR} instead. As we have mentioned, {\BTPR} and {\mTPR} will be asymptotically equal in many high dimensional models, so it will not make a big difference in choosing either as the power measure. For the purposes of this paper, we specifically choose {\mTPR} as the power measure due to its ability to facilitate a more straightforward theoretical framework. 

In addition to the aforementioned, there are two other crucial optimality concepts, namely, the optimal test with {\BFDR} control and the optimal test with {\mFDR} control. 
\begin{definition}[Optimal test with {\BFDR} control]\label{def:BFDR-optimality}
Let $\{ P_{\bZ \vert \bbeta_0} \}_{\bbeta_0 \in \cB} \subseteq \ProbFam(\R^d \times \cY)$ be a family of distributions and let $\Pi \subseteq \ProbFam(\cB)$ be a prior on $\bbeta_0$. Let $\alpha, \eps \in [0, 1]$ be two scalars. For a procedure $\bT_\star: \Omega \to \{ 0, 1\}^d$, we call it is $(\alpha, \Pi, \eps)$-optimal with {\BFDR} control, if its {\BFDR} with prior $\Pi$ is controlled at level $\alpha$, 
\begin{align}\label{eq:bfdr_control_level_alpha}
\BFDR(\bT_\star, \Pi) \le \alpha,
\end{align}
and its {\mTPR} is nearly maximized across all tests that control {\BFDR} at level $\alpha$, 
\begin{align}
\mTPR(\bT_\star, \Pi) \ge \max_{\bT: \Omega \to \{0, 1\}^d }\Big\{ \mTPR(\bT, \Pi): \BFDR(\bT, \Pi) \le \alpha \Big\} - \eps.
\end{align}
%We say a procedure is $(\alpha, \Pi)$-{\BFDR} optimal if it is $(\alpha, \Pi, 0)$-{\BFDR} optimal. 
\end{definition}

\begin{definition}[Optimal test with {\mFDR} control]\label{def:mfdr_optimality}
Let $\{ P_{\bZ \vert \bbeta_0} \}_{\bbeta_0 \in \cB} \subseteq \ProbFam(\R^d \times \cY)$ be a family of distributions and let $\Pi \subseteq \ProbFam(\cB)$ be a prior on $\bbeta_0$. Let $\alpha, \eps \in [0, 1]$ be two scalars. For a procedure $\bT_\star: \Omega \to \{ 0, 1\}^d$, we say that it is $(\alpha, \Pi, \eps)$-optimal with {\mFDR} control, if its {\mFDR} with prior $\Pi$ is controlled at level $\alpha$,
\begin{align}
\mFDR(\bT_\star, \Pi) \le \alpha,
\end{align}
and its {\mTPR} is nearly maximized across all tests that control {\mFDR} at level $\alpha$, 
\begin{align}
\mTPR(\bT_\star, \Pi) \ge \max_{\bT: \Omega \to \{0, 1\}^d }\Big\{ \mTPR(\bT, \Pi): \mFDR(\bT, \Pi) \le \alpha \Big\} - \eps.
\end{align}
%We say a procedure is $(\alpha, \Pi)$-{\BFDR} optimal if it is $(\alpha, \Pi, 0)$-{\BFDR} optimal. 
\end{definition}

% It is worth pointing out that Definition \ref{def:BFDR-optimality} and \ref{def:mfdr_optimality} are mostly parallel, and they both used {\mTPR} as the measure of power. As we have mentioned, again, {\BTPR} and {\mTPR} will be asymptotically equal in many high dimensional models, so it will not make a big difference choosing either as the measure of power. We didn't use {\BTPR} as the measure of power in Definition \ref{def:BFDR-optimality}, since our theoretical framework will be simpler using {\mTPR}. \lc{why don't we delete this paragraph? it's almost the same as the previous one.} \tj{agreed: looks too redundant}

%In this work, we propose a new testing procedure which is asymptotically {\FDR} optimal. The {\BFDR} optimality will help to illustrate our proposed procedure is {\FDR} optimal. 

An essential connection exists between the optimal test with frequentist {\FDR} control and the optimal test with {\BFDR} control, as per Definition \ref{def:optimal-frequentist-fdr} and \ref{def:BFDR-optimality}. The following lemma shows that the power of the optimal test with {\BFDR} control gives an upper bound for the power of the optimal test with frequentist {\FDR} control. The proof of this is directly evident from Definition \ref{def:optimal-frequentist-fdr} and \ref{def:BFDR-optimality}. 
\begin{lemma}\label{lem:bound}
Let $\{ P_{\bZ \vert \bbeta_0} \}_{\bbeta_0 \in \cB} \subseteq \ProbFam(\R^d \times \cY)$ be a family of distributions and let $\Pi \subseteq \ProbFam(\cB)$ be a prior. Assume that $\{ P_{\bZ \vert \bbeta_0} \}_{\bbeta_0 \in \cB} \subseteq \cM(P_\bX)$ for a distribution $P_\bX \in \cP(\R^d)$. Then the {\mTPR} value of the $(\alpha, \cM(P_\bX), \Pi, \eps)$-optimal test with frequentist {\FDR} control is always less than or equal to the {\mTPR} value of the $(\alpha, \Pi, \eps)$-optimal test with {\BFDR} control. 
\end{lemma}
\begin{proof}
Note that $\BFDR(\bT, \Pi)= \E_{\bbeta_0 \sim \Pi} \FDR(\bT, P_{\bZ \vert \bbeta_0})\leq  \E_{\bbeta_0 \sim \Pi} [\alpha]=\alpha$ for any test statistics $\bT$ with frequentist {\FDR} below $\alpha$ (c.f. Eq. \eqref{eq:fdr_control_level_alpha}). That is, any test statistics  with  frequentist {\FDR}  below $\alpha$ (c.f. Eq. \eqref{eq:fdr_control_level_alpha}) also has {\BFDR} below $\alpha$ (c.f. Eq. \eqref{eq:bfdr_control_level_alpha}). Lemma~\ref{lem:bound} follows immediately. 
\end{proof}

Lemma \ref{lem:bound} suggests that, in order to prove that a test statistics $\bT$ is near-optimal with frequentist {\FDR} control, it is sufficient to demonstrate two things: (1) $\bT$ controls finite-sample {\FDR}; (2) the {\mTPR} of $\bT$ is close to the {\mTPR} of the optimal test with {\BFDR} control. In this paper, we will follow this approach to show that {\PoPCe} and {\PoEdCe} (which will be described in Section \ref{sec:FO}) are near-optimal with finite-sample frequentist {\FDR} control.

%We will first propose a testing procedure $\CPoP$ that is the optimal test with {\BFDR} control. Then we propose two other testing procedures  {\PoPCe} and {\PoEdCe}, which we claim to be asymptotically {\FDR} optimal. We first prove that those procedures control {\FDR}, and then prove that {\mTPR} of both procedures asymptotically reaches {\mTPR} of $\CPoP$ which is {\BFDR} optimal. Making use of Lemma 1, a testing procedure that controls {\FDR} and at the same time reaches the power of {\BFDR} optimal procedure is {\FDR} optimal, and thus we claim both  {\PoPCe} and {\PoEdCe} are asymptotically {\FDR} optimal.

%\subsection{Bayesian generalized linear model}
% 
%In order to get a powerful testing procedure, statisticians often assume a parametric model for the joint distribution $P_{\bX, Y}$. In this paper, we consider the generalized linear model: for some fixed $\bbeta$, we assume $P_{\bX, Y}(\bx, y) = P_{\bX, Y \vert \bbeta}(\bx, y \vert \bbeta) = P_\bX(\bx) P_{Y \vert \bX, \bbeta}(y \vert \bx, \bbeta )$, $P_{Y \vert \bX, \bbeta}(y \vert \bx, \bbeta) = P(y \vert \< \bbeta, \bx\>)$ where $P(y \vert t)$ is the conditional distribution of $y$ given score $t$. We assume a prior for $\bbeta$ in which $(\beta_j)_{j \in [d]} \sim_{i.i.d.} \Pi$ where $\Pi = (1 - \pi) \delta_0 + \pi \Pi_\star$. Here $\pi \in (0, 1)$ and $\Pi_\star \subseteq \ProbFam(\R)$ does not contain delta mass at $0$. In the generalized linear model, the $j$-th null hypothesis (with some technical assumptions on $P_\bX$ and $P(y \vert t)$) can be shown to be equal to the hypothesis $\beta_j = 0$. 

\subsection{Summary of contributions and paper outline}

\begin{itemize}
\setlength\itemsep{0.4em}
\item {\bf Optimal procedures with {\mFDR} and {\BFDR} control.} In Section \ref{sec:mfdr}, we introduce two procedures: {\TPoP} ({\tt T}runcating the {\tt Po}sterior {\tt P}robability) and {\CPoP} ({\tt C}umulative {\tt Po}sterior {\tt P}robability). Both methods are based on the truncation of local fdrs, the posterior probabilities of the hypotheses being null. We show that {\TPoP} is the optimal test with {\mFDR} control, while {\CPoP} is the optimal test with {\BFDR} control. It should be noted that neither {\TPoP} nor {\CPoP} are entirely new procedures \cite{muller2004optimal, sun2007oracle,xie2011optimal}, and their optimality proofs are primarily based on Bayesian decision theory.

\item {{\bf  Asymptotic power of the Bayesian optimal procedures.}} %{\TPoP} {\bf and} {\CPoP}{\bf.} } 
In Section \ref{sec:asymptotic_FDP_TPP}, we examine the Bayesian linear model with isotropic Gaussian covariates and obtain the analytical formula for the asymptotic {\TPP} and {\FDP} associated with the optimal procedures, {\CPoP} and {\TPoP}. The derivation of this formula is primarily heuristic, leveraging the replica method \cite{mezard2009information}, a useful tool originating from spin-glass theory within statistical physics. The validity of the derived analytical formula is subsequently demonstrated through numerical simulations. 

% \sm{1. We provide procedure that controls frequentist fdr in the Model-X setting, and achieves optimal power in the Bayesian setup... 2. The procedures are given by PoPCe, based on local fdr, CRT, dCRT, eBH. We also provide computational efficient variant PoEdCe. 3. The conjecture is based on heuristc calculations.  We perform numerical simulations to verify our conjecture. 
% }

% in Section \ref{sec:FDR_optimality}

\item{\bf Optimal procedures with frequentist {\FDR} control.}
In Section \ref{sec:FO}, we establish  procedures that control frequentist {\FDR} in the Model-X setting and are conjecturally near-optimal in the Bayesian linear model, suggesting an affirmative answer to our question. Specifically, we introduce the {\PoPCe} procedure ({\tt Po}sterior {\tt P}robability + {\tt C}onditional randomization test + {\eBH}) along with its computationally efficient variant, {\PoEdCe} ({\tt Po}sterior {\tt E}xpectation + {\tt d}istilled {\tt C}onditional randomization test + {\eBH}). We prove that both of these procedures control the frequentist FDR from finite samples under any data-generating model $P \in \cM(P_\bX)$, even in instances of model misspecification. When the data originates from a Bayesian linear model with isotropic covariates,  we propose the conjecture that {\PoPCe} and {\PoEdCe} achieve near-optimal power. We arrive at this conjecture through heuristic calculations, which we subsequently confirm via numerical simulations.

\item {\bf Bayesian linear model as a benchmark.} This result establishes the Bayesian linear model as a benchmark for power comparison amongst various multiple testing procedures. In other words, the efficacy of these multiple testing procedures can be evaluated in relation to the optimal power of {\PoEdCe} within the Bayesian linear model. 

\end{itemize}

\paragraph{Practical implications} We would like to emphasize that while the {\PoPCe} and {\PoEdCe} procedures serve as theoretical tools for the power analysis of {\FDR}-controlling procedures, we do not recommend their direct application in practical scenarios. Indeed, these procedures could possibly be powerless under model misspecification in real-world datasets. An interesting open question is thus whether one can enhance the power under model misspecification while maintaining finite-sample validity and Bayes optimality.

\subsection{Notations and conventions}

%We typically denote scalar-valued variables by standard font, e.g., $x,y,s,u$, and vector or matrix-valued variables by bold font, e.g., $\Xb,\Yb,\bbeta,\xb$.

Through the paper, for an integer $n$, we denote $[n] = \{ 1, 2, \ldots, n\}$. We denote the samples by $\{ (\bx_i, y_i) \}_{ i \in [n]} \subseteq \cX \times \cY$, the response vector by $\Yb = (y_1, \ldots, y_n)^\sT \in \R^n$ and the covariate matrix by $\Xb \in\R^{n\times d}$. We also denote the rows of $\Xb$ by $\{ \bx_i\}_{i\in[n]}$ and the columns by $\{\xb_j\}_{j\in[d]}$. We use $\bbeta_{-j} \in \R^{d-1}$ to denote the sub-vector of $\bbeta \in \R^d$ with the $j$-th coordinate removed, use $\bx_{i, -j} \in \R^{d-1}$ to denote the sub-vector of $\bx_i \in \R^d$ with the $j$-th coordinate removed, and use $\Xb_{-j} \in \R^{n \times (d-1)}$ to denote the sub-matrix of $\Xb$ with the $j$-th column removed.

For a measurable space $S$, we let $\ProbFam(S)$ denote the set of all Borel probability measures on the space. For a distribution $P\in\ProbFam{(\cX\times\cY)}$, we use $(\bX,Y) \in \R^d \times \R$ to denote the random variables that follow the distribution $P$, and use $\bX_{-j} \in \R^{d-1}$ to denote sub-vector of $\bX$ with the $j$-th coordinate removed. We use $\cL(\bX)$ to denote a distribution over the random vector $\bX$, and use $\cL(X_j | \bX_{-j} = \bx_{-j})$ to denote the conditional distribution of $X_j$ given $\bX_{-j} = \bx_{-j}$ under the law $\cL(\bX)$. 

In this paper, a mathematical statement is termed as a \textit{formalism} if it can be derived heuristically and confirmed numerically. In the appendix, we present several \textit{formalisms} which support the intuitions behind our conjectures.

% \sm{Add notations. }

% $\Xb$ and $\bX$. $\cP$. 

\section{Other related works}

% There is a large body of related work, as multiple testing is one of the most exciting topics in statistical inference. Here we only  review the most closely related work due to space limitations.

%They considered a mixture Bayesian model with a latent Bernoulli parameter, and define the local fdr as the posterior probability such that the latent parameter equals zero. 

Multiple testing is a central topic of statistical inference and has inspired numerous studies. However, due to space constraints, we will focus on the most relevant works. 

The concept of the false discovery rate (\FDR) was introduced in the frequentist context by \cite{benjamini1995controlling}, who also proposed a procedure (the Benjamini-Hochberg procedure, hereafter {\BH}) to control the {\FDR} given independent p-values. The original {\BH} procedure is conservative by a factor of $\pi_0$, where $\pi_0$ is the proportion of null hypotheses. To address this, the same authors \cite{benjamini2000adaptive} recommended estimating $\pi_0$ using large p-values. Later, \cite{benjamini2001control} demonstrated that the {\BH} procedure also controls the {\FDR} when the p-values satisfy the PRDS condition, which is less stringent than independence. However, the {\BH} procedure fails to control the {\FDR} with arbitrarily correlated p-values unless adjusted by a log factor, leading to a loss of power. As a solution, \cite{wang2022false} introduced the {\eBH} procedure (Benjamini-Hochberg with e-values), which controls the {\FDR} even in the presence of arbitrary correlations among e-values \cite{vovk2021values,shafer2019language}, a more manageable mathematical tool than p-values.

There is an alternate line of research that studies optimal procedures concerning Bayes FDR control. \cite{efron2001empirical, efron2005local} introduced the local fdr framework, which applies the empirical Bayes approach to multiple testing problems. Here, the local fdr represents the posterior probability that the null hypothesis holds. \cite{sun2007oracle, xie2011optimal} demonstrated that truncating the local fdr is the optimal procedure for controlling the marginal {\FDR} (\mFDR) within a Bayes setting, and they suggested truncating the cumulative local fdr for adaptivity. \cite{muller2004optimal,muller2006fdr,storey2007optimal} proved that truncating the local fdr in certain adaptive manners is optimal for specific Bayesian criteria. Conversely, other studies, such as \cite{abramovich2006adapting,ma2020global,zhang2020estimation}, derived {\FDR} control procedures in the super-sparse regime. These works derived the ``optimal'' procedures in the regime such that the optimal power can asymptotically approach one, a different context than what we consider in this paper.

Several {\FDR}-controlling methods based on the Knockoff filter and its variants have been recently introduced. \cite{barber2015controlling} presented the fixed design Knockoff procedure, which controls the finite-sample {\FDR} in linear models. However, this procedure is applicable only when the sample size exceeds the dimension. \cite{candes2018panning} introduced the Model-X Knockoff procedure, demonstrating remarkable flexibility by controlling the finite-sample FDR in any probabilistic model, as long as the distribution of covariates $\bX$ is known. %\cite{gimenez2019improving} considers the multiple-Knockoff procedure and shows that multiple-Knockoff has improved stability and is sometimes more powerful. 
Inspired by the Knockoff procedure, \cite{xing2021controlling, ke2020power, dai2022false, dai2020scale} proposed to control {\FDR} in an asymptotic sense using mirror statistics. They reveal that these methods attain greater power when features are highly correlated and when the parameter vector is less sparse. Furthermore,~\cite{spector2022powerful} enhanced the power of Model-X Knockoff procedures by creating knockoffs that minimize the reconstructability of the features.

Under the assumption of a linear model and high-dimensional proportional asymptotics, a series of works \cite{su2017false,wang2020which,weinstein2017power,weinstein2020power,wang2020complete,liu2019power,hu2019asymptotics,bu2021characterizing,wang2022high} derive the precise limit of the {\FDP} and {\TPP} for various variable selection methods, such as LASSO, $\ell_p$-ridge regression, SLOPE, and their respective Knockoff variations. Specifically, \cite{weinstein2017power} computes the asymptotic power of the Knockoff procedure for the LASSO statistics, while \cite{weinstein2020power} extends this result to the Knockoff procedure for the truncated LASSO coefficient statistics. Moreover, \cite{hu2019asymptotics} investigates the trade-off curve and optimal regularization for the SLOPE procedure. The precise calculations in these studies are primarily built upon recent advancements in the high-dimensional asymptotics of $M$-estimators, as demonstrated in works such as \cite{donoho2009message,rangan2011generalized, bayati2011dynamics, karoui2013asymptotic, donoho2016high, berthier2020state}.

The conditional randomization test (\CRT) \cite{candes2018panning} is a Model-X procedure closely associated with the Knockoff filter. {\CRT} generates a valid p-value by calculating the rank of base statistics among its resampled variants. \cite{candes2018panning,wang2022high} show that {\CRT} outperforms Model-X Knockoff in terms of power within certain statistical models, although {\CRT} comes with a higher computational load. Variants of {\CRT} include the conditional permutation test \cite{berrett2020conditional}, the holdout randomization test \cite{tansey2022holdout}, and the distilled conditional randomization test (\dCRT) \cite{liu2022fast}. Among these, {\dCRT} is especially notable for its significant reduction in the computational cost of {\CRT}, making it a key point of interest in this paper.

From a technical viewpoint, our main results and conjectures exploit the asymptotics of Bayes estimators of high dimensional models, as outlined in references such as  \cite{barbier2019optimal, barbier2019adaptive, barbier2016mutual, deshpande2015asymptotic, lelarge2019fundamental,barbier2021performance}. Furthermore, we borrow heuristic tools from statistical physics literature, including the replica method and the interpolation method \cite{mezard2009information, talagrand2010mean}. We should note that it is unclear whether local fdrs used in our procedures are efficiently computable: further exploration of these computational aspects can be found in works \cite{zdeborova2016statistical, barbier2019optimal, celentano2020estimation}.

\section{Statistical limits of Bayesian procedures}\label{sec:BO}

%\sm{Propagate the notations of the introduction. Remind what is $\Omega$. Don't use $\Null$, instead using $\cH_0(P)$. and ... }

%\sm{Add links to equations and give references: principle, if something is not new, or related to something in the literature, should add reference. }

%\sm{Need to have sentences description the purpose of the section/definition/proposition}

% We first study optimal procedures with {\mFDR} control and with {\BFDR} control. These procedures will serve as the base procedures for procedures frequentist {\FDR} control, which will be studied in Section \ref{sec:FO}. 

We begin by deriving the limiting statistical power of FDR controlling procedures within a Bayesian framework. We show that truncating the local fdr ({\TPoP}) is the optimal procedure with {\mFDR} control, and truncating the cumulative local fdr ({\CPoP}) is the optimal procedure with {\BFDR} control (Section \ref{sec:mfdr}). We then consider the Bayesian linear model with isotropic covariates, wherein we derive the asymptotic {\FDP} and {\TPP} for both {\TPoP} and {\CPoP} (Section \ref{sec:asymptotic_FDP_TPP}). Numerical simulations are provided for comparing {\TPoP} and {\CPoP} with the thresholding LASSO procedure (Section \ref{sec:simulation_Bayes}). The statistical limits of {\CPoP} and {\TPoP} will be used to estabilish the frequentist optimality of {\PoPCe} and {\PoEdCe}, to be introduced in Section \ref{sec:FO}.

%{\TPoP} and {\CPoP} procedures will serve as the base procedures for procedures frequentist {\FDR} control, which will be studied in Section \ref{sec:FO}. 

%In this section, we will introduce a procedure which truncate the local fdr which are optimal {\BFDR} controlling procedure and analyze its power asymptotically. Besides, we define {\mFDR} optimality and suggest another testing procedure which also truncates the local fdr and attains such optimality. Recall the Bayes model where we have family of distributions $\{ P_{\bZ \vert \bbeta} \}_{\bbeta \in \cB} \subseteq \ProbFam(\R^d \times \cY)$ parameterized by $\bbeta \in \cB$, and given a prior $\bbeta \sim \Pi \in \ProbFam(\cB)$. We further assume $\Pi$ has a point mass at $0$, and consider the case where the $j$-th null hypothesis is equivalent with $H_{0j}:\beta_j=0$. Note that the Bayesian linear model is an example of a model that meets our assumption under some mild conditions (see Assumption \ref{ass:Bayesian_linear_model} for details). 

% \subsection{The optimal Bayesian procedures}

\subsection{The optimal Bayesian procedures} \label{sec:mfdr}
%\sm{Define what is mFDR optimality}\\
%Here we first introduce a new notion of optimality; optimal test with {\mFDR} control. \sm{Explain why we }

The local false discovery rate (local fdr) \cite{efron2001empirical, efron2005local} is a widely-used tool in multiple hypothesis testing. Assuming a Bayesian model, the local fdr calculates the posterior probability of a hypothesis being null. Following the setup of Section \ref{sec:freq}, we denote $P_j(\cD)$ to be the $j$-th local fdr, under the Bayesian model $\cD \sim P_{\bZ \vert \bbeta_0}$ with prior $\bbeta_0 \sim \Pi$, 
\begin{align}\label{eqn:def-local-fdr}
P_j(\cD) = \P( j \in \cH_0(P) \vert \cD) = \P(\beta_{0, j} = 0\vert \cD).
\end{align}
Intuitively, a larger local fdr suggests a higher likelihood of the corresponding hypothesis being null under the Bayesian model. 

\paragraph{The {\TPoP} procedure for $\mFDR$ control} The {\TPoP} procedure $\bT_P(\cD;t)$, which represents {\tt T}runcating the {\tt Po}sterior {\tt P}robability, truncates the local fdr at level $t$, i.e., rejecting the hypotheses that are unlikely to be null, 
\begin{align}\label{eqn:def_T_P}
T_{P, j}(\cD;t) = \ones\{P_j(\cD) < t \}, ~~~~~ \bT_P = (T_{P, 1}, \ldots, T_{P, d}). 
\end{align}
Prior studies have demonstrated that {\TPoP} gives the optimal {\mTPR} with {\mFDR} control (c.f. Definition \ref{def:mfdr_optimality}) in specific statistical models \cite{sun2007oracle, xie2011optimal}. Extending these results, we next present a general regularity assumption under which we can show the optimality of {\TPoP}. % below shows that the {\TPoP} procedure with a properly chosen threshold is the optimal test with {\mFDR} control, under the following continuity assumption on the model distribution. 
\begin{assumption}[Regularity]\label{ass:marginal_optimality}
Under the Bayesian model $\cD \sim P_{\bZ \vert \bbeta_0}$ with prior $\bbeta_0 \sim \Pi$, the conditional densities $p(\cD|\beta_{0, j}=0)$ and $p(\cD|\beta_{0, j}\ne0)$ exist. Furthermore, $\P(P_j(\cD)<t|\beta_{0, j}=0)$ and $\P(P_j(\cD)<t|\beta_{0, j}\ne0)$ are continuous in $t$ for each $j\in[d]$.
\end{assumption}

We next show the optimality of {\TPoP} under the regularity assumption (proof in Appendix \ref{sec:Bayes_optimality_proof}). 
\begin{proposition}[Optimality of {\TPoP}]\label{prop:marginal_optimality}
Let $A = (\inf_t\mFDR(\bT_P( \cdot;t), \Pi), \sup_t\mFDR(\bT_P(\cdot;t), \Pi) )$ and let Assumption \ref{ass:marginal_optimality} hold. Then for any $\alpha \in A$, there exists $t = t(\alpha) \in (0, 1)$ such that 
\begin{align}\label{eqn:mFDR-alpha}
\mFDR(\bT_P( \cdot;t(\alpha)), \Pi) = \alpha. 
\end{align}
Moreover, for any test statistics $\bT: \Omega \to \{0, 1\}^d$ with $\mFDR(\bT, \Pi) \le \alpha$, we have 
\begin{align}\label{eqn:mTPR-optimal}
 \mTPR(\bT_P( \cdot;t(\alpha) ), \Pi) \ge  \mTPR(\bT, \Pi) .
\end{align}
That is, $\bT_P(\cdot;t(\alpha) )$ is an $(\alpha,\Pi, 0)$-optimal procedure with {\mFDR} control (c.f. Definition \ref{def:mfdr_optimality}). %We remark that here $\bT_P$ implicitly depends on $\Pi$. 
\end{proposition}

%\sm{Here} We include the proof of Proposition \ref{prop:marginal_optimality} in Appendix \ref{sec:Bayes_optimality_proof}. The regularity condition (Assumption \ref{ass:marginal_optimality}) is technical, which ensures the existence of $t(\alpha)$ satisfying Eq. (\ref{eqn:mFDR-alpha}) for any $\alpha \in A$.  

We should note that Proposition \ref{prop:marginal_optimality} does not constitute an entirely new discovery; the optimality of {\TPoP} has previously been established in specific statistical models. As an example, \cite{sun2007oracle} proved that {\TPoP} attains the smallest marginal false negative ratio in the case of a mixture model $X_i|\theta_i\sim\theta_i F_0 + (1-\theta_i) F_1$, where $\{ \theta_i \}_{i \in [d]}$ are independent Bernoulli random variables. Proposition \ref{prop:marginal_optimality} extends and adapts these results to general Bayesian models.

% \subsection{The optimal procedure with {\BFDR} control}\label{sec:bfdr}

%We show that the optimal procedure with {\BFDR} control is similar to truncating the local fdr ({\TPoP}). %However, to achieve optimality with {\BFDR} control, the truncation threshold need to depend on the data. 
% Indeed, the optimal procedure truncates the cumulative local fdr, and we call it the {\CPoP} procedure (standing for {\bf C}umulative {\bf Po}sterior {\bf P}robability). 

\paragraph{The {\CPoP} procedure for {\BFDR} control} 
%We next consider the optimal procedure with {\BFDR} control as in Definition \ref{def:BFDR-optimality}. 
The {\CPoP} procedure $\bC_P$, which represents {\tt C}umulative {\tt Po}sterior {\tt P}robability, truncates the local fdr $\{ P_j(\cD) \}_{j \in [d]}$ (c.f. Eq. (\ref{eqn:def-local-fdr})) at some data dependant threshold that is determined by the cumulative local fdr, 
\begin{align} \label{eqn:def-Cpop}
C_{P, j}( \cD;\lambda) = \ones\{P_j(\cD) < P_{\what K(\lambda, \cD)}(\cD) \}, ~~~~~ \bC_P = (C_{P, 1}, \ldots, C_{P, d}),
\end{align}
where $\what K(\lambda, \cD) \in [d]$ is the number of rejections given by
\begin{align}
\textstyle \what K(\lambda, \cD) \equiv&~ \textstyle  \arg\max_{K \in [d]} \Big( K  - \big( 1 - \lambda N / (K \vee 1) \big) \sum_{j = 1}^K P_{(j)}(\cD) \Big),\label{eqn:def-K}\\
N \equiv&~ \E_{\bbeta_0 \sim \Pi}[\#\{ j : j \not \in \cH_0(P)\}] = \E_{\bbeta_0 \sim \Pi}[\#\{ j : \beta_{0, j} \neq 0 \}]. \label{eqn:def-N}
\end{align}
Here $N$ stands for the expected number of nonnulls, and $\{ P_{(j)}(\cD) \}_{j \in [d]}$ are the order statistics of the local fdr $\{ P_j(\cD) \}_{j \in [d]}$ in increasing order $P_{(1)}(\cD) \le P_{(2)}(\cD) \le \cdots \le P_{(d)}(\cD)$.

Proposition \ref{prop:Bayes_optimality} below shows that the {\CPoP} procedure with a properly chosen $\lambda$ is the optimal test with {\BFDR} control (c.f. Definition \ref{def:BFDR-optimality}), under the following continuity assumption of the model distribution. 
\begin{assumption}[Continuity]\label{ass:Bayes_optimality}
Under the Bayesian model $\cD \sim P_{\bZ \vert \bbeta_0}$ with prior $\bbeta_0 \sim \Pi$, the distribution of the random vector $(P_1(\cD),P_2(\cD),...,P_d(\cD))$ is absolutely continuous to the Lebesgue measure on $\R^{d}$. 
\end{assumption}

\begin{proposition}[Optimality of {\CPoP}]\label{prop:Bayes_optimality}
Let $A = (\inf_\lambda \BFDR(\bC_P(\cdot;\lambda), \Pi),\sup_\lambda \BFDR(\bC_P( \cdot;\lambda), \Pi) )$ and let Assumption \ref{ass:Bayes_optimality} hold. Then for any $\alpha \in A$, there exists $\lambda = \lambda(\alpha) \in \R_{\ge 0}$ such that 
\begin{align}\label{eqn:BFDR-alpha}
\BFDR(\bC_P(\cdot;\lambda(\alpha)), \Pi) = \alpha. 
\end{align}
Moreover, for any test statistics $\bT: \Omega \to \{0, 1\}^d$ with $\BFDR(\bT, \Pi) \le \alpha$, we have 
\begin{equation}\label{eqn:BFDR_optimal}
\begin{aligned}
\mTPR(\bC_P( \cdot;\lambda(\alpha)), \Pi) \ge \mTPR(\bT, \Pi). 
\end{aligned}
\end{equation}
That is, $\bC_P( \cdot;\lambda(\alpha))$ is an $(\alpha,\Pi, 0)$-optimal procedure with {\BFDR} control (c.f. Definition \ref{def:BFDR-optimality}). %(note that $\bC_P$ implicitly depends on $\Pi$).
\end{proposition}

The proof of Proposition \ref{prop:Bayes_optimality} is contained in Appendix \ref{sec:Bayes_optimality_proof_BFDR}. We note that the continuity condition (Assumption \ref{ass:Bayes_optimality}) is technical, ensuring the existence of $\lambda(\alpha)$ satisfying Eq. (\ref{eqn:BFDR-alpha}) for any $\alpha \in A$. This assumption is mild and is satisfied as long as $\cD$ admits a continuous probability density function, and the map $\cD\mapsto (P_1(\cD),...,P_d(\cD))$ is non-degenerate almost everywhere. A concrete example satisfying this assumption is the Bayesian linear model (see Assumption \ref{ass:Bayesian_linear_model} for details). Furthermore, we note that results similar to Proposition \ref{prop:Bayes_optimality} have also been shown in specific statistical models \cite{muller2004optimal,muller2006fdr}. Proposition \ref{prop:Bayes_optimality} extends and adapts these results to general Bayesian models. 

%Indeed, \cite{sun2007oracle} proved that {\TPoP} is optimal in a sense that it reaches smallest marginal false negative ratio in the case of a mixture model $X_i|\theta_i\sim\theta_i F_0 + (1-\theta_i) F_1$, where $\theta_i$ is independent Bernoulli random variables. Proposition \ref{prop:marginal_optimality} extended such results to Bayesian models of our setup. 

%\noindent\textbf{Remark.} Assumption \ref{ass:Bayes_optimality} is in general satisfied under the general Bayesian setting with extra conditions $\cD=(X,y)\in \R^{c}(c\geq d)$, $\cD$ has probability density function and the map $\cD\mapsto (P_1(\cD),...,P_d(\cD))$ is non-degenerate almost everywhere. A concrete example satisfying this Assumption is the Bayesian linear model (see Assumption \ref{ass:Bayesian_linear_model} for details), in which one can show the above condition holds.

\subsection{The limiting power in Bayesian linear model}\label{sec:asymptotic_FDP_TPP}

%\sm{I am here}
%In Section \ref{sec:mfdr}, \ref{sec:bfdr} we consider the general Bayesian setting and propose two procedures that achieve the {\mFDR} and {\BFDR}-optimality respectively. However, with the following extra assumption of a Bayesian linear model, we are able to precisely calculate the FDP-TPP trade-off curves of the optimal procedures in the asymptotic limit. 

We next derive the limiting power of {\TPoP} and {\CPoP} within the Bayesian linear model. The precise statement of the Bayesian linear model is presented in the forthcoming assumption. 

\begin{assumption}[Bayesian linear models]\label{ass:Bayesian_linear_model}
Assume that we observe $n$ samples $\cD \equiv \{(\bx_i, y_i) \}_{i \in [n]}$, wherein a linear relationship is formed between the response and covariates $y_i = \< \bx_i, \bbeta_0 \> + \eps_i$.  In this equation, $(\eps_i)_{i \in [n]} \sim_{i.i.d.} \cN(0, \sigma^2)$ are Gaussian noises, and $\bbeta_0 = (\beta_{0, 1}, \ldots, \beta_{0, d})^\sT \in \R^d$ is the coefficient vector. In matrix form, we have $\Yb = \Xb \bbeta_0 +\beps$ where $\Yb = (y_1, \ldots, y_n)^\sT \in \R^n$ and $\Xb = (\bx_1, \ldots, \bx_n)^\sT = (\xb_1, \ldots, \xb_d) \in \R^{n \times d}$. We further assume a product prior on $\bbeta_0$, wherein $(\beta_{0, j})_{j \in [d]} \sim_{i.i.d.} \Pi \in \ProbFam(\R)$. It is also assumed that the prior distribution gives $\Pi = \pi_0 \delta_0 + (1 - \pi_0)\Pi_\star$, where $\delta_0$ is the Dirac-delta distribution at $0$, $\pi_0 \in (0, 1)$ is the proportion of null variables, and $\Pi_\star \in \ProbFam(\R)$ is a general distribution without any mass at $0$. Finally, we assume that the covariates follow the isotropic Gaussian distribution $(\bx_i)_{i \in [n]} \sim_{i.i.d.} P_\bX = \cN(\bzero, (1/n) \id_d)$. 
\end{assumption}

%\noindent
%\textbf{Remark.}
%Our assumption on the input covariates $\bx_i$ is technical and would enable us to derive the asymptotic FDP and TPP.  It is also interesting to consider the case when $\bx_i$ are anisotropic, for example, are i.i.d. $\cN(\bzero,(1/n)\Sigma)$ with a general covariance matrix $\Sigma$. In this paper, we focus on isoptropic Gaussian design, and we defer anisotropic case to future work. \\

%The FDP and TPP can be expressed as a functional of the joint empirical distribution of the coordinates $\hat \mu \equiv (1 / d) \sum_{j = 1}^d \delta_{(\bbeta_{0, j}, \cP_j(\cD))}$. Indeed, for the FDP and TPP of the {\TPoP} statistics $\bT(t; \cdot)$, we have 
%\[
%\FDP(\bT_P(t; \cdot)) = \frac{\P_{(\beta_0, P_j) \sim \hat \mu}( \beta_0 = 0, P_j < t)}{\P_{(\beta_0, P_j) \sim \hat \mu}(P_j < t)},~~~~ \TPP(\bT_P(t; \cdot)) = \frac{\P_{(\beta_0, P_j) \sim \hat \mu}( \beta_0 \neq 0, P_j < t)}{\P_{(\beta_0, P_j) \sim \hat \mu}( \beta_0 \neq 0)}. 
%\]

Despite this strong model assumption, it is worth noting that we will later develop procedures that control frequentist {\FDR} under much weaker assumptions of the linear model. But for now, our focus is deriving the power of {\TPoP} and {\CPoP}, under the Bayesian linear model with this strong assumption. 

% It is worth noting that Assumption \ref{ass:Bayesian_linear_model} is relatively strong, and may not align with real datasets. However, we will develop procedures in Section \ref{sec:FO}, which control frequentist {\FDR}, even when this model assumptions are not met. We expect that these proposed procedures are powerful if this model turns out to be approximately correct. On the other hand, in this section, we focus on this strong assumption, and our goal is to figure out the exact power of {\TPoP} and {\CPoP}. 

\paragraph{The Bayes risk in the proportional limit regime} The Bayesian linear model holds a special interest in the high-dimensional regime, where $n, d \to \infty$ and $n / d \to \delta \in (0, \infty)$. This specific regime has been widely examined in the literature \cite{tanaka2002statistical,donoho2009message, rangan2011generalized, bayati2011dynamics, donoho2016high,barbier2019optimal, barbier2019adaptive, barbier2016mutual,celentano2020lasso}. Prior studies have focused on deriving the asymptotic Bayes risk, defined as
\begin{equation}\label{eqn:limiting_risk_linear_model}
R(\delta) = \lim_{n, d \to \infty, n/d \to \delta} \frac{1}{d} \big\| \bbeta_0  - \E[\bbeta_0 | \cD ] \big\|_2^2, 
\end{equation}
where $\E[\bbeta_0 | \cD ]$ is the posterior expectation of the coefficient vector, the Bayes optimal estimator. Notably, Tanaka \cite{tanaka2002statistical} used heuristic statistical physics methods to provide a simple formula for the asymptotic Bayes risk of the high-dimensional linear model, which coincides with the Bayes risk of a scalar Bayes estimation problem. The validity of this formula was first rigorously proved by \cite{barbier2016mutual} through the interpolation method.

More specifically, the high dimensional Bayesian linear model is tightly connected to the following scalar Bayes estimation problem: we consider a scalar signal, $\beta_0$, having a prior distribution $\Pi$, and we obtain a noisy observation, $Y$, of the signal via an additive Gaussian channel, as expressed in:
\begin{equation}\label{eqn:scalar_model}
{\rm Signal:}~~~ \beta_0 \sim \Pi, ~~~~~~~~~ {\rm Observation:}~~~ Y = \beta_0 + \tau G \in \R,~~~~~~~~~ {\rm Noise:}~~~ G \sim \cN(0, 1),
\end{equation}
where $\tau \in \R_{\ge 0}$ represents the noise level to be determined. In this scalar model, given the observation $Y$, the Bayes optimal estimator, considering the squared loss, is the posterior expectation estimator, as presented in:
\begin{equation}\label{eqn:posterior_expectation_low}
\what \beta = \cE(Y; \Pi, \tau) \in \R,~~~~~~~ {\rm where}~~~ \cE(y; \Pi, \tau) = \E_{(\beta_0, Z) \sim \Pi \times \cN(0, 1)}[ \beta_0 \vert \beta_0 + \tau Z = y]. 
\end{equation}
As a result, the Bayes risk of the scalar model yields:
\begin{equation}
R(\tau, \delta) = \E_{(\beta_0, Z) \sim \Pi \times \cN(0, 1)}[(\beta_0 - \cE(\beta_0 + \tau Z ; \Pi, \tau))^2]. 
\end{equation}
Tanaka \cite{tanaka2002statistical} shows that the limiting risk (\ref{eqn:limiting_risk_linear_model}) of the high dimensional Bayesian linear model, under Assumption \ref{ass:Bayesian_linear_model}, coincides with the limiting risk of the scalar Bayes estimation problem
\begin{equation}
R(\delta) = R(\tau_\star, \delta).\label{eqn:equi_risk}
\end{equation}
Here, the noise level $\tau_\star$ is given by the global minimizer of a potential function $\phi$ 
\begin{align}
\tau_\star=\argmin_{\tau\geq 0}
\phi(\tau^2; \Pi, \delta, \sigma^2) \equiv\argmin_{\tau\geq 0}\Big\{ \frac{\delta \sigma^2}{2 \tau^2} - \frac{\delta}{2} \log\Big( \frac{\delta \sigma^2}{\tau^2}\Big) + \MI(\Pi, \tau^2)\Big\},\label{eqn:potential_effective_noise}
\end{align}
where $\MI$ is the mutual information between $\beta_0$ and $Y$ in the model \eqref{eqn:scalar_model}, 
\[
\MI(\Pi, \tau^2) = \E_{\beta_0, Y}\Big[\log\Big( \frac{p(Y| \beta_0)}{p(Y)}\Big) \Big] = - \frac{1}{2} - \E_{\beta_0, Y}\log\Big\{ \int e^{-(Y - \beta)^2 / (2\tau^2)} \Pi(\de \beta) \Big\}. 
\] 
Taking derivative of $\phi$ with respect to $\tau^2$, we deduce that $\tau_\star$ satisfies the following self-consistent equation
\begin{equation}\label{eqn:self_consistent}
\tau^2 = \sigma^2 + \frac{1}{\delta} R(\tau, \delta). 
\end{equation}

\paragraph{The limiting {\FDP} and {\TPP}} Our primary focus here is the asymptotic {\FDP} and {\TPP} (c.f. Eq. (\ref{eqn:FDP_TPP})) of the {\TPoP} and {\CPoP}, which depend on the joint empirical distribution of $\{ (\beta_{0, j}, P_j(\cD))\}_{j \in [d]}$ (recall that $P_j(\cD) = \P(\beta_{0, j} = 0| \cD)$ gives the local fdr). Using the heuristic replica calculation in Appendix \ref{sec:justification_conj_1}, we demonstrate that the joint empirical distribution of $\{ (\beta_{0, j}, P_j(\cD))\}_{j \in [d]}$ is likewise linked to its counterpart in the scalar model. Specifically, we reconsider the scalar model~\eqref{eqn:scalar_model}, and consider the associated hypothesis testing problem: test the null hypothesis that $\beta_0 = 0$ given the observation $Y$. According to the Neyman-Pearson lemma, the optimal test corresponds to the likelihood ratio test, equivalent to truncating the local fdr $\Phi$ of the scalar model, as given by
\begin{equation}\label{eqn:posterior_probability_low}
\Phi = \PostProb(Y; \Pi, \tau_\star), ~~~~~~~ \PostProb(y; \Pi, \tau) =  \P_{(\beta_0, Z) \sim \Pi \times \cN(0, 1)}( \beta_0 = 0 \vert \beta_0 + \tau Z = y ). 
\end{equation}
The replica calculations suggest that for any sufficiently smooth function $\psi: \R \times [0, 1] \mapsto \R$, there is
\begin{align}\label{eqn:replica-main-test-result}
\textstyle \lim_{d\to \infty, n/d \to \delta} \frac{1}{d}\sum_{j = 1}^d \psi(\beta_{0, j}, P_j(\cD))  = &~ \E[ \psi(\beta_0, \Phi )],
\end{align}
where $(\beta_0, \Phi)$ follows the joint distribution specified by Eq.~(\ref{eqn:scalar_model}) and (\ref{eqn:posterior_probability_low}). This equation gives rise to the subsequent conjecture, stating that the asymptotic {\FDP} and {\TPP} of {\TPoP} (and also {\CPoP}) correspond to the type-I error and the power of the scalar hypothesis testing problem. 

\begin{conjecture}[Limiting {\FDP} and {\TPP} of {\TPoP} and {\CPoP}]\label{conj:Posterior_probability_curve}
Let Assumption \ref{ass:Bayesian_linear_model} hold. The {\FDP} and {\TPP} of the {\TPoP} procedure (Eq.~(\ref{eqn:def_T_P})) with parameter $t$ gives
\begin{equation}\label{eqn:lim_tpop}
\begin{aligned}
\lim_{d\to \infty, n/d \to \delta} \FDP(\bT_P(\,\cdot\,;t)) =&~ \P( \beta_0 = 0 | \Phi  < t), \\
\lim_{d\to \infty, n/d \to \delta} \TPP(\bT_P( \,\cdot\,;t)) =&~ \P( \Phi  < t | \beta_0 \neq 0). 
\end{aligned}
\end{equation}
The {\FDP} and {\TPP} of the {\CPoP} procedure (Eq.~(\ref{eqn:def-Cpop})) with parameter $\lambda$ gives
\begin{equation}\label{eqn:lim_cpop}
\begin{aligned}
\lim_{d\to \infty, n/d \to \delta} \FDP(\bC_P( \,\cdot\,;\lambda)) =&~ \P( \beta_0 = 0 | \Phi  < t_\star(\lambda)), \\
\lim_{d\to \infty, n/d \to \delta} \TPP(\bC_P( \,\cdot\,;\lambda)) =&~ \P( \Phi  < t_\star(\lambda) |  \beta_0 \neq 0 ), 
\end{aligned}
\end{equation}
where $t_\star(\lambda)$ is given by
\begin{equation}\label{eqn:lambda_to_tstar}
t_\star(\lambda) = \argmax_{t \in [0, 1]} \Big( \P(\Phi < t) - \big( 1 - \lambda (1 - \pi_0) / \P(\Phi < t) \big) \cdot \E[ \Phi \cdot \ones\{ \Phi < t\} ] \Big). 
\end{equation}
\end{conjecture}

The intuitions of the conjecture are provided in Appendix \ref{sec:justification_conj_1}. Notably, under similar assumptions of the Bayesian linear model, the analogous asymptotics of {\FDP} and {\TPP} have been rigorously derived for the thresholded LASSO procedure in \cite{weinstein2020power}, leveraging the approximate message passing (AMP) machinery. Applying the AMP machinery to our conjecture is not a straightforward task. The proof of this conjecture poses an intriguing open question and is a topic we plan to explore in future works. 
 %In this work, we use the statistical physics tool, replica method, to derive Conjecture \ref{conj:Posterior_probability_curve}. Because of the physics tools used in the derivation is a heuristic but not a mathematical proof, this power analysis is left as a conjecture. The derivation procedure is contained in Appendix \ref{sec:powerproof}. 

%Note that we can choose $\lambda$ or $t$ by solving $\lim_{d\to \infty, n/d \to \delta} \FDP = \alpha$, then corresponding {\TPoP} or $\CPoP$ will have $\FDP=\alpha$ asymptotically, and thus will be {\mFDR}, {\BFDR} optimal with level $\alpha$ respectively. Furthermore, when $\lambda$ is properly chosen corresponding to each $t$, {\TPoP} and $\CPoP$ has same asymptotic limit of {\FDP} and $\TPP$.

%Note that the thresholds $t$ and $\lambda$ control the tradeoff between {\FDP} and {\TPP}. 
% to verify Conjecture \ref{conj:Posterior_probability_curve}
 
Conjecture \ref{conj:Posterior_probability_curve} immediately reveals that, despite being different procedures, {\TPoP} and {\CPoP} yield asymptotically identical {\FDP}-{\TPP} tradeoff curves. This is not unexpected: {\CPoP} corresponds to truncating the local fdr at a certain data-dependent threshold; in high dimension, this threshold will concentrate and coincide with the threshold employed in {\TPoP}.

\subsection{Numerical simulations}\label{sec:simulation_Bayes}

\begin{figure}\centering
\includegraphics[width=0.9\textwidth]{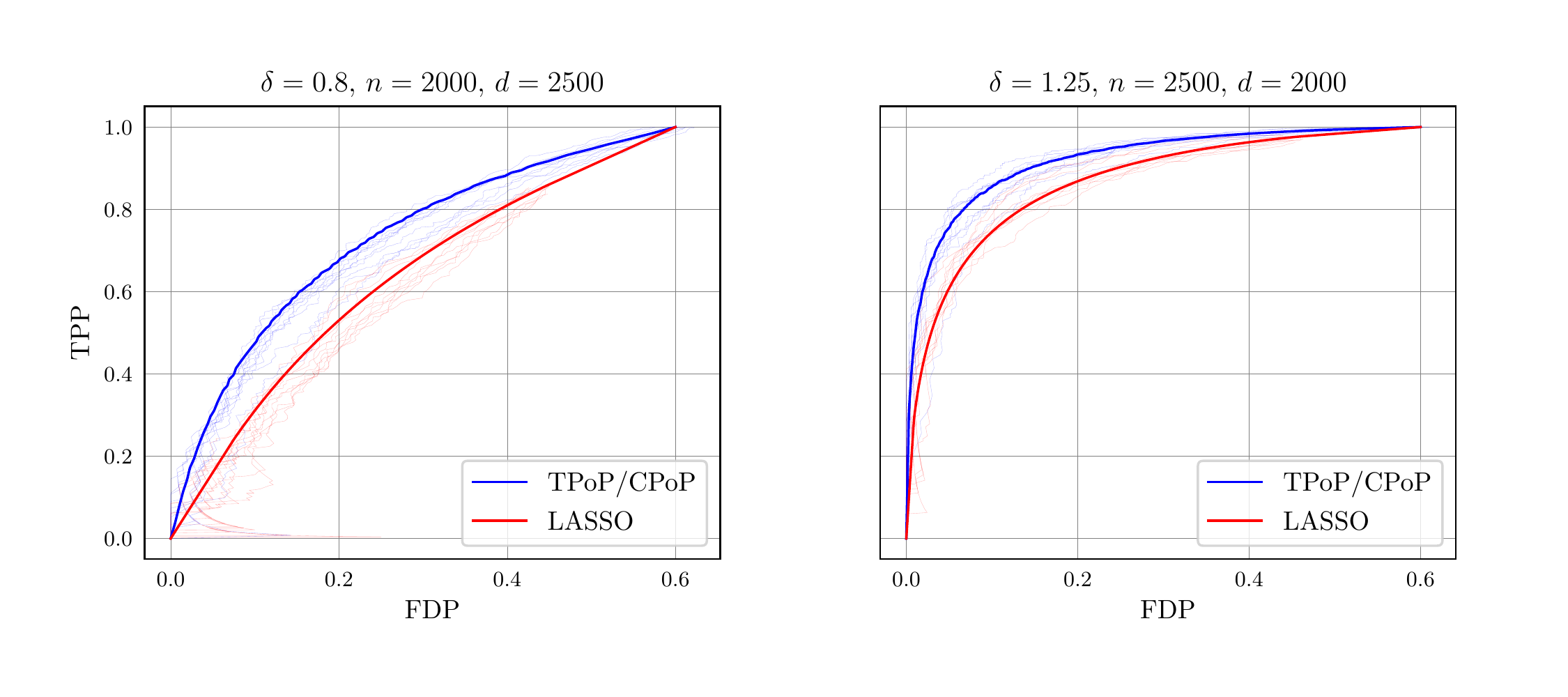}
\vskip-0.3cm
\caption{FDP-TPP tradeoff curves of {\TPoP}/{\CPoP} and thresholded LASSO. Dark thick lines are the analytical prediction of $\lim\FDP$ and $\lim\TPP$, and thin light curves represent realizations of each procedure from 10 simulated instances. Each panel corresponds to different values of $\delta$: from left to right, $\delta=0.8,1.25$ and $n=2000,2500$ respectively. In all panels, $\sigma=0.25$, and $\Pi$ is a three delta prior with mass 0.6 at zero and mass 0.2 at both $\pm 1$. } 
\label{fig:comparison}
\end{figure}

We next perform numerical simulations illustrating the {\FDP}-{\TPP} tradeoff curves of {\TPoP}/{\CPoP} and the thresholded LASSO procedure, considering two distinct values of $\delta = n/d$. It is important to note that the trade-off curve for {\CPoP} aligns with that of {\TPoP} as both procedures reject hypotheses with a small local fdr.

The thresholded LASSO procedure rejects the hypotheses with large absolute value of the corresponding LASSO coefficient. More specifically, thresholded LASSO rejects the $j$-th hypothesis when $|\what{\beta}_j(\lambda)|>t$ for some cutoff $t$, where $\what\bbeta(\lambda)=\argmin_{\bbeta}\frac12\lVert \Yb -\Xb \bbeta \rVert_2^2+\lambda\| \bbeta\|_1$. \cite{weinstein2020power} derives the asymptotic {\FDP} and {\TPP} of the thresholded LASSO procedure by deriving the following formula: in the limit of $n,d\rightarrow\infty$ and $n / d \to \infty$, for any sufficiently smooth function $\psi: \R \times [0, 1] \mapsto \R$, we have 
\begin{equation}\label{eqn:lasso_limit}
\lim_{d\to \infty, n/d \to \delta} \frac{1}{d}\sum_{j = 1}^d \psi(\beta_{0, j}, \what{\beta}_j(\lambda))  =  \E_{(\beta_0,Z)\sim\Pi\times\cN(0,1)}[ \psi(\beta_0, \eta_{\alpha'\tau'}(\beta_0+\tau'Z) )].
\end{equation}
Here, $\alpha'$ and $\tau'$ are the unique solutions of the self-consistent equation given by:
\begin{equation}\label{eqn:scf-lasso}
\begin{aligned}
\tau'^2&=\sigma^2+\frac1\delta\E_{(\beta_0,Z)\sim\Pi\times\cN(0,1)}(\eta_{\alpha'\tau'}(\beta_0+\tau'Z)-\beta_0)^2, \\
\lambda&=\bigg(1-\frac1\delta\P_{(\beta_0,Z)\sim\Pi\times\cN(0,1)}(|\beta_0+\tau'Z|\ge\alpha'\tau')\bigg)\alpha'\tau',
\end{aligned}
\end{equation}
where $\eta_\theta(x)\equiv \sign(x)\cdot(|x|-\theta)_+$ gives the soft-thresholding operator. The characterization of the joint empirical distribution of $\{(\beta_{0, j}, \what \beta_j(\lambda))\}_{j \in [d]}$ mirrors our Formalism \ref{fm:post_mean_zero}, instrumental in deducing Conjecture \ref{conj:Posterior_probability_curve}. This characterization can then be employed to derive the limiting {\FDP} and {\TPP} of the thresholded LASSO procedure. Finally, \cite{weinstein2020power} demonstrates that the optimal choice of $\lambda$ in the thresholded LASSO procedure, yielding the optimal {\FDP}-{\TPP} trade-off curve, is provided by:
\begin{equation}\label{eqn:lasso_lambda}
\lambda_\star=\argmin_{\lambda>0}\lim_{d\to \infty, n/d \to \delta}\frac1d \big\| \what\beta(\lambda)-\beta_0 \big\|_2^2. 
\end{equation}
% it is known that for each cutoff $t$, the asymptotic limit of {\FDP} and $\TPP$ are minimized and maximized respectively, amongst the thresholded LASSO procedure with different choices of $\lambda$. 

Figure \ref{fig:comparison} showcases the analytical predictions of {\FDP} and {\TPP} for {\TPoP} and thresholded LASSO, alongside numerically simulated curves. We simulate 10 instances of $(\Yb, \Xb)$ from the Bayesian linear model, as detailed in Assumption \ref{ass:Bayesian_linear_model}, for $\delta=0.8, 1.25$ ($n=2000,2500$ and $d=2500,2000$ respectively). We choose $\sigma=0.25$ and $\Pi=0.6\cdot\delta_0+0.2 \cdot\delta_{-1}+0.2 \cdot\delta_{1}$ (so that $\pi_0=0.6$, and $\Pi_\star= 0.5 \cdot\delta_{-1}+ 0.5\cdot \delta_{1}$).
For each {\TPoP} simulated curve, we fix a linear model instance, gradually raise the cutoff $t$ from $0$ to $1$, and compute the empirical $(\FDP,\TPP)$ for each cutoff $t$. For analytical curves of {\TPoP}, we first solve the self-consistent equation outlined in Eq. \eqref{eqn:self_consistent} to obtain $\tau_\star$ for each $t$, and then derive the asymptotic $\FDP(\bT_P( \cdot; t))$ and $\TPP(\bT_P( \cdot; t))$ following Conjecture \ref{conj:Posterior_probability_curve}. For each simulated curve of the thresholded LASSO, we first determine the optimal regularization parameter $\lambda_\star$ according to Eq. (\ref{eqn:lasso_lambda}), then increase the cutoff $t$ from $0$ to $3$ and calculate the empirical $(\FDP,\TPP)$ at each cutoff $t$. For thresholded LASSO analytical curves, we use Eq. \eqref{eqn:lasso_limit} and \eqref{eqn:scf-lasso} to calculate the limiting $(\FDP, \TPP)$. 

As shown in Figure \ref{fig:comparison}, the simulated curves closely align with the corresponding analytical curves, which are derived using Conjecture \ref{conj:Posterior_probability_curve}. Moreover, at the same {\FDP} level, {\TPoP} consistently achieves a higher {\TPP} than the optimally regularized thresholded LASSO. These findings are in line with the optimality results presented for {\TPoP} and {\CPoP} in Proposition \ref{prop:marginal_optimality} and \ref{prop:Bayes_optimality}.

% \begin{figure}
% \includegraphics[width=0.9\textwidth]{figs/fig2}
% \caption{Bayes/thresholded LASSO tradeoff curves. Thin curves represent realizations and thick curves are the theoretical predictions. Parameters for each simulations are the same with that of figure \ref{fig:optimalcurve}.}
% \label{fig:comparison}
% \end{figure}

% \section{Optimal procedures with frequentist {\FDR} control}\label{sec:FO}

\section{Achieving the optimal power with frequentist {\FDR} control}\label{sec:FO}

%In this section, we introduce our main testing procedures, the {\PoPCe} and {\PoEdCe}. We assume the Model-X setting from Section \ref{sec:freq}. Namely, the marginal distribution of the covariates $\bX$ of  $P\in P(\cX\times \cY)$ is known to be $P_\bX$, i.e., $P\in \cM(P_\bX)$. We claim that our proposed procedures can achieve FDR-optimality, in the sense that they have valid FDR control under finite sample Model-X setting, and asymptotically achieve the optimal $\TPP$ when the data are from a Bayesian linear model (Assumption \ref{ass:Bayesian_linear_model}). For  {\PoPCe} and {\PoEdCe}, we  assume the prior $\Pi$, noise level $\sigma^2$ of the Bayesian linear model are known. An empirical version of {\PoEdCe}, $\EPoEdCe$, which guarantee FDR-optimality under arbitrary Bayesian linear model, is also proposed in this section.

We have established that {\TPoP} and {\CPoP} are Bayes-optimal under the condition of correct model specification, though they may not control the frequentist {\FDR} in cases of misspecified models. As Lemma \ref{lem:bound} illustrates, the power (\mTPR) of the optimal test with frequentist {\FDR} control cannot exceed the power of the optimal test with {\BFDR} control, which aligns with the power of {\CPoP}. This naturally leads us to question whether the inequality of Lemma \ref{lem:bound} is tight. More precisely, can a test with frequentist {\FDR} control achieve the optimal power of {\CPoP}? This section introduces two testing procedures that affirmatively answer this question.

In Section \ref{sec:CRT-eBH}, we first revisit three recently proposed methodologies designed to control the {\FDR} from finite samples. Building on these methodologies, in Section \ref{sec:PoPCe}, we devise procedures {\PoPCe} and {\PoEdCe} that control the frequentist {\FDR} in linear models within the Model-X setting, given a known $\cL(\bX)$. We demonstrate that these two procedures achieve the power of {\CPoP} asymptotically, assuming correctly specified models, in Section \ref{sec:FDR_optimality}.  Numerical simulations of the proposed procedures can be found in Section \ref{sec:numeric-frequentist}. 

%\sm{Reemphasize the question, whether there is xxxx which control frequentist FDR and achieve the optimal statistical power...} \sm{Say Lemma 1 tells us that, as long as we an design a procedure that attains the Bayes powerxxxx, then it is frequentist optimal xxx. }
% Finally, in Section \ref{sec:EPoEdCe} we  propose an empirical Bayes variant of {\PoEdCe} and show its validity and optimality without knowing the prior. 

%\lc{should be "in the Model-X setting with known $\cL(\bX)$"? }

% In this section, we develop optimal procedures that control {\FDR} from finite samples. We will build on several recently developed testing procedures, and we review them in Section \ref{sec:CRT-eBH}. In Section \ref{sec:PoPCe}, we introduce the {\PoPCe} and {\PoEdCe} procedures, and show their optimality in Section \ref{sec:FDR_optimality}. Finally, we propose an empirical Bayes variant of {\PoEdCe}, and show its optimality in Section \ref{sec:EPoEdCe}. 

% \subsection{Conditional randomization test (\CRT) and Benjamini-Hochberg with e-values (\eBH)}\label{sec:CRT-eBH}

\subsection{Building blocks}\label{sec:CRT-eBH}

%Before presenting {\PoPCe} and {\PoEdCe}, we briefly introduce several procedures  that are used in the construction of them. 

{\PoPCe} and {\PoEdCe} are built upon three multiple testing methodologies: conditional randomization test (\CRT), distilled conditional randomization test (\dCRT), and Benjamini-Hochberg with e-values (\eBH). A brief review of these methodologies is provided below.

\vskip0.2cm
\noindent
\textbf{Conditional Randomization Test (\CRT)} \cite{candes2018panning}. Consider the model setup as described in Section~\ref{sec:model-setup}. Given the dataset $(\Yb, \Xb)$ and access to the joint distribution of the covariates $\cL(\bX)$, the conditional randomization test transforms a base statistic into a valid p-value. Here, the base statistic, denoted as $T_j=T(\Yb, \Xb_{-j}, \xb_{j})$, provides an estimate of the contribution of covariate $X_j$ to the outcome $Y$. This p-value will be valid under the null hypothesis $H_{0j}: Y \perp X_j | \bX_{-j}$. Specifically, {\CRT} calculates a p-value $p_j$ by executing the subsequent three steps:
\begin{itemize}
\item[(1)] Generate $K$ conditionally independent covariates $\tilde \xb_j^{(k)}| \Xb_{-j} \sim \cL(X_j | \bX_{-j})$ for $k \in [K]$, where $\cL(X_j | \bX_{-j})\in\ProbFam(\R)$ is the conditional distribution induced by $\cL(\bX)$. 
\item[(2)] Compute the associated statistics $T_j^{(k)}= T(\Yb, {\Xb}_{-j}, \tilde\xb_j^{(k)})$ for $k \in [K]$. 
\item[(3)] Take $p_j = [1 + \sum_{k = 1}^K \{ T_j^{(k)} \le T_j \} ] / (1 + K)$ as the proportion of $\{ T_j^{(k)} \}_{k \in [K]}$ that are smaller than or equal to $T_j$.
\end{itemize}
Under the null hypothesis $H_{0j}$, since $(\Yb,\Xb_{-j},\tilde \xb_j^{(k)})$ and $(\Yb,\Xb_{-j},\xb_j)$ are identically distributed, $p_j$ follows the uniform distribution over $\{1/(K+1),2/(K+1),...,1\}$, confirming it as a valid p-value.

% Since $(\Yb,\Xb_{-j},\tilde \xb_j^{(k)})\overset{d}{=}(\Yb,\Xb_{-j},\xb_j)$ under the null hypothesis $H_{0j}$, $p_j$ follows the uniform distribution over $\{1/(K+1),2/(K+1),...,1\}$ and hence is a valid p-value. 

One potential concern is that the {\CRT} p-values $\{ p_j \}_{j \in [d]}$ are not necessarily independent. Consequently, the application of the Benjamini-Hochberg (\BH) procedure on $\{ p_j \}_{j \in [d]}$ may not guarantee control over the {\FDR} from finite samples. Another limitation of {\CRT} is its computational burden. The procedure necessitates computing the base statistics function, $T$, a total of $K \times d$ times. This can be computationally intensive when $T$ represents a complicated statistic, such as the LASSO estimator.

\vskip0.2cm
\noindent
\textbf{Distilled Conditional Randomization Test (\dCRT)} \cite{liu2022fast}. The Distilled Conditional Randomization Test (\dCRT) is similar to \CRT, but it alleviates the computational burden by employing a specialized base statistic $T$. This is represented as $T(\Yb, \Xb_{-j},\xb_j)= \tilde T(\Yb, \xb_j, \bd_y, \bd_x)$, where $\bd_y=\bd_y(\Yb, \Xb_{-j})$ and $\bd_x=\bd_x(\Xb_{-j})$ are distilled statistics, encoding the information of $\Xb_{-j}$ contained in $\Yb$ and $\xb_j$ respectively. Given this structure, $T$ depends on $\Xb_{-j}$ only through $\bd_y, \bd_x$, allowing us to reduce the repetitive computations seen in step (2) of {\CRT}. 

For example, consider a scenario where we intend to use the LASSO estimator to derive a base statistic. We can choose $\bd_y \equiv \Xb_{-j} \what\bbeta$, where $\what\bbeta = \what\bbeta_\lambda(\Yb,\Xb_{-j})$ is the LASSO solution for fitting $\Yb$ on $\Xb_{-j}$ with regularization parameter $\lambda$. Concurrently, we let $\bd_x \equiv \E[\xb_j|\Xb_{-j}]$. Then we can choose $T(\Yb,  \Xb_{-j},\xb_j) \equiv |\<\Yb - \bd_y,\xb_{j}- \bd_x\>|$ as the base statistics.

\vskip0.2cm
\noindent
\textbf{Benjamini-Hochberg procedure with e-values (\eBH)} \cite{wang2022false}. {\eBH} is designed for finite-sample {\FDR} control. This method is a variant of the Benjamini-Hochberg (\BH) procedure \cite{benjamini1995controlling}, using e-values as substitutes for p-values. Specifically, a random variable, $e$, is termed a \textit{valid e-value} if the expectation under the null hypothesis satisfies $\E_{H_{0}}[e] \leq 1$. Given $d$ hypotheses and their corresponding $d$ e-values, denoted as $\{ e_j \}_{j \in [d]}$, the {\eBH} procedure rejects the $k$ hypotheses with the largest e-values (ordered from largest to smallest as $\{ e_{(j)} \}_{j \in [d]}$). Here 
\begin{align}
k=\max \left\{m \in\{0,...,d\} :  \frac{d_0}{m e_{(m)}}\leq \alpha \right\},~~~~ d_0 \text{ is the number of true null hypothesis}. 
\end{align}
Contrasting with the {\BH} procedure, which necessitates additional structural assumptions on the p-values (e.g., PRDS) to ensure finite-sample {\FDR} control \cite{benjamini2001control}, {\eBH} consistently controls {\FDR} at level $\alpha$, irrespective of correlation among the e-values \cite{wang2022false}. 

A noteworthy point is that one can use any p-to-e calibrator \cite{vovk2021values,shafer2019language} to convert a valid p-value to a valid e-value, hence providing {\eBH} with great flexibility for valid {\FDR} control. However, applying a naive p-to-e calibrator might result in a power loss, and special treatments are required to make {\eBH} as powerful as the traditional {\BH} procedure.

\subsection{{\PoPCe} and {\PoEdCe} procedures}\label{sec:PoPCe}

% \subsubsection{The {\PoPCe} procedure}
%\sm{Write a detailed description of {\PoPCe} and {\PoEdCe} algorithms: need to be detailed, see Section 4.1 of \url{https://arxiv.org/pdf/2110.04184.pdf}. Need to have: (1) line-to-line description of procedures outside of the procedure box; (2) explanation of hyperparameters. (3) Explanation of key steps, and give reference to papers and equations when necessary. (4) In the algorithms, add colored comments}
%We first describe the {\PoPCe} procedure. 

% Since {\TPoP} only guarantees {\mFDR} control when the Bayesian model is known,  

The {\PoPCe} ({\tt P}osterior {\tt P}robability + {\tt C}onditional randomization test + {\eBH}) procedure employs {\TPoP} as the base statistics, wrapping it using {\CRT} (in the Model-X setting with known $\cL(\bX)$) and {\eBH}. Specifically, we first apply {\CRT} to {\TPoP} to generate p-values, denoted by $\{ p_j\}_{j \in [d]}$. Subsequently, we construct valid e-values from these $p$-values using a carefully chosen p-to-e calibrator. Eventually, we implement the {\eBH} procedure, which controls {\FDR} from finite samples. The full algorithm is presented in Algorithm \ref{alg:popce}. Each step is explained as follows:
\begin{itemize}
\item Line \ref{line_1_1}-\ref{line_1_2} (Compute the p-to-e calibration threshold): We first compute the p-to-e calibration threshold $q=\Psi(t(\alpha-\eps))$. Here, $\Psi$ represents the cumulative distribution function (CDF) of $\PostProb(\tau_\star Z;\Pi,\tau_\star)$ when $Z\sim\cN(0,1)$ (c.f. Eq. (\ref{eqn:posterior_probability_low}) for the definition of $\cP$). This is inherently the asymptotic CDF of the local fdr of a null coordinate in the Bayesian linear model. Furthermore, 
\begin{equation}\label{eqn:definition-t-alpha}
\begin{aligned}
t(\alpha-\eps) \equiv&~  \max\Big\{  s \in [0, 1]: \lim_{d\to \infty, n/d \to \delta}\FDP( \bT_P(\cdot; s) ; \Pi) \le \alpha - \eps \Big\} \\
=&~ \max\{s: \P( \beta_0 = 0 | \PostProb(Y; \Pi, \tau_\star)  < s) \le \alpha - \eps\}
\end{aligned}
\end{equation}
represents the effective truncation threshold of the {\TPoP} procedure, required for calibrating the effective {\FDR} at level $\alpha - \eps$ in the Bayesian linear model. Notice that the second equality above is due to the limiting formula of {\FDP} for {\TPoP}, which is explicitly given in Conjecture \ref{conj:Posterior_probability_curve}. The p-to-e calibration threshold $q$ is used to calculate e-values in Line \ref{line_1_10}. 

%Intuitively speaking, $q$ is an estimate of the proportion of true null hypotheses that we will reject, which is a consistent estimate when the model is correctly specified. Such a choice of $q$ is essential for us to obtain near-optimal power. 

\item Line \ref{line_1_4}-\ref{line_1_9} (Conditional randomization test): We first compute the local fdr $u_j = P_j(\Yb, \Xb) = \P(\beta_{0, j} = 0 | \Yb, \Xb)$ for each coordinate $j \in [d]$. Subsequently, for each coordinate $j \in [d]$, we generate conditionally independent covariates $\{ \tilde \xb_j^{(k)} \}_{k \in [K]}$ given $\Xb_{-j}$ from the conditional distribution $\cL(X_j | \bX_{-j})$. We then compute the corresponding local fdrs denoted by $\{ u_j^{(k)} = P_j(\Yb, \Xb_{-j}, \tilde \xb_j^{(k)}) \}_{k \in [K]}$. Finally, we let $p_j$ be the proportion of $\{ u_j^{(k)} \}_{k \in [K]}$ that are smaller than $u_j$. When the null hypothesis $H_{0j}: Y \perp X_j|\bX_{-j}$ holds, $u_j$ and $u_{j}^{(k)}$ have the same distribution since 
\[
(\Yb,\Xb_{-j},\tilde \xb_j^{(k)})\overset{d}{=}(\Yb,\Xb_{-j},\xb_j).
\]
Hence, $\{ p_j \}_{j \in [d]}$  are valid p-values. %\sm{Here} we will show that $u_j^{(k)}$ has asymptotically the same distribution whether the $j$-th null hypothesis holds or not. 
%We will later show that converting local fdrs $\{ u_j\}_{j \in [d]}$ to p-values $\{ p_j \}_{j \in [d]}$ using the conditional randomization test will not lose power: rejecting local fdrs $\{ u_j \}_{j \in [d]}$ less than $t(\alpha-\eps)$ ({\TPoP}) is asymptotically equivalent to rejecting p-values $\{ p_j \}_{j \in [d]}$ less than $q$. 

\item Line \ref{line_1_10} (Construct e-values): We construct valid e-values $\{ e_j \}_{j \in [d]}$ from p-values $\{ p_j \}_{j \in [d]}$ using the p-to-e calibrator $e_j= \ones\{ p_j\le q \} /q$. Here, $q$ is previously computed in Line \ref{line_1_1}-\ref{line_1_2}.

\item Line \ref{line_1_12} (\eBH): We finally implement {\eBH} on $\{ e_j\}_{j \in [d]}$, which ensures frequentist {\FDR} control at level $\alpha$. Simple algebra demonstrates that {\eBH} is equivalently rejecting the hypotheses $\{j : p_j \le q \}$ if 
\begin{equation}\label{eqn:ebh-condition}
q \pi_0 d /  |\{j: p_j \le q \}| < \alpha,
\end{equation}
and rejecting nothing otherwise. As will be demonstrated later, when the Bayesian linear model is well-specified, the choice of $q$ results in the concentration of $|\{j: p_j \le q \}|/d$, implying that $q \pi_0 d /  |\{j: p_j \le q \}| \approx \alpha - \eps$. Thereby, condition \eqref{eqn:ebh-condition} will be satisfied with high probability, leading {\PoPCe} to reject the hypotheses $\{j : p_j \le q \}$. We will later show that this is further asymptotically equivalent to rejecting the hypotheses $\{j : u_j \le t(\alpha - \eps)\}$. Therefore, {\PoPCe} asymptotically rejects the same hypotheses as {\TPoP}, thereby possessing near-optimal power. 
% Denote $k_1 = |\{j: e_j=1/q\} |$. Note that $\hat k\in\{0,k_1\}$. As for optimality, since in our construction $e_j$ either equals $1/q$ or $0$, it follows that eBH either rejects all hypotheses with $p_j\leq q$ or rejects nothing. When it is the former case, our procedure rejects the same hypotheses as {\TPoP} asymptotically and hence is optimal.   It follows from the definition of $q$  that  $q\pi_0 d/k_1 \approx \FDR \approx \alpha-\eps<\alpha$ with probability converging to one. This implies our procedure rejects no hypothesis (i.e., $\hat k=0$) with probability converging to zero, and therefore is optimal under the Bayesian linear model. 

\item Hyperparameters ($K,\eps$): We remark that the choice of $(K, \eps)$ will not affect the validity of {\PoPCe}: irrespective of the chosen hyperparameters, {\PoPCe} ensures finite-sample {\FDR} control. However, this choice does impact the asymptotic optimality. In particular, $K$ is the number of times to re-sample the base statistics in {\CRT}. A larger $K$ brings the p-values closer to the uniform distribution under the null hypothesis, albeit at a higher computational cost. Furthermore, in Line \ref{line_1_1}, we choose $\eps>0$ to be a small number to calibrate {\FDP} closer to $\alpha$ and thereby attain a better power. Nevertheless, we do not want $\eps$ to be excessively small, ensuring that condition \eqref{eqn:ebh-condition} happens with high probability. 

%$\eps>0$ is some small quantity to ensure $\hat k \neq 0$ in Line \ref{line_1_12} with probability converging to one, and the {\mTPR} of {\PoPCe} is close to $\CPoP$. {\PoPCe} with a smaller $\eps$ has {\mTPR} closer to the optimal {\mTPR}  asymptotically, but is more likely to reject nothing under finite sample (i.e., the convergence to limiting {\mTPR} could be slower).
\end{itemize}

\begin{algorithm}
  \caption{The {\PoPCe} procedure}
   \label{alg:popce}
   \begin{algorithmic}[1]
      \REQUIRE $\{ (\bx_i, y_i) \}_{i \in [n]} = (\Xb,\Yb)$; {\FDR} level $\alpha \in (0, 1)$; distribution $\P_{\bX}$;  null proportion $\pi_0$; prior $\Pi$ and noise level $\sigma^2$; hyperparameters $K \in \N$, and $\eps > 0$. 
      \STATE\COMMENT {\blue {Compute the p-to-e calibration threshold}}
      \STATE \label{line_1_1} Compute $\tau_\star^2$ which solves Eq.~\eqref{eqn:potential_effective_noise} with prior $\Pi$ and noise level $\sigma^2$. Compute $t = \max\{  s \in [0, 1]: \lim_{d\to \infty, n/d \to \delta}\FDP( \bT_P(\cdot;s) ; \Pi) \le \alpha - \eps \}$ (c.f. Eq. \eqref{eqn:definition-t-alpha}). 
      \STATE \label{line_1_2} Compute $q=\Psi(t)$ where $\Psi$ is the CDF of $\PostProb(\tau_\star Z;\Pi,\tau_\star)$ when $Z\sim\cN(0,1)$.
      \FOR{$j \in [d]$}
     \STATE \label{line_1_5} \COMMENT {\blue {Conditional randomization test}}
      \STATE \label{line_1_4} Denote $P_j(\Yb, \Xb) = \P(\beta_{0, j}=0\vert\Yb,\Xb)$. Compute $u_j = P_j(\Yb, \Xb)$. 
      \FOR{$k \in [K]$}
      \STATE Sample $\tilde \xb_j^{(k)} = (\tilde x_{1j}^{(k)}, \ldots ,\tilde x_{nj}^{(k)})^\sT$ where $\tilde x_{ij}^{(k)} \sim \cL(X_j \vert  \bX_{-j} = \bx_{i, -j})$ independently. 
      \STATE Compute $u_j ^{(k)}= P_j(\Yb,\Xb_{-j},\tilde \xb_j^{(k)})$. 
      \ENDFOR\label{line_1_8}
      \STATE \label{line_1_9}Compute $p_j = (1/(K + 1)) ( 1 + \sum_{k = 1}^K \ones\{ u_j \ge u_j^{(k)} \})$. 
       \STATE\COMMENT {\blue {Compute e-values}}
      \STATE \label{line_1_10}Compute $e_j = \ones\{ p_j \le q \} / q$. 
      \ENDFOR
      \STATE\COMMENT {\blue {eBH}}
      \STATE \label{line_1_12}Reject the hypotheses with the $\hat k$ largest e-values, where 
	\[
	\hat k = \max \Big\{ k: \frac{\pi_0 d}{k e_{(k)}} \le \alpha \Big\}. 
	\]
   \end{algorithmic}
\end{algorithm}

\paragraph{Computational cost}

Line \ref{line_1_4}-\ref{line_1_9} of Algorithm \ref{alg:popce} necessitate the computation of local fdrs of the high dimensional Bayesian linear model. This process can be computationally demanding if Markov Chain Monte Carlo is utilized. In our numerical implementation, we opt to calculate local fdrs using the approximate message passing (AMP) algorithm \cite{donoho2009message} with subsequent post-processing. The AMP algorithm has been demonstrated to converge to the true posterior in several models \cite{mezard2009information,bayati2011dynamics,deshpande2015asymptotic,barbier2019optimal}. However, it is also worth noting that there are statistical models in which AMP does not reach the true posterior \cite{barbier2019optimal}. Despite this, the potential non-convergence of the AMP algorithm should not raise concern for the following reasons: (1) the finite-sample control of {\FDR} is valid for any base statistics and is thus applicable even if AMP does not converge to the correct posterior; (2) in our simulation configurations, AMP does indeed converge to the correct posterior, thereby achieving optimal power in these models.

%With this modification, we propose the {\PoEdCe} (Posterior Expectation + distilled conditional randomization test + eBH)  procedure (Algorithm \ref{alg:poedce}). The full algorithm for {\PoEdCe} is presented in Algorithm \ref{alg:poedce}. 

% In Line \ref{line_2_4}-\ref{line_2_10} we construct $s_j^{(k)} = \< \by - \bX_{-j} \what\bbeta_{-j}, \tilde \xb_j^{(k)}\>$ and let $u_j ^{(k)}= \PostProb( ( \tau^2 / \sigma^2) s_j^{(k)}; \Pi, \tau)$, where $\tilde\xb_j^{(k)}\in\R^{n},k=1,...,K$ are independent samples  of $\bx_j|\bX_{-j}$. ($s_j,u_j$ are constructed similarly.) 

%However, {\PoEdCe} is much more efficient than  {\PoPCe} since we only  compute the posterior expectation $\what \beta_{-j}$ for each coordinate once.

\paragraph{The {\PoEdCe} procedure}

In {\PoPCe}, the local fdr (posterior probability) needs to be computed for $(K+1) \times d$ times. To alleviate the computational burden, we propose a similar procedure {\PoEdCe} ({\tt P}osterior {\tt E}xpectation +  {\dCRT} + {\eBH}), where {\CRT} is replaced by {\dCRT} \cite{liu2022fast}. {\PoEdCe} only requires a single computation of the posterior expectation for each coordinate, and thereby reducing computational costs by a factor of $K + 1$ (assuming the computational costs for posterior expectation and posterior probability are equal). 

Specifically, we illustrate the difference between {\PoEdCe} and {\PoPCe} in Algorithm \ref{alg:poedce}: the only difference lies in the construction of the base statistics $\{ u_j \}_{j \in [d]}$ and their resampled version $\{ \{ u_j^{(k)}\}_{k \in [K]} \}_{j \in [d]}$ (see Line \ref{line_1_5}-\ref{line_1_8} of Algorithm \ref{alg:popce}). In {\PoEdCe}, the base statistics is taken to be $u_j = \PostProb( ( \tau_\star^2 / \sigma^2) s_j; \Pi, \tau_\star)$ for $s_j = \< \Yb - \Xb_{-j} \what \bbeta_{-j}, \xb_j\>$. Here, $\what \bbeta_{-j}$ represents the posterior expectation of $\btheta_0 \in \R^{d-1}$ given observation $(\Yb, \Xb_{-j})$, presuming the statistical model  $\Yb = \Xb_{-j} \btheta_{0} + \beps \in \R^n$, where $\theta_{0, i}\sim_{i.i.d.} \Pi$ and $\eps_i \sim_{i.i.d.} \cN(0, \sigma^2)$ (see Line \ref{line_2_4}-\ref{line_2_8}). The resampled version of the base statistics possesses a similar form (see Line \ref{line_2_9}-\ref{line_2_10}). Similar to {\PoPCe}, {\PoEdCe} also ensures control over frequentist {\FDR}. We will later demonstrate that the asymptotic distribution of $\{ u_j, \{ u_j^{(k)}\}_{k \in [K]} \}_{j \in [d]}$ in {\PoEdCe} aligns with those in {\PoPCe}, implying that {\PoEdCe} and {\PoPCe} possess approximately equal power.  

Similar to {\PoPCe}, in our numerical simulations, we employ the AMP algorithm to compute the posterior expectation of the Bayesian linear model (see Line \ref{line_2_4} of Algorithm \ref{alg:poedce}). By the same argument, the lack of a convergence guarantee for the AMP algorithm does not compromise the finite-sample control of {\FDR}.

\begin{algorithm}[]
  \caption{The {\PoEdCe} procedure (replacing Line \ref{line_1_5}-\ref{line_1_8} of Algorithm \ref{alg:popce} with the following)}
   \label{alg:poedce}
   \begin{algorithmic}[1]
       % \REQUIRE $\{(\bx_i, y_i)\}_{i \in [n]} = ( \Xb,\Yb)$; {\FDR} level $\alpha \in (0, 1)$; distribution $\P_{\bX}$; null proportion $\pi_0$; prior  $\Pi$ and noise level $\sigma^2$; hyperparameters $K \in \N$, and $\eps > 0$. 
       %  \STATE Line - in Algorithm \ref{alg:popce}
        \STATE\COMMENT {\blue {Distilled conditional randomization test}}
      \STATE \label{line_2_4}Compute $\what \bbeta_{-j}$, the posterior expectation of $\btheta_0 \in \R^{d-1}$ given observation $(\Yb, \Xb_{-j})$, assuming the statistical model $\Yb = \Xb_{-j} \btheta_{0} + \beps \in \R^n$, where $\theta_{0, i}\sim_{i.i.d.} \Pi$ and $\eps_i \sim_{i.i.d.} \cN(0, \sigma^2)$. 
      \STATE Compute $s_j = \< \Yb - \Xb_{-j} \what \bbeta_{-j}, \xb_j\>$. 
      \STATE \label{line_2_8} Compute $u_j = \PostProb( ( \tau_\star^2 / \sigma^2) s_j; \Pi, \tau_\star)$. 
      \FOR{$k \in [K]$}
      \STATE \label{line_2_9} Sample $\tilde \xb_j^{(k)} = (\tilde x_{1j}^{(k)}, \ldots ,\tilde x_{nj}^{(k)})^\sT$ where $\tilde x_{ij}^{(k)} \sim \cL(X_j \vert  \bX_{-j} = \bx_{i, -j})$ independently. 
      \STATE Compute $s_j^{(k)} = \< \Yb - \Xb_{-j} \what \bbeta_{-j}, \tilde \xb_j^{(k)}\>$. 
      \STATE \label{line_2_10}Compute $u_j ^{(k)}= \PostProb( ( \tau_\star^2 / \sigma^2) s_j^{(k)}; \Pi, \tau_\star)$. 
      \ENDFOR
      % \STATE Line - in Algorithm \ref{alg:popce}
   \end{algorithmic}
\end{algorithm}

\paragraph{Empirical Bayes for estimating the prior}

The implementation of {\PoPCe} and {\PoEdCe} presumes knowledge of the prior distribution $\Pi$ and the noise level $\sigma^2$. We also consider a situation where both the prior distribution $\Pi$ and the noise level $\sigma^2$ are unknown and, in response, propose an Empirical Bayes variant of {\PoEdCe}, named {\EPoEdCe}. We demonstrate that {\EPoEdCe} also controls {\FDR} from finite samples, and attains near-optimal power whenever the data are generated from a Bayesian linear model with unknown prior and noise levels. Detailed discussions about {\EPoEdCe} can be found in Appendix \ref{sec:EPoEdCe}.

\subsection{Frequentist validity and statistical optimality}\label{sec:FDR_optimality}

It is guaranteed that {\PoPCe} and {\PoEdCe} will control the frequentist {\FDR} from finite samples, as stated in Theorem \ref{thm:frequentist_FDR_control} below. The frequentist {\FDR} control is ensured by the {\eBH} procedure in Line \ref{line_1_12} of Algorithm \ref{alg:popce}, and the validity of the $p$-values of {\CRT} obtained in Line \ref{line_1_9} (see Appendix \ref{sec:freq_proof} for the detailed proof). 

\begin{theorem}[Frequentist {\FDR} control of {\PoPCe} and {\PoEdCe}]\label{thm:frequentist_FDR_control}
For any joint distribution $P\in\cM(P_{\bX})$ (c.f. Eq. \eqref{eqn:M-P-X}), suppose that $\{(\bx_i, y_i)\}_{i \in [n]}$ are i.i.d. from $P$. Let $\bT_\star$  be either {\PoPCe} or {\PoEdCe}. Then we have the frequentist {\FDR} control (c.f. Eq. \eqref{eqn:frequentist-fdr}) 
\[
\FDR(\bT_\star, P) \le \alpha.
\] 
\end{theorem}

%Next, we show that  {\PoPCe} and {\PoEdCe} control the frequentist FDR and have asymptotically optimal {\mTPR}.
% The design of {\PoEdCe} ensures its frequentist {\FDR} control, as stated in Theorem \ref{thm:frequentist_FDR_control_poedce}. Its proof is also straightforward by showing that each $p$-value obtained in Algorithm \ref{alg:poedce} is valid (see Appendix \ref{sec:freq_proof} for the detailed proof).
% \begin{theorem}[Frequentist {\FDR} control of {\PoEdCe}]\label{thm:frequentist_FDR_control_poedce}
% For any joint distribution $P\in\cM(P_{\bX})$ (c.f. Eq. \eqref{eqn:M-P-X}), suppose that $\{(\bx_i, y_i)\}_{i \in [n]}$ are i.i.d. from $P$, then the {\PoEdCe} procedure $\bT_\star$ (Algorithm \ref{alg:poedce}) controls the frequentist {\FDR} (c.f. Eq. \eqref{eqn:frequentist-fdr}),
% \[
% \FDR(\bT_\star, P) \le \alpha. 
% \] 
% \end{theorem}
% The proof of Theorem \ref{thm:frequentist_FDR_control} is straightforward after showing each $p$-values obtained in both procedures are valid. See Appendix \ref{sec:freq_proof} for a detailed proof. 

%We show asymptotic FDR optimality of two procedures in two steps. First we prove that they have frequentist FDR control under any joint distribution $P\in \cM(P_\bX)$ (c.f. Eq. (\ref{eqn:M-P-X})). Then

We subsequently introduce a conjecture proposing that {\PoPCe} and {\PoEdCe} are asymptotically optimal procedures with frequentist {\FDR} control (c.f. Definition \ref{def:optimal-frequentist-fdr}).  We will present a heuristic argument supporting this conjecture in Appendix \ref{sec:PoPCe_PoEdCe_power}, and we will validate the conjecture numerically in Section \ref{sec:numeric-frequentist}. %showing that when the model is well-specified, the asymptotic power of {\PoPCe} and {\PoEdCe} is the same as the power of {\TPoP} and {\CPoP}, when {\FDR} level is $\alpha$. 

\begin{conjecture}[Optimality of {\PoPCe} and {\PoEdCe}]\label{conj:PoPCe_PoEdCe_power}
In the asymptotic regime where $n,d \to \infty$, $n / d \to \delta$, $K = K_n \to \infty$, and $\eps = \eps_n \to 0$, and under the conditions of the Bayesian linear model as per Assumption \ref{ass:Bayesian_linear_model}, both {\PoPCe} and {\PoEdCe} have the same asymptotic power as {\CPoP} (for $\bT_\star$ to be either {\PoPCe} or {\PoEdCe})
\begin{equation}\label{eqn:mTPR-T-C}
\lim_{n \to \infty}\frac{1}{d} {\mTPR(\bT_\star, \Pi)}=\lim_{n \to \infty}\frac{1}{d} {\mTPR(\bC_P( \cdot;\lambda(\alpha)), \Pi)} .
\end{equation}
Subsequently, as per Proposition \ref{prop:Bayes_optimality}, both procedures are asymptotically {\BFDR} optimal (c.f. Definition \ref{def:BFDR-optimality}):
\[
\lim_{n \to \infty}\frac{1}{d} {\mTPR(\bT_\star, \Pi)}\geq \lim_{n \to \infty} \frac{1}{d}\max_\bT\Big\{ \mTPR(\bT, \Pi): \BFDR(\bT, \Pi) \le \alpha \Big\}.
\]
Consequently, according to Lemma \ref{lem:bound}, {\PoPCe} and {\PoEdCe} are both asymptotically $(\alpha, \cM(P_\bX), \Pi, o_n(1))$-optimal procedures with frequentist {\FDR} control (c.f. Definition \ref{def:optimal-frequentist-fdr}). 
\end{conjecture}

%\sm{Expand here. Giving more intuitions. }

% We present the intuitions of Conjecture \ref{conj:PoPCe_PoEdCe_power} in Appendix \ref{sec:PoPCe_PoEdCe_power}. 
%This Conjecture claims that both {\PoPCe} and {\PoEdCe} are FDR-optimal, meaning that they are procedures that control the {\FDR} below level $\alpha$ and achieve the maximal power under a known Bayesian linear model. 

Generally speaking, the conjecture builds on the intuition that the {\PoPCe} procedure (as well as {\PoEdCe}) will, in an asymptotic sense, reject the same set of hypotheses as {\CPoP}, when the model is well-specified. This intuition stems from the derivation of the asymptotic empirical distribution of local fdrs $\{ P_j(\cD) \}$ and $p$-values $\{ p_j \}$ in {\PoPCe} under the Bayesian linear model (c.f. Formalism \ref{fm:post_mean_zero}, \ref{fm:pvalue_permu} and \ref{fm:poedce2}). More specifically, in the limit as $n, p \to \infty$, we can ``marginally" treat
\[
\begin{aligned}
(\beta_{0, j}, P_j(\cD)) \stackrel{\cdot}{\sim}&~ (\Pi, \cP(\Pi+\tau_\star Z)), \\
(\beta_{0,j},p_j) \stackrel{\cdot}{\sim}&~ (\Pi,\Psi(\cP(\Pi +\tau_\star Z))),
\end{aligned}
\]
and we use a dot above the $\sim$ symbol to indicate that the approximation holds only in a restricted sense. Above, $Z \sim \cN(0,1)$, $\tau_\star$ is the constant determined through (\ref{eqn:potential_effective_noise}), $\cP(\, \cdot \,) = \PostProb(\, \cdot \,; \Pi, \tau_\star)$ is given by (\ref{eqn:posterior_probability_low}), and $\Psi(t) = \P_{Z \sim \cN(0, 1)}(\PostProb(\tau_\star Z) \le t)$ is the CDF of $\PostProb(\tau_\star Z)$, a strictly increasing function. 

Given that {\CPoP} and {\PoPCe} (as well as {\PoEdCe}), respectively, threshold $\{ P_j(\cD) \}$ and  $\{ p_j \}$ at some levels that are calibrated to control the asymptotic {\BFDR} level $\alpha$, the monotonicity of $\Psi$ dictates that both procedures asymptotically reject lower values of $\{P_j(\cD)\}$ at the same level, hence rejecting the same set of hypotheses (c.f. Appendix \ref{sec:PoPCe_PoEdCe_power} for details).

%\sm{1. The CRT p-values and the local fdrs (and ground truth) have asymptotically the same empirical distribution, up to a transformation (give equations); 2. The threshold is correctly calibrated. } %More specifically, in these two procedures we also have $(\beta_{0,j},p_j)\overset{\cdot}{\sim}(\Pi,\PostProb(\Pi+\tau Z;\Pi,\tau))$, which means $p_j$ are essentially the same as local fdr. Therefore,  by choosing a proper threshold we are able to  recover {\TPoP} and achieve optimality. 
%In addition, the reason to inverse $p_j$ and use e-values is to ensure finite sample FDR control. \sm{Add related work}

%Now, it follows immediately from Theorem~\ref{thm:frequentist_FDR_control} and Conjecture \ref{conj:PoPCe_PoEdCe_power} that {\PoPCe} and {\PoEdCe} are asymptotically $(\alpha, \mcP, \Pi,0)$-{\FDR} optimal with $\mcP$ being the set of all joint distribution of $X$ and $y$. 

\subsection{Numerical simulations}\label{sec:numeric-frequentist}

We perform numerical simulations, comparing the predicted and simulated {\FDP} and {\TPP} of {\PoPCe}, {\PoEdCe}, and {\EPoEdCe}. We first focus on well-specified Bayesian linear models and demonstrate that the simulated curves align with analytical predictions. Following this, we examine misspecified models to verify that these procedures maintain control over the frequentist FDR from finite samples. 

\subsubsection{{\FDP} and {\TPP} in well-specified models}

\begin{figure}\centering
\includegraphics[width=1\textwidth]{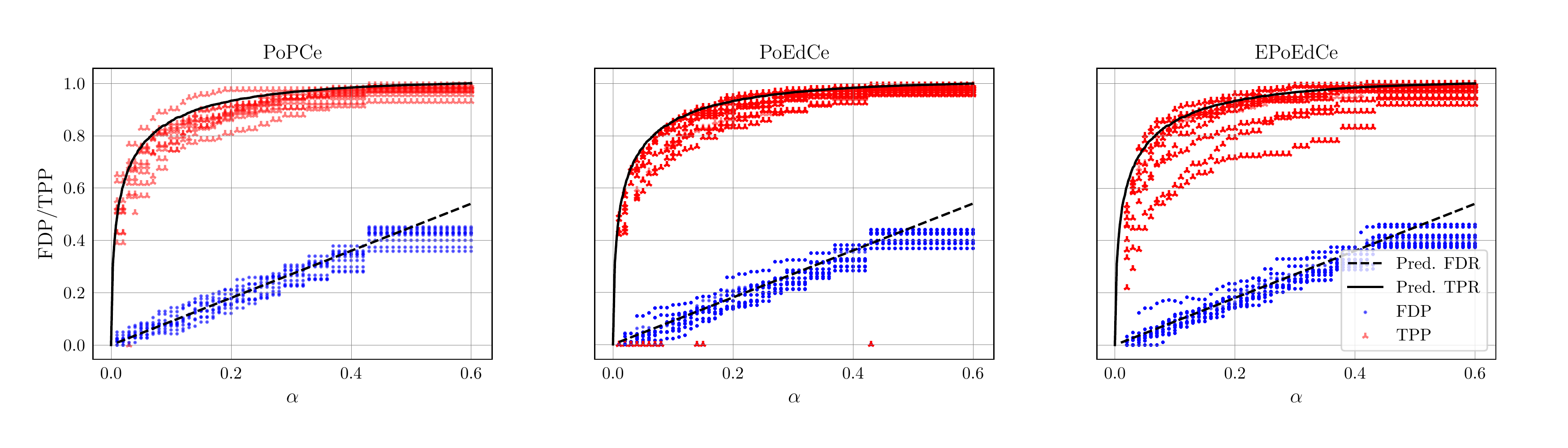}
\vskip-0.3cm
\caption{{\FDP, \TPP} versus $\alpha$ for {\PoPCe}, {\PoEdCe} and {\EPoEdCe}. Blue and red dots are $(\FDP,\alpha)$ and $(\TPP,\alpha)$. The real black curve is the analytical prediction of {\FDP-\TPP}, and the dashed black line is $y=0.9\alpha$, the desired {\FDR} level. For all testing procedures, we generate $10$ simulated instances with $\delta=1.25$, $n=500$, $\sigma=0.25$, and $\Pi=0.6\cdot\delta_0 +0.2\delta_1+0.2\delta_{-1}$. }
\label{fig:popce_epoedce}
\end{figure}

Figure \ref{fig:popce_epoedce} displays the realized {\FDP} and {\TPP} of {\PoPCe}, {\PoEdCe}, and {\EPoEdCe} against the nominal level $\alpha$, using data generated from a well-specified Bayesian linear model. We simulate $10$ instances of $(\Yb,\Xb)$ with $n=500$, $d=400$, $\delta=1.25$, $\sigma=0.25$, and prior distribution $\Pi=0.6\cdot\delta_0 +0.2\delta_1+0.2\delta_{-1}$. For each simulated instance, we apply each {\PoPCe}, {\PoEdCe} and {\EPoEdCe} across various {\FDR} control levels $\alpha$ from $0$ to $0.6$. For each $\alpha$, we select $\eps =0.1\alpha$, expecting that the {\FDP} will concentrate at level $\alpha-\eps=0.9\alpha$. We set the number of repetitions $K=1000$ for the {\CRT} sub-routine. For {\EPoEdCe}, we choose $M=50$, the number of blocks in the empirical Bayes procedure. The realized {\FDP} and {\TPP} against the level $\alpha$ for all three procedures are then plotted. The analytical prediction curve of {\TPP} and the line $y=0.9\alpha$ serving as the analytical prediction curve of {\FDP} are also included. As demonstrated in Figure \ref{fig:popce_epoedce}, the {\TPP} tightly concentrates around the optimal {\TPR} level, while the {\FDP} closely adheres to the predicted {\FDR} level of $0.9\alpha$, for all three procedures. 

\subsubsection{{\FDP} and {\TPP} of {\PoEdCe} with well-specified and misspecified models}

% \begin{figure}\centering

% \caption{Realized {\TPP} of {\PoEdCe} with the same setup as Figure \ref{fig:poedce_fdp}. 
% }
% \label{fig:poedce_tpp}
% \end{figure}

Figure \ref{fig:poedce_fdp} displays the realized {\FDP} and {\TPP} of {\PoEdCe} against the nominal level $\alpha$ under well-specified and misspecified models. We simulate $10$ instances of $(\Yb,\Xb)$ for $\delta=0.8,1.25$ (corresponding to $n=400,d=500$ and $n=500,d=400$ respectively), setting $\sigma=0.25$ and $\Pi=0.4\cdot\delta_0 + 0.5\cdot\delta_1 + 0.1\cdot\delta_{-1}$. For each simulated instance, we employ the {\PoEdCe} procedure under three different model assumptions (considering three sets of model parameters $(\Pi, \sigma)$ as inputs of {\PoEdCe}): (1) a well-specified model, where  {\PoEdCe} utilizes the true model parameters $(\Pi, \sigma)$ that generate the data; (2) a misspecified model with an incorrect noise level $\sigma =0.5$; (3)  a misspecified model with an inaccurate prior $\Pi=0.6\cdot\delta_0 +0.2\delta_1+0.2\delta_{-1}$. In all scenarios, we select $\eps=0.1\alpha$ and set the repetition number to $K=1000$.

The upper panel demonstrates that {\FDP} is controlled under level $0.9\alpha$, even in instances of model misspecification. It also illustrates that the {\FDP} concentrates around level $0.9\alpha$ when the model is well-specified. The lower panel shows that the {\TPP} concentrates on the analytical prediction when the model is well-specified, thereby aligning with our conjecture. However, in model instances with a misspecified prior or an erroneous noise level, the power of {\PoEdCe} falls below the optimal {\TPP}-{\FDP} tradeoff curve.

% We plot realized {\FDP} against level $\alpha$ in Figure \ref{fig:poedce_fdp}. 
%for three different model assumptions. with desired $\FDR$ control level $0.9\alpha$ against $\alpha$. 

\begin{figure}\centering
\includegraphics[width=0.77\textwidth]{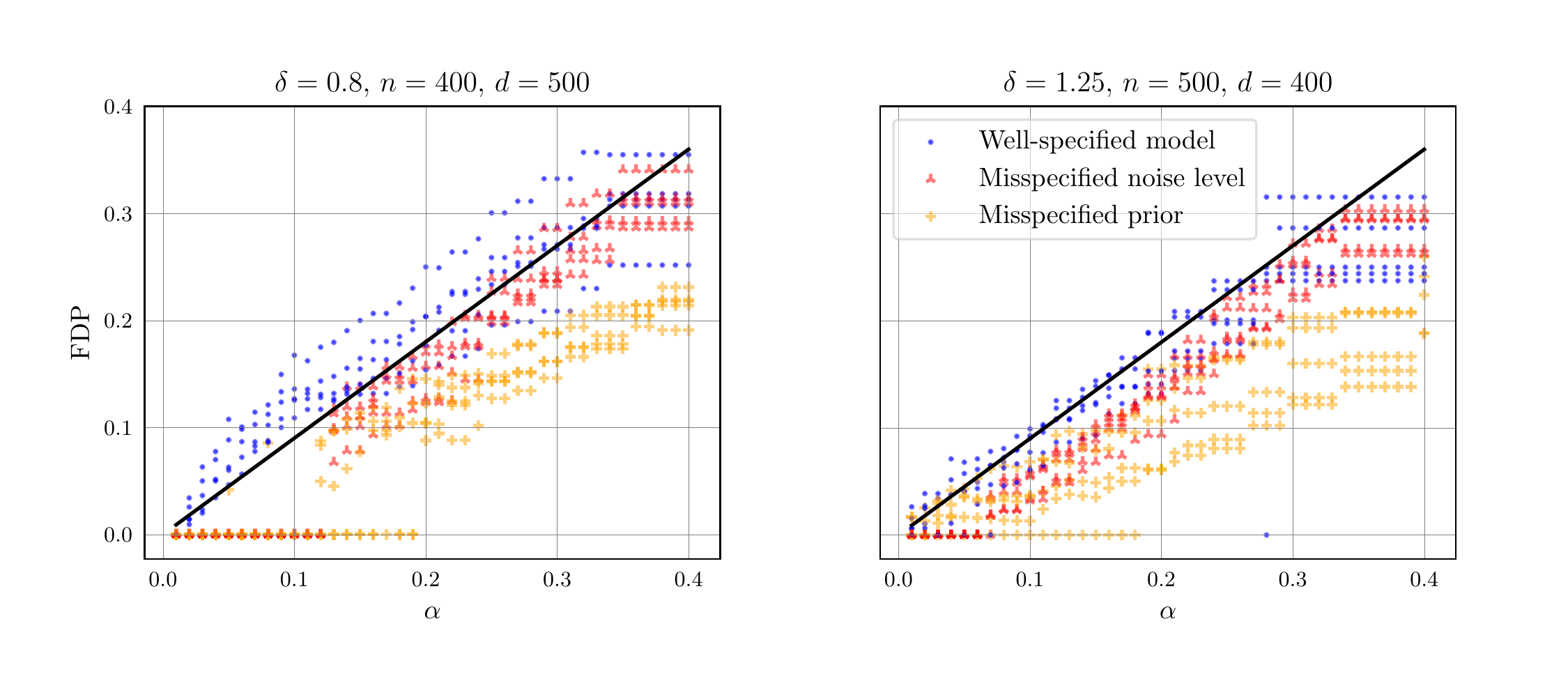}
\vskip-0.5cm
\includegraphics[width=0.77\textwidth]{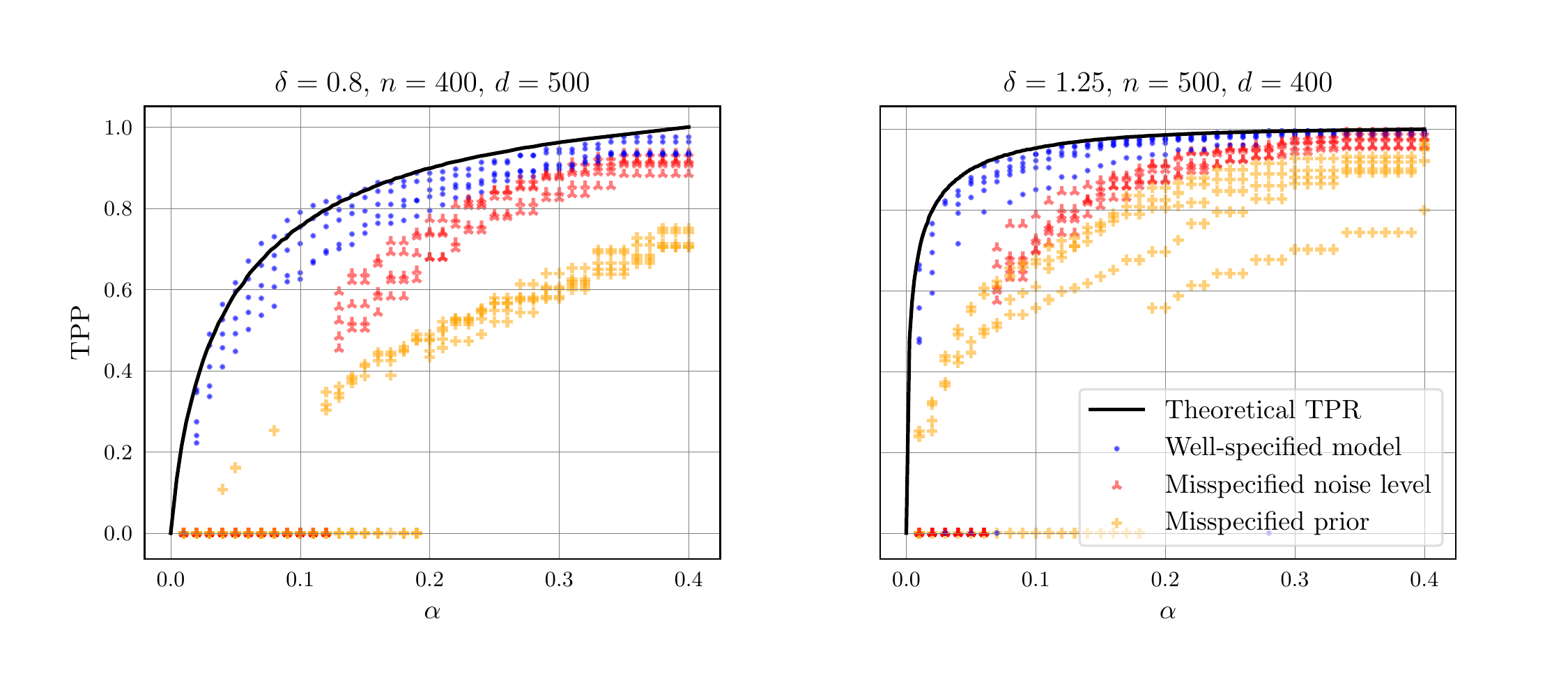}
\vskip-0.3cm
\caption{Realized {\FDP} and {\TPP} of {\PoEdCe} with two $\delta=0.8,1.25$ ($n=500,d=400$ and $n=400,d=500$ respectively). The Bayesian linear model is generated with model parameters $\sigma=0.25$ and $\Pi=0.4\cdot\delta_0 + 0.5\cdot\delta_1 + 0.1\cdot\delta_{-1}$. Hyperparameters are chosen to be $\epsilon=0.1\alpha$ and $K=1000$. The inputs of {\PoEdCe} are three sets of model parameters: well-specified model with correct $(\Pi, \sigma)$, misspecified model with noise level $\sigma=0.5$, and misspecified model with prior $\Pi=0.6\cdot\delta_0 +0.2\delta_1+0.2\delta_{-1}$. Upper panel: Blue, red and yellow dots are $(\FDP,\alpha)$ for different model parameters. The black line is the {\FDR} level $y = 0.9\alpha$. Lower panel: Blue, red and yellow dots are $(\TPP,\alpha)$ for different model parameters. The black curve is the analytical prediction of {\TPP}. }
\label{fig:poedce_fdp}
\end{figure}

\section{Conclusion and discussion}

In this paper, we proposed multiple testing procedures with frequentist {\FDR} control, which are also near-optimal under Bayesian linear models. We begin by calculating an upper bound of power, for any procedure with finite-sample {\FDR} control. This statistical limit is demonstrated as asymptotically achievable by two testing procedures, {\PoPCe} and {\PoEdCe}. These procedures control {\FDR} from finite samples under the Model-X framework and are conjectured to be near-optimal when the Bayesian linear model is well-specified. We provide the intuition behind these conjectures and employ numerical simulations to corroborate the validity and optimality of the proposed procedures. 

Our work establishes the Bayesian linear model as a reference point for power comparisons among various multiple testing procedures (for example, knockoffs \cite{candes2018panning}, mirror statistics \cite{xing2021controlling}, dBH \cite{fithian2020conditional}, etc). In other words, the effectiveness of a multiple testing procedure can be assessed relative to the power of {\PoEdCe} within the Bayesian linear model. On the other hand, we would like to emphasize that while the {\PoPCe} and {\PoEdCe} procedures serve as theoretical tools for the power analysis of {\FDR}-controlling procedures, we do not recommend their direct application in practical scenarios. Indeed, these procedures could possibly be powerless under model misspecification in practice.

This paper presents several important questions for further exploration. Firstly, our optimality conjecture for {\PoPCe} and {\PoEdCe} is based on heuristic calculations, and a significant challenge would be to formally prove this conjecture. An essential step towards this goal is the derivation of the asymptotics for the joint empirical distribution of the parameters and local fdrs. An approach to consider might be the application of advanced Gaussian interpolation techniques, such as those utilized in \cite{barbier2019optimal}. 

Furthermore, an intriguing question is the design of optimal procedures that extend beyond the model assumptions of isotropic Gaussian covariates and the Bayesian linear model. For example, a natural extension is Bayesian generalized linear models with anisotropic Gaussian covariates. Under this assumption, designing procedures with finite-sample {\FDR} control is a straightforward problem; the challenging question, however, is how to achieve near-optimality under well-specified models. 

Finally, we notice that when the model has certain misspecification, {\PoPCe} and {\PoEdCe} are often too conservative and do not reject any hypothesis. This outcome arises because we select the p-to-e calibrator to be a truncation function with an estimated truncation threshold, and {\eBH} might reject nothing if the estimated threshold is inaccurate. A compelling open question, therefore, is whether one can enhance or optimize the power under model misspecification while maintaining finite-sample validity and Bayes optimality. %We believe dependence adjusted {\BH} procedure (dBH) \cite{fithian2020conditional} may be less-conservative when the model is misspecified, but preserve the finite-sample validity and Bayes-optimality. 

\section*{Acknowledgement}

This project is supported by NSF grant DMS-2210827 and CCF-2315725. We thank Will Fithian and his group for helpful discussions.

\bibliographystyle{alpha}
\bibliography{arxiv_v2.bbl}

% \clearpage

% \tableofcontents

\clearpage

\tableofcontents

\appendix

\section{Proof of Proposition \ref{prop:marginal_optimality}}\label{sec:Bayes_optimality_proof}

Throughout the proof of this proposition, all the probability and expectation are with respect to the randomness in $\bbeta_0$ and $\cD$. We first prove the existence of $t = t(\alpha)$ such that Eq. (\ref{eqn:mFDR-alpha}) holds. 

Recall the definition of $\mFDR$ as in Eq. (\ref{eqn:def_mFDR_mTPR}). For any $t\in(0,1)$, note that we can rewrite $\mFDR(\bT_p(\cdot; t),\Pi)$ as
$$\mFDR(\bT_p(\cdot; t),\Pi)=\frac{\sum_{j\in[d]}\P(P_j(\cD)<t,\beta_{0,j}=0)}{\sum_{j\in[d]}\P(P_j(\cD)<t,\beta_{0,j}=0)+\P(P_j(\cD)<t,\beta_{0,j}\ne0)}.$$
By Assumption \ref{ass:marginal_optimality}, $\P(P_j(\cD)<t | \beta_{0,j} = 0)$ and $\P(P_j(\cD)<t | \beta_{0,j} \neq 0)$ are continuous in $t$, so that we have $\P(P_j(\cD)<t,\beta_{0,j} = 0)$ and $\P(P_j(\cD)<t,\beta_{0,j} \neq 0)$ are also continuous in $t$. Thus $\mFDR(\bT_p(\cdot;t),\Pi)$ is also continuous in $t$. Now for any $\alpha\in A = (\inf_t\mFDR(\bT_P( \cdot;t), \Pi),\sup_t\mFDR(\bT_P( \cdot;t), \Pi))$, by intermediate value theorem, there exists $t(\alpha)$ such that $\mFDR(\bT_P(\cdot;t(\alpha)),\Pi)=\alpha$, i.e., Eq. (\ref{eqn:mFDR-alpha}) holds.

Now we prove the second part of the proposition. First we define $p(\cD \vert \beta_{0,j} = 0)$ to be the conditional density of $\cD$ given $\beta_{0,j} = 0$ and $p(\cD \vert \beta_{0,j} \neq 0)$ to be the conditional density of $\cD$ given $\beta_{0,j} \neq 0$ as in Assumption \ref{ass:marginal_optimality}.  By the Bayes formula, we have 
\[
P_j(\cD) = \frac{p(\cD|\beta_{0,j}=0) \P(\beta_{0,j} = 0)}{p(\cD|\beta_{0,j}=0) \P(\beta_{0,j} = 0) + p(\cD|\beta_{0,j} \neq 0) \P(\beta_{0,j} \neq 0)}, 
\]
so that $P_j(\cD) < t$ is equivalent to 
\begin{equation}\label{eqn:equation3_in_proof_prop_1}
\begin{aligned}
p(\cD|\beta_{0,j}=0) \P(\beta_{0,j} = 0) \times (1-t) <  p(\cD|\beta_{0,j} \neq 0) \P(\beta_{0,j} \neq 0) \times  t. 
\end{aligned}
\end{equation}
For any test $\bT: \Omega \to \{0, 1\}^d$, we have
\begin{equation}\label{eqn:equation1_in_proof_prop_1}
\begin{aligned}
t\cdot\E[\TP(\bT)]-(1-t)\cdot\E[\FD(\bT)]=&\sum_{j=1}^d \Big( t \cdot \P(T_j=1,\beta_{0,j}\ne0)-(1-t)\cdot \P(T_j=1,\beta_{0,j}=0)\Big)\\
=& \sum_{j=1}^d  \int_{\Omega} T_j(\cD)\Big( t \cdot p(\cD|\beta_{0,j}\ne0) \P(\beta_{0,j} \neq 0)-(1-t) \cdot p(\cD|\beta_{0,j}=0)  \P(\beta_{0,j} = 0) \Big) \de \cD.
\end{aligned}
\end{equation}
Therefore, by the definition of $\bT_P(\cdot;t)$ as in Eq. (\ref{eqn:def_T_P}) and by Eq. (\ref{eqn:equation3_in_proof_prop_1}), we have that $\bT_P(\cdot;t)$ maximizes $t\cdot\E[\TP(\bT)]-(1-t)\cdot\E[\FD(\bT)]$, i.e., for any $\bT: \Omega \to \{0, 1\}^d$, we have
\begin{equation}\label{eqn:equation4_in_proof_prop_1}
t\cdot\E[\TP(\bT)]-(1-t)\cdot\E[\FD(\bT)] \le t\cdot\E[\TP(\bT_P(\cdot;t))]-(1-t)\cdot\E[\FD(\bT_P(\cdot;t))]. 
\end{equation}

We next show that $t = t (\alpha) > \alpha$. Define $T_*=\E[\TP(\bT_P(\cdot; t(\alpha)))]$ and $F_*=\E[\FD(\bT_P(\cdot; t(\alpha)))]$. Then by the definition of $\bT_P$ as in Eq. (\ref{eqn:def_T_P}) and by Eq. (\ref{eqn:equation1_in_proof_prop_1}), we have
\begin{align}
t\cdot T_* -(1-t)\cdot F_* \ge 0. 
\end{align} 
We first show that $t\cdot T_* -(1-t)\cdot F_* >0 $. We use proof by contradiction. Assume the contrary holds which gives $t\cdot T_* -(1-t)\cdot F_* = 0$. Then the integrand in Eq. (\ref{eqn:equation1_in_proof_prop_1}) is 0 almost everywhere for all $j\in[d]$. Thus, by the fact that $P_j(\cD)<t$ is equivalent to Eq. (\ref{eqn:equation3_in_proof_prop_1}), we have $\bT_P(\cdot;t)=0$ almost everywhere in $\cD$. Then $T_*=F_*=0$, and so that $\mFDR$ of $T_P(\cdot;t(\alpha))$ is equal to $0$, which contradicts the definition of $t(\alpha)$. As a consequence, we have 
\[
0 < (1-\alpha)(t\cdot T_*-(1-t)\cdot F_*) = (t-\alpha)T_*+(1-t)(\alpha(T_*+F_*)-F_*)=(t-\alpha)T_*,
\]
where the last equality uses the fact that $\mFDR(\bT_P(\cdot;t(\alpha)),\Pi)=\alpha$ so that $\alpha(T_*+F_*)-F_* = 0$. Since $T_* \ge 0$, we have from  the equation above that $t > \alpha$. 

Now for any test $\bT_0: \Omega \to \{ 0, 1\}^d$ with $\mFDR(\bT_0,\Pi)\le\alpha$, we define $T=\E[\TP(\bT_0)]$, $F=\E[\FD(\bT_0)]$. Since $\mFDR(\bT_0,\Pi)\le\alpha$, we have 
\begin{equation}\label{eqn:equation2_in_proof_prop_1}
\alpha(T+F)-F\ge0. 
\end{equation}
As a consequence, for $t=t(\alpha)$, we have
\begin{align*}
(t-\alpha)T\le&~ (t-\alpha)T+(1-t)(\alpha(T+F)-F)~~~~~~~~~~~~~~~~~~~~~~~~~~~ \text{ by } (\ref{eqn:equation2_in_proof_prop_1})\\
=&~(1-\alpha)(t\cdot T-(1-t)\cdot F)\\
\le&~(1-\alpha)(t\cdot T_*-(1-t)\cdot F_*)~~~~~~~~~~~~~~~~~~~~~~~~~~~~~~~~~~~ \text{ by } (\ref{eqn:equation4_in_proof_prop_1}) \\
=&~ (t-\alpha)T_*+(1-t)(\alpha(T_*+F_*)-F_*)\\
=&~(t-\alpha)T_*,~~~~~~~~~~~~~~~~~~~~~~~~~~~~~~~~~~~\text{ since } \mFDR(\bT_P(\cdot;t(\alpha)),\Pi)=\alpha.
\end{align*}
Since we have shown that $t - \alpha > 0$, it follows  from the equation above that $T \le T_*$. This further implies  $$\mTPR(\bT_0,\Pi)=T/\E_{\bbeta_0}[\#\{ j : j \not \in \Null\}]\le T_*/\E_{\bbeta_0}[\#\{ j : j \not \in \Null\}]=\mTPR(\bT_P(\cdot;t(\alpha)),\Pi),$$ which proves Eq. (\ref{eqn:mTPR-optimal}).

\section{Proof of Proposition \ref{prop:Bayes_optimality}}\label{sec:Bayes_optimality_proof_BFDR}

To prove the proposition, we start with the following lemmas. 
\begin{lemma}\label{lem:optimal_BFDR_BTPE}
For any fixed $\lambda > 0$, we define
\begin{equation}\label{eqn:equation1_lemma2}
\begin{aligned}
U(\bT, \lambda) \equiv \mTPR(\bT, \Pi) - \lambda \cdot \BFDR(\bT, \Pi). 
\end{aligned}
\end{equation}
Then $\bT=\bC_{P}(\, \cdot \,;\lambda)$ maximizes $U(\bT, \lambda)$ over $\bT:\Omega \to \{0, 1\}^d$, where the definition of $\bC_P$ is given by Eq. (\ref{eqn:def-Cpop}). 
\end{lemma}

\begin{proof}[Proof of Lemma \ref{lem:optimal_BFDR_BTPE}]
Recall the definition of $\bC_P$ in Eq. (\ref{eqn:def-Cpop}), $N$ in (\ref{eqn:def-N}), {\mTPR} in (\ref{eqn:def_BFDR}), and {\mTPR} in (\ref{eqn:def_mFDR_mTPR}). 
%\begin{equation}\label{eqn:equation2_lemma2}
%\begin{aligned}
%C_{P}(\lambda,\cD) = (\ones\{ P_j(\cD) \le P_{(\what K(\lambda,\cD))}(\cD) \})_{j \in [d]}, 
%\end{aligned}
%\end{equation}
%where
%\begin{equation}\label{eqn:equation3_lemma2}
%\begin{aligned}
%\what K(\lambda,\cD) = \arg\max_{K \in [d]} \Big( K / N - ( 1/N + \lambda / (K \vee 1) ) \sum_{j = 1}^K P_{(j)}(\cD) \Big). 
%\end{aligned}
%\end{equation}
Note that for  $\bT:\Omega\rightarrow\{0,1\}^d$, we have 
\[
\begin{aligned}
U(\bT, \lambda) =&~ \E_{\cD, \bbeta_0}\Big[\frac{\TP(\bT)}{N} - \lambda \frac{\FD(\bT)}{\Rej(\bT) \vee 1}\Big] \\
=&~  \E_{\cD, \bbeta_0}\Big[ \sum_{K = 0}^d \ones\{ \Rej(\bT) = K \} \Big( \frac{\TP(\bT)}{N} - \lambda \frac{\FD(\bT)}{K \vee 1} \Big) \Big] \\
=&~ \E_{\cD, \bbeta_0}\Big[ \sum_{K = 0}^d \sum_{j = 1}^d \ones\{ \Rej(\bT) = K, T_j = 1 \} \Big( \frac{\ones\{ j \not \in \Null\} }{N} - \lambda \frac{\ones\{ j \in \Null\}}{K \vee 1} \Big) \Big] \\
=&~ \E_{\cD}\Big[ \sum_{K = 0}^d \sum_{j = 1}^d \ones\{ \Rej(\bT) = K, T_j = 1 \} \Big( \frac{\P( j \not \in \Null  \vert \cD) }{N} - \lambda \frac{\P(j \in \Null \vert \cD) }{K \vee 1} \Big) \Big] \\
=&~ \E_{\cD}\Big[ \sum_{K = 0}^d \sum_{j = 1}^d \ones\{ \Rej(\bT) = K, T_j = 1 \} \Big( 1 / N - ( 1/N + \lambda / (K \vee 1) ) \P(j \in \Null \vert \cD) \Big) \Big]. \\
\end{aligned}
\]
Now, define $\overline U : \{ 0, 1\}^d \times \Omega \to \R$ as (for $\overline\bT\in \{0,1\}^d$ and $\cD \in \Omega$)
\begin{equation}\label{def_ubar}
\begin{aligned}
\overline U(\overline \bT, \cD) = \sum_{K = 0}^d \sum_{j = 1}^d \ones\{ \Rej(\overline \bT) = K, \overline T_j = 1 \} \Big( 1 / N - ( 1/N + \lambda / (K \vee 1) ) \P(j \in \Null \vert \cD) \Big),
\end{aligned}
\end{equation}
where $\Rej(\overline \bT)$ is the number of $1$ in $\overline \bT$. In order to maximize $U(\bT, \lambda)$ over $\bT  :\Omega \rightarrow \{0, 1\}^d$, we just need to maximize $\overline U(\overline \bT, \cD)$ over $\overline \bT \in \{0, 1\}^d$ for any fixed $\cD$, and set $\bT(\cD) = \overline \bT$. 

To do this, note that for any fixed $K$, we have 
\[
\max_{\overline \bT: \Rej(\overline \bT) = K} \sum_{j = 1}^d \ones\{\overline T_j = 1 \} \Big( 1 / N - ( 1/N + \lambda / (K \vee 1) ) P_j(\cD) \Big) =  \Big( K / N - ( 1/N + \lambda / (K \vee 1) ) \sum_{j = 1}^K P_{(j)}(\cD) \Big),
\]
where $\{ P_{(j)}(\cD) \}_{j \in [d]}$ are the order statistics of the local fdr $\{ P_j(\cD) \}_{j \in [d]}$ with $P_{(1)}(\cD) \le P_{(2)}(\cD) \le \cdots \le P_{(d)}(\cD)$. To maximize $\overline U(\overline \bT, \cD)$, we should select the number of rejection $K$ such that the right hand side of the equation above is maximized, which give rise to the $\what K$ as in Eq. (\ref{eqn:def-K}). This implies that $\bC_P$ attains the maximum of $U$. %So that $\overline \bT$ with corresponding $K$ attains the maximum. 
\end{proof}

\begin{lemma}\label{lem:argmax_lambda}
For functions $A,B:\{-1,+1\}^d\mapsto \R$ and $\lambda>0$, let $x_\lambda\in \argmax_{x\in \{-1,+1\}^d} \{A(x)-\lambda B(x)\}$. Then if $\lambda_1<\lambda_2$, we have $B(x_{\lambda_1})\ge B(x_{\lambda_2})$.
\end{lemma}

\begin{proof}[Proof of Lemma \ref{lem:argmax_lambda}]
For any $\lambda_1<\lambda_2$, by definition of $x_{\lambda_1},x_{\lambda_2}$
$$\lambda_2(B(x_{\lambda_1})-B(x_{\lambda_2}))\geq A(x_{\lambda_1})-A(x_{\lambda_2})\geq \lambda_1(B(x_{\lambda_1})-B(x_{\lambda_2})).$$ Thus $B(x_{\lambda_1})\geq B(x_{\lambda_2})$ since $\lambda_2>\lambda_1$.
\end{proof}

\begin{lemma}\label{bound_zero}
Let $Q$ be a probability measure on a measurable space $\cS$. Suppose that $f(x,\lambda)$ is a bounded function on $\cS\times \R$  such that (1) $f(x,\lambda)$ is a monotonic, right-continuous step function in $\lambda$ for any fixed $x \in \cS$; (2) $Q(\{x|f(x,\cdot) \text{\, is discontinuous at\,}\lambda\})=0$ for any $\lambda \in \R$. Then $\E_{x \sim Q} f(x,\lambda)$ is continuous in $\lambda$.
\end{lemma}

\begin{proof}[Proof of Lemma \ref{bound_zero}] 
Without loss of generality, we assume that $f(x,\lambda)$ is non-decreasing in $\lambda$ and $f(x,0)=0$ for all $x\in\cS$. Denote  the set of discontinuous points of $f(x,\cdot)$ by $D_x$ and let  $V_x: \cS \mapsto \R$ be  $V_x(a)\equiv \lim_{\lambda\to a+}f(x,\lambda)-\lim_{\lambda\to a-}f(x,\lambda)$.  Since $f(x,\cdot)$ is monotonic, the left and right limits exist and hence $V_x$ is well-defined. Note that $V_x(a)$ is non-negative and $V_x(a)>0$ iff $f(x,\cdot)$ is discontinuous at $a$. For any $x\in\cS$, we also define a measure  $\Lambda_x((-\infty,\lambda_0])\equiv\sum_{\lambda\leq\lambda_0} V_x(\lambda)$. We further define a finite measure $\tilde Q$ on $\cS\times \R $ by $\tilde Q(A,B)\equiv \int_{x\in A} \int_{\lambda\in B}  \de \Lambda_x \de Q$.

 Now we prove the lemma. By definition, it suffices to show $\lim_{\lambda\to\lambda_0\pm} \E_{x \sim Q}f(x,\lambda)=\E_{x \sim Q}f(x,\lambda_0)$ for all $\lambda_0\in\R$. For the right limit, we have
 \begin{align*}
 \lim_{\lambda\to\lambda_0+} \E_{x \sim Q}f(x,\lambda)-\E_{x \sim Q}f(x,\lambda_0)&=\lim_{\eps\to 0+}\int_{x \in \cS} [f(x,\lambda+\eps)-f(x,\lambda) ]\de Q\\
 &=\lim_{\eps\to 0+}\int_{x \in \cS, \lambda\in \R} \ones\{\lambda\in(\lambda_0,\lambda_0+\eps]\} \de \tilde Q=0,
 \end{align*} where the last convergence equality comes from the dominated convergence theorem. For the left limit, 
\begin{align}
 \lim_{\lambda\to\lambda_0-} \E_{x \sim Q}f(x,\lambda)-\E_{x \sim Q}f(x,\lambda_0)&=\lim_{\eps\to 0+}\int_{\cS}[f(x,\lambda-\eps)-f(x,\lambda) ] \de Q\nonumber\\
 &=-\lim_{\eps\to 0+}\int_{x \in \cS,\lambda\in \R} \ones\{\lambda\in(\lambda_0-\eps,\lambda_0]\} \de \tilde Q\nonumber\\&=-\int_{x \in \cS,\lambda\in \R} \ones\{\lambda=\lambda_0\} \de \tilde Q\nonumber\\
 &=-\int_{\cS} V_x(\lambda_0) \de Q \label{measure_fubini},
 \end{align}
 where the third equality uses the dominated convergence theorem and the last equality follows from the definition of $\tilde Q$ and $\Lambda_x$ and Fubini's theorem. 
Since we assume $f(x,\lambda)$ is bounded and the set of $x$ at which $f(x,\cdot)$ is discontinuous at $\lambda$ has zero measure, $0\leq \int_{x \sim Q} V_x(\lambda) \de Q\leq 2\sup_{x,\lambda}|f (x, \lambda) |\cdot Q(\{x|f(x,\cdot) \text{\, is discontinuous at\,}\lambda\})=0$ for all $\lambda$. Thus \eqref{measure_fubini} equals zero and it completes the proof.
\end{proof}

\begin{proof}[Proof of Proposition \ref{prop:Bayes_optimality}]
We first prove that for any $\alpha \in A$, there exists $\lambda = \lambda(\alpha)$ such that 
$
\BFDR(\bC_P( \cdot;\lambda(\alpha)), \Pi) = \alpha
$. It suffices to show  that  $\BFDR(\bC_P(\cD;\lambda),\Pi)$ is continuous in $\lambda$. For any fixed $\cD$, we have $\bC_P(\cD;\lambda) \in \argmax_{\overline\bT}\overline U(\overline \bT, \cD)$ as shown in the proof of Lemma \ref{lem:optimal_BFDR_BTPE} and note that $\max_{\overline\bT}\overline U(\overline \bT, \cD)=\max_{\overline\bT}\Big[\E_{\bbeta_0|\cD}\frac{\TP(\overline\bT)}{N} - \E_{\bbeta_0|\cD}\lambda \frac{\FD(\overline\bT)}{\Rej(\overline\bT) \vee 1}\Big] $. Then it follows from Lemma \ref{lem:argmax_lambda} that $\E_{\bbeta_0|\cD} \FDP(\bC_P(\cD,\lambda))$ is non-increasing in $\lambda$. Moreover, in the proof of Lemma \ref{lem:optimal_BFDR_BTPE} we have shown that  
$\E_{\bbeta_0|\cD} \FDP(\bC_P(\cD;\lambda))=\sum_{j = 1}^{\hat{K}(\lambda,\cD)} P_{(j)}/{ (\hat K(\lambda,\cD) \vee 1)}.$ 
Therefore $\hat K(\lambda,\cD)$ is non-increasing in $\lambda$ and $\E_{\bbeta_0|\cD} \FDP(\bC_P(\cD;\lambda))$ has at most $d$ discontinuous points as a non-increasing function in $\lambda$.  For any $\lambda>0$, define 
$$M(\lambda)\equiv\{ \cD|\lambda \text{\phantom{ }is a discontinuous point of\phantom{ }} \E_{\bbeta_0|\cD} \FDP(\bC_P(\cD;\lambda))\}.$$ 
%Note that any discontinuous point $\lambda$ can be viewed as a function of $(P_1(\cD),...,P_d(\cD))$. 
Note that for fixed $\cD$, $\lambda$ being a discontinuous point of $\E_{\bbeta_0|\cD} \FDP(\bC_P(\cD;\lambda))$ implies that $\what{K}(\lambda,\cD)$ is discontinuous at $\lambda$. Thus $\cD\in M(\lambda)$ implies that $f_{K,\lambda}(\cD)\equiv K / N - ( 1/N + \lambda / (K \vee 1) ) \sum_{j = 1}^K P_{(j)}(\cD)$ are equal for some two different $K$'s. Therefore $\{(P_1(\cD),...,P_d(\cD)):\cD\in M(\lambda) \}\subseteq \bigcup_{1\le K_1<K_2\le d} \{(P_1(\cD),...,P_d(\cD)): f_{K_1,\lambda}(\cD)=f_{K_2,\lambda}(\cD)\}$, which is a finite union of solutions of linear equation, is a measure zero set with respect to the Lebesgue measure. Let $\tilde P(\cdot)$ be the induced probability measure of $(P_1(\cD),...,P_d(\cD))$ when $\cD \sim P$. Then we have $P(M(\lambda))=\tilde P(\{(P_1(\cD),...,P_d(\cD)): \cD\in M(\lambda) \})=0$ following directly from Assumption \ref{ass:Bayes_optimality}. Applying Lemma \ref{bound_zero} to $\E_{\bbeta_0|\cD} \FDP(\bC_P(\cD;\lambda))$ then gives the desired result that $\BFDR(\bC_P(\cD;\lambda))$ is continuous in $\lambda$. This proves the existence of $\lambda(\alpha)$ satisfying Eq. \eqref{eqn:BFDR-alpha}. 

Now we prove the second part of the proposition: given $\bT$ with $\BFDR(\bT,\Pi)\leq\alpha$, we show Eq. (\ref{eqn:BFDR_optimal}). Note that
\begin{align*}
\mTPR(\bT,\Pi)&=U(\bT,\lambda)+\lambda\cdot \BFDR(\bT,\Pi)\\&\le
U(\bT,\lambda)+\lambda\alpha \\&\le
U(\bC_P(\cdot;\lambda(\alpha)),\lambda)+\lambda \cdot \BFDR(\bC_P(\cdot;\lambda(\alpha)),\Pi)~~~~~~~~~~~~~ \text{ by Lemma } \ref{lem:optimal_BFDR_BTPE}
\\&=
\mTPR(\bC_P(\cdot;\lambda(\alpha)),\Pi).
\end{align*}
This proves Eq. (\ref{eqn:BFDR_optimal}) and hence concludes the proof.
\end{proof}

\section{The empirical Bayes variant: the {\EPoEdCe} procedure}\label{sec:EPoEdCe}

We have shown that {\PoPCe} and {\PoEdCe} control frequentist {\FDR} from finite-samples, and attain near-optimal power when the data are generated from a Bayesian linear model with a known prior $\Pi$ and a known noise level $\sigma^2$. In this section, we consider the setting when we do not know the prior and the noise level, and propose Empirical Bayes {\PoEdCe} (\EPoEdCe). {\EPoEdCe} also controls {\FDR} from finite-samples, and attains near-optimal power whenever the data are generated from a Bayesian linear model with unknown prior and noise level. The full algorithm is presented in Algorithm \ref{alg:epoedce}.

At a high level, {\EPoEdCe} first estimates the prior and the noise level using nonparametric methods, and then applies the {\PoEdCe} procedure. To ensure that the computed p-values are valid under the null hypothesis, we need to use a covariate-splitting method to estimate the prior and the noise level. In the following, we give a line-by-line description of {\EPoEdCe} (Algorithm \ref{alg:epoedce}): 
\begin{itemize}
\item Line \ref{algo_3_1}-\ref{algo_3_3} (Split the covariates and estimate the prior and the noise level): We first split the indices of covariates into $M$ equal-sized blocks $C_m \subseteq [d], m\in[M]$. For each $m \in [M]$, we use the dataset $\{((x_{ij})_{j \not\in C_m}, y_i) \}_{i \in [n]}$ to estimate the prior and the noise level. This can be done using standard nonparametric procedures, e.g., nonparametric maximum likelihood estimate (NPMLE) in \cite{robbins1950generalization,kiefer1956consistency}. The details of our numerical implementation of this part are presented in Appendix \ref{sec:empirical_bayes_estimation_prior}. %is given in Appendix \ref{sec:empirical_bayes_estimation_prior}.

\item Line \ref{algo_3_4}-\ref{algo_3_5} (Compute the p-to-e calibration threshold using the estimated prior and noise level): We compute the p-to-e calibration threshold $q_m$ for each block $m \in [M]$, using the same approach as {\PoPCe} (Line \ref{line_1_1}-\ref{line_1_2} of Algorithm \ref{alg:popce}). In Line \ref{algo_3_4}, we take
\begin{equation}\label{eqn:overline-FDR-algorithm}
\overline \FDR(s; \Pi, \tau) = \P_{(\beta_0, G) \sim \Pi \times \cN(0, 1)}( \beta_0 = 0 | \PostProb(\beta_0 + \tau G; \Pi, \tau) < s),
\end{equation}
where $\PostProb(y; \Pi, \tau) = \P_{(\beta_0, Z) \sim \Pi \times \cN(0, 1)}( \beta_0 = 0 \vert \beta_0 + \tau Z = y )$ (as defined in Eq. \ref{eqn:posterior_probability_low}).
\item Line \ref{algo_3_9}-\ref{algo_3_19} (Apply {\PoEdCe} using the estimated prior and noise level): For each coordinate $j$, compute the e-value $e_j$ using the same steps as in {\PoEdCe}, with $(\Pi,\sigma^2)$ replaced by $(\what \Pi_{-m},\what\sigma_{-m}^2)$, and threshold $q$ replaced by $q_m$. Here $m$ is the block that the coordinate $j$ belongs to. 
% \item Line \ref{algo_3_19} (eBH): Apply the eBH procedure to the e-values obtained in Line \ref{algo_3_17}. 

\item Hyperparameter $M$: We remark that the choice of  $M$ will not affect the validity of {\EPoEdCe}: for any choice of $M$, {\EPoEdCe} has frequentist {\FDR} control. The choice of $M$ will also have a small effect on the asymptotic power as long as the estimated prior and noise level are consistent. In numerical simulations, we choose $M = 50$.   %$M$ is the number of blocks. We use $M-1$ folds to estimate the prior (noise level) and perform {\PoEdCe} on the remaining fold using the  estimated prior (noise level).
    \end{itemize}

\begin{algorithm}
  \caption{The {\EPoEdCe} procedure}
  
   \begin{algorithmic}[1]
   \label{alg:epoedce}
      \REQUIRE $\{(\bx_i, y_i)\}_{i \in [n]} = (\Yb, \Xb)$; {\FDR} level $\alpha \in (0, 1)$; distribution $\P_{\bX}$; null proportion  $\pi_0$; hyperparameters $K, M \in \N$, and $\eps > 0$. 
       \STATE\COMMENT {\blue {Split the covariates and estimate the prior and noise level}}
      \STATE \label{algo_3_1}Partition $[d]$ into equal-sized blocks $\{ C_m\}_{m \in [M]}$. Let $\iota: [d] \to [M]$ with $\iota(j) = m$ iff $j \in C_m$. 
      \FOR{$m \in [M]$}
      \STATE \label{algo_3_3} Estimate the prior $\what \Pi_{-m}$ and noise level $\hat \sigma_{-m}^2$ using the dataset $\{((x_{ij})_{j \not\in C_m}, y_i)\}_{i \in [n]}$. 
        \STATE\COMMENT {\blue {Compute the p-to-e calibration threshold using the estimated prior and noise level}}
      \STATE \label{algo_3_4}Compute $\hat \tau_{-m}^2$ which solves the self-consistent equation (\ref{eqn:self_consistent}) with prior $\Pi = \what \Pi_{-m}$, noise level $\sigma^2 = \hat \sigma_{-m}^2$, and $\delta$ re-scaled to be $\delta m / (m-1)$. Compute $t_m = \max\{  s \in [0, 1]: \overline \FDR( s ; \what \Pi_{-m}, \hat \tau_{-m}^2) \le \alpha - \eps \}$, where $\overline \FDR$ is as defined in Eq \eqref{eqn:overline-FDR-algorithm}. 
     \STATE\label{algo_3_5} Compute $q_m=\Psi(t_m)$ where $\Psi$ is the CDF of $\PostProb(\hat\tau_{-m} Z;\what \Pi_{-m},\hat\tau_{-m})$ when $Z\sim\cN(0,1)$. 
      \ENDFOR
    \STATE\COMMENT {\blue {Apply {\PoEdCe} using the estimated prior and noise level}}
      \FOR{$j \in [d]$}\label{algo_3_9}
      \STATE Compute $\what \bbeta_{-j}$, the posterior expectation of $\btheta_0 \in \R^{d-1}$ given observation $(\Yb, \Xb_{-j})$, assuming the statistical model $\Yb = \Xb_{-j} \btheta_{0} + \beps \in \R^n$, where $\theta_{0, l}\sim_{i.i.d.} \what \Pi_{-\iota(j)}$ and $\eps_i \sim_{i.i.d.} \cN(0, \hat \sigma_{- \iota(j)}^2)$. 
      \STATE\label{algo_3_s_j} Compute $s_j = \< \Yb - \Xb_{-j} \what \bbeta_{-j}, \xb_j\>$. 
      \STATE Compute $u_j = \PostProb( ( \hat \tau_{-\iota(j)}^2 /\hat \sigma_{-\iota(j)}^2) s_j; \what \Pi_{-\iota(j)}, \hat \tau_{-\iota(j)})$. 
      \FOR{$k \in [K]$}
      \STATE Sample $\tilde \xb_j^{(k)} = (\tilde x_{1j}^{(k)}, \ldots ,\tilde x_{nj}^{(k)})^\sT$ where $\tilde x_{ij}^{(k)} \sim \cL(X_j \vert  \bX_{-j} = \bx_{i, -j})$ independently. 
      \STATE Compute $s_j^{(k)} = \< \Yb - \Xb_{-j} \what \bbeta_{-j}, \tilde \xb_j^{(k)}\>$. 
      \STATE Compute $u_j ^{(k)}= \PostProb( ( \hat \tau_{-\iota(j)}^2 /\hat \sigma_{-\iota(j)}^2) s_j^{(k)}; \what \Pi_{-\iota(j)}, \hat \tau_{-\iota(j)})$. 
      \ENDFOR
      \STATE Compute $p_j = (1/(K + 1)) ( 1 + \sum_{k = 1}^K \ones\{ u_j \ge u_j^{(k)} \})$. 
      \STATE \label{algo_3_17}Compute $e_j = \ones\{ p_j \le q_{\iota(j)} \} / q_{\iota(j)}$. 
      \ENDFOR
        % \STATE\COMMENT {\blue {eBH}}
      \STATE \label{algo_3_19}Reject the hypotheses with the $\hat k$ largest e-values, where 
	\[
	\hat k = \max \Big\{ k: \frac{\pi_0 d}{k e_{(k)}} \le \alpha \Big\}. 
	\]
   \end{algorithmic}
\end{algorithm}

The covariate-splitting method for estimating the prior ensures that for each null coordinate $j$, $\what \bbeta_{-j}$ only depends on $(\Yb, \Xb_{-j})$. Therefore,  $(\Yb - \Xb_{-j} \what \bbeta_{-j})$ and $\xb_j$ are independent. This further ensures that $p_j$ is a valid p-value, and hence ensures the validity of {\EPoEdCe}. This gives the following theorem with proof in Section \ref{sec:freq_proof}. 
%Regarding finite sample {\FDR} control, since for each coordinate $j$ we use independent data for estimating $\Pi,\sigma^2$ and computing $e_j$, it can be verified that the conditional randomization test argument still holds and {\FDR} control is guaranteed. To summarize,  $\EPoEdCe$ also has frequentist FDR control and achieves   asymptotically optimal {\mTPR} when the data come from a Bayesian linear model (see Theorem \ref{thm:frequentist_FDR_control_EB} and  Conjecture \ref{conj:EPoPCe_EPoEdCe_power} for details).

\begin{theorem}[Frequentist {\FDR} control of {\EPoEdCe}]\label{thm:frequentist_FDR_control_EB}
For any joint distribution $P \in \cM(P_{\bX})$ (c.f. Eq. \ref{eqn:M-P-X}), suppose that $\{(\bx_i, y_i)\}_{i \in [n]}$ are i.i.d. from $P$, then the {\EPoEdCe} procedure $\bT_\star$ (Algorithm \ref{alg:epoedce}) controls the frequentist {\FDR},
\[
\FDR(\bT_\star, P) \le \alpha. 
\] 
\end{theorem}

If $(\what \Pi_{-m},\what\sigma_{-m}^2)$ is a consistent estimator of $(\Pi,\sigma^2)$ (e.g., we believe that the NPMLE estimator is consistent \cite{zhong2022empirical}), $\EPoEdCe$ will have asymptotically the same power as {\PoEdCe} and hence is also near-optimal. Indeed, we have the following conjecture for the asymptotic optimality of {\EPoEdCe}, which follows from Conjecture \ref{conj:PoPCe_PoEdCe_power} and the continuity of $(t_m, \what\beta_j,\hat\tau_{-\iota(j)},\cP)$ (and hence $(s_j, u_j, p_j, e_j)$) as functions of $(\Pi, \sigma^2)$. 

\begin{conjecture}[Optimality of {\EPoEdCe}]\label{conj:EPoPCe_EPoEdCe_power}
Consider the asymptotic regime $n,d \to \infty$, $n / d \to \delta$, $K = K_n \to \infty$, and $\eps = \eps_n \to 0$ slow enough. If the estimator $(\what\Pi_{-m},\hat\sigma^2_{-m})$ is consistent, namely $\what\Pi_{-m}$ converges weakly to $\Pi$ in probability and $\hat\sigma^2_{-m}$ converges to $\sigma^2$ in probability for all $m \in [M]$, then under the conditions of the Bayesian linear model as per Assumption \ref{ass:Bayesian_linear_model}, {\EPoEdCe} has the same asymptotic power as {\CPoP} (for $\bT_\star$ to be {\EPoEdCe})
\[
\lim_{n \to \infty}\frac{1}{d} {\mTPR(\bT_\star, \Pi)}=\lim_{n \to \infty}\frac{1}{d} {\mTPR(\bC_P( \cdot;\lambda(\alpha)), \Pi)} .
\]
Subsequently, as per Proposition \ref{prop:Bayes_optimality}, the {\EPoEdCe} procedure is asymptotically {\BFDR} optimal (c.f. Definition \ref{def:BFDR-optimality}):
\[
\lim_{n \to \infty}\frac{1}{d} {\mTPR(\bT_\star, \Pi)}\geq \lim_{n \to \infty} \frac{1}{d}\max_\bT\Big\{ \mTPR(\bT, \Pi): \BFDR(\bT, \Pi) \le \alpha \Big\}.
\]
Consequently, according to Lemma \ref{lem:bound}, {\EPoEdCe} is asymptotically $(\alpha, \cM(P_\bX), \Pi, o_n(1))$-optimal procedure with frequentist {\FDR} control (c.f. Definition \ref{def:optimal-frequentist-fdr}). 
\end{conjecture}

% \begin{proof}[Justification of Conjecture \ref{conj:EPoPCe_EPoEdCe_power}]
% This conjecture follows from Conjecture \ref{conj:PoPCe_PoEdCe_power} and the continuity of $t_m,\what \beta_j,\hat\tau_{-\iota(j)},\cP$ (and hence $s_j$, $u_j$, $p_j$, $e_j$) as functions of $\Pi$ and $\sigma$.
% \end{proof}

\subsection{Estimation of the prior $\Pi$ in {\EPoEdCe}}\label{sec:empirical_bayes_estimation_prior}

The log-likelihood function of observing $(\Xb, \Yb)$ given prior $\Pi \in \ProbFam(\R)$ and noise level $\sigma^2 \in \R$ is given by 
\[
\log p_{(\Pi, \sigma^2)}( \Yb,\Xb) = \log \int_{\R^d} \frac{1}{(\sqrt{2 \pi \sigma^2})^d} \exp\Big\{ - \frac{\| \Yb - \Xb \bbeta_0\|_2^2}{ 2 \sigma^2} \Big\} \Pi(\de \bbeta_0). 
\]
In principle, we can jointly estimate $(\Pi, \sigma^2)$ using nonparametric maximum likelihood estimate (NPMLE; \cite{robbins1950generalization,kiefer1956consistency}). However, in our numerical implementation of {\EPoEdCe}, due to the heavy computational burden of NPMLE, we consider a simpler parametric setting instead. In particular, we assume that the noise level $\sigma^2$ is known and the true prior  $\Pi$ is a three point distribution supported on $\{-1,0,1\}$, i.e., $\Pi= \pi_0 \delta_0 + \pi_1 \delta_1 + (1 - \pi_0 - \pi_1) \delta_{-1}$. We further consider the setting when the null proportion $\pi_0$ is given, and we only estimate a single parameter $\pi_1$. 

In order to estimate $\pi_1$, we further use a heuristic method as following. First, we choose some $\lambda>0$ and compute the ridge regression estimator
\begin{equation}\label{eqn:ridge-formula}
\what\bbeta_{\ridge}=(\Xb^\top\Xb+2\lambda \id_d)^{-1}\Xb^\top\Yb.
\end{equation}
Under Assumption \ref{ass:Bayesian_linear_model} (the bayesian linear model), in the limit of $n, d \to \infty, n / d \to \delta$, the empirical distribution of entries of $\what\bbeta_{\ridge}$ satisfies that for any sufficiently smooth function $\psi$, we have
 \begin{align*}
\lim_{n, d \to \infty, n / d \to \delta} \frac{1}{d}\sum_{j= 1}^d \psi(v \cdot \what\beta_{\ridge,j}) = \lim_{n, d \to \infty, n / d \to \delta} \frac{1}{d}\sum_{j= 1}^d \psi(u_j),
 \end{align*} 
 where $u_j = \beta_{0, j} + \tau G_j$ with $(\beta_{0, j}, G_j)_{j \in [d]} \sim_{i.i.d.} \Pi \times \cN(0, 1)$, %where $\beta_{0,j}$ is the $j$-th entry of the true parameter vector $\bbeta_0$, $G_j$ are i.i.d. standard Gaussian, 
and $\tau=\tau(\delta,\sigma^2,\lambda), v=v(\delta,\sigma^2,\lambda)$ are the unique solution to the equations  
% \begin{align*}
% \delta(v-1-2\lambda)v&=v-1\\
% \delta(\tau^2-\sigma^2)&=\frac{\tau^2}{v^2}+\delta^2(v-1-2\lambda)^2\rho,
% \end{align*} 
%with $\rho=\E_{X\sim\Pi}X^2=1-\pi_0$.  
\begin{equation}\label{eqn:ridge_self-consistent}
 \begin{aligned}
 \delta(v-1-2\lambda)v&=v-1,\\
 \delta(\tau^2-\sigma^2)&=\frac{\tau^2}{v^2}+\delta^2(v-1-2\lambda)^2 (1-\pi_0). 
 \end{aligned}
 \end{equation} 
%\lc{what is the $\times$ here?}\tj{just multiplication}
Notice that the likelihood function of observing $\bu = \{ u_j\}_{j \in [d]}$ given prior $\Pi$ is given by 
 \begin{align*}
 \log q_{\Pi}(\bu; \tau^2) = \sum_{i=1}^{d}\log \int \frac{1}{\sqrt{2\pi\tau^2}} \exp\left\{- \frac{(u_j -\theta)^2}{2\tau^2}\right\}\Pi(\de\theta).
 \end{align*}
This suggests that we can take an estimator of form $\what \Pi = \hat \pi_1 \delta_1 +\pi_0 \delta_0 + (1 - \hat \pi_{1} - \pi_0) \delta_{-1}$ where $\hat \pi_1$ is given by 
\begin{equation}\label{eqn:MLE_pi1}
\hat \pi_1 = \argmax_{\pi_1\in[0,1-\pi_0]} \log q_{\Pi}(v \times \what\bbeta_{\ridge}; \tau^2).
\end{equation}
Note that this is essentially a one-dimensional optimization problem, and computing $q_{\Pi}$ only involves one-dimensional integrations. Hence this is numerically non-expensive to be solved. 

In summary, our algorithm to estimate $\what \Pi = \hat \pi_1 \delta_1 +\pi_0 \delta_0 + (1 - \hat \pi_{1} - \pi_0) \delta_{-1}$ is given as following: (1) Compute $\what\bbeta_{\ridge}$ via Eq. (\ref{eqn:ridge-formula}) for a fixed $\lambda > 0$; (2) Solve $(\tau, v)$ which is the unique solution to Eq. (\ref{eqn:ridge_self-consistent}) (asuming $\sigma^2$ and $\pi_0$ is known); (3) Compute $\hat \pi_1$ being the solution of Eq. (\ref{eqn:MLE_pi1}).

%So, $v\cdot \what \beta_{\ridge,j}$ can be seen as i.i.d. samples generated from a Bayesian model $x\sim\mathcal N(\theta,\tau^2)$ with known variance $\tau^2$ and unknown parameter $\theta\sim \Pi$. 

%It then follows that the log-likelihood function of $\what\bbeta_{\ridge}$ 
% \begin{align*}\log p(\what\bbeta_{\ridge}|\Pi)&\approx
% \sum_{i=1}^{d}\log p(\what \beta_{\ridge,j}|\Pi)\approx \sum_{i=1}^{d}\log \int P(v\cdot \what \beta_{\ridge,j}|\theta)\Pi(\theta)\de\theta\\&=\sum_{i=1}^{d}\log \int \frac{1}{\sqrt{2\pi\tau^2}} \exp\left\{\frac{-(v\cdot \what \beta_{\ridge,j}-\theta)^2}{2\tau^2}\right\}\Pi(\theta)\de\theta.
% \end{align*}

\section{Proofs of Theorem \ref{thm:frequentist_FDR_control}, \ref{thm:frequentist_FDR_control_EB}}\label{sec:freq_proof}

\begin{proof}[Proof of Theorem \ref{thm:frequentist_FDR_control}]
Here we give the proof for {\PoEdCe}. The proof for {\PoPCe} is almost the same. 

By Theorem 5.1 in \cite{wang2022false}, it suffices to show that $(e_j)_{j\in[d]}$ defined in {\PoEdCe}   are valid e-values.  Recall that we denote $(y_i)_{i\in[n]}$ by $\Yb\in\R^n$, $(\bx_i)_{i\in[n]}$ by $\Xb\in\R^{n\times d}$, the $j$-th column of $\Xb$ by $\xb_j$. We also let $\Xb_{-j}$ be the matrix obtained by removing $j$-th column of $\Xb$. For any $j\in [d]$, under the null hypothesis $H_{j0}: Y\indep X_j|\bX_{-j}$, we have
   $$\xb_j|(\Xb_{-j},\Yb) \overset{d}{=} \xb_j|\Xb_{-j}\overset{d}{=}\tilde{\xb}_j^{(k)}|\Xb_{-j}\overset{d}{=}\tilde{\xb}_j^{(k)}|(\Xb_{-j},\Yb)$$
   holds for each $k\in[K]$. Moreover, $\xb_j$ and $\tilde{\xb}_j^{(k)}$ are independent conditional on $(\Xb_{-j}, \Yb)$ by construction. Since the posterior expectation $\what{\bbeta}_{-j}$ is a function of $(\Xb_{-j}, \Yb)$,  it follows from the conditional independence of $\xb_j$ and $(\tilde{\xb}_j^{(k)})_{k\in[K]}$ that $s_j$ and $({s}_j^{(k)})_{k\in [K]}$ are i.i.d.  conditional on $(\Xb_{-j}, \Yb)$, and therefore $u_j$ and $(u_j^{(k)})_{k\in[K]}$ are also  i.i.d. conditional on $(\Xb_{-j}, \Yb)$. Combining this with symmetry and taking expectation over $(\Xb_{-j}, \Yb)$, we obtain $P(p_j\leq c/(K+1))\leq c/(K+1)$ for $c=1,2,...,K+1$, for any $P$ such that $H_{j0}$ holds. Hence $p_j$ is a valid p-value for each $j\in[d]$. Finally, letting $e_j = \ones\{ p_j \le t \} / t$ converts a valid p-value into a valid e-value since $\int_\R \ones\{x\le t \}/t \cdot\de x =1$.
\end{proof}

\begin{proof}[Proof of Theorem \ref{thm:frequentist_FDR_control_EB}]
 Similar to the proof of Theorem \ref{thm:frequentist_FDR_control}, it suffices to show that $(e_j)_{j\in[d]}$ defined in Algorithm \ref{alg:epoedce} are valid e-values. Since $\hat \tau_{-\iota(j)}$, $\hat \sigma_{-\iota(j)}$ and $\what \Pi_{-\iota(j)}$ are constructed using $\Yb,\Xb_{-j} (\supseteq \Xb_{-\iota(j)})$, it follows that $\what \bbeta_{-j}$ is independent of $\xb_j$ under the $j$-th null hypothesis $H_{j0}: Y\indep X_j|\bX_{-j}$. Then following the same argument as in the proof of Theorem  \ref{thm:frequentist_FDR_control}, we have that for any $j\in[d]$, under the null, $s_j$ and $(s_j^{(k)})_{k\in[K]}$ are i.i.d. conditional on $(\Xb_{-j}, \Yb)$.

 Using again the fact that $\hat \tau_{-\iota(j)}$, $\hat \sigma_{-\iota(j)}$ and $\what \Pi_{-\iota(j)}$ are constructed using $\Yb,\Xb_{-j} (\supseteq \Xb_{-\iota(j)})$, and $u_j$ (or $(u_j^{(k)})_{k\in[K]}$) can be viewed as a function of $\hat \tau_{-\iota(j)}$, $\hat \sigma_{-\iota(j)}$, $\what \Pi_{-\iota(j)}$ and $s_j$ (or $(s_j^{(k)})_{k\in[K]}$), it follow that $u_j$ and $(u_j^{(k)})_{k\in[K]}$ are i.i.d. conditional on $(\Xb_{-j}, \Yb)$ under the null hypothesis. Combining this with symmetry, we obtain $P(p_j\leq c/(K+1)|\Yb,\Xb_{-j})\leq c/(K+1)$ for $c=1,2,...,K+1$, for any $P$ such that $H_{j0}$ holds. 
 Finally, 
 $$\E_{H_{j0}} [e_j]=\E_{H_{j0}}\Bigg[\frac{\ones\{ p_j \le t_{-\iota(j)} \}}{t_{-\iota(j)}}\Bigg]=\E_{H_{j0}}\Bigg[\E\left.\left[\frac{\ones\{ p_j \le t_{-\iota(j)} \}}{t_{-\iota(j)}}\right|\Yb,\Xb_{-j}\right]\Bigg]=\E_{H_{j0}}\Bigg[\frac{P(p_j \le t_{-\iota(j)}|\Yb,\Xb_{-j})}{t_{-\iota(j)}}\Bigg]\leq 1,$$
 where the last step follows from the fact that $p_j$ is a valid p-value conditional on $(\Yb,\Xb_{-j})$, and $t_{-\iota(j)}$ is a function of $(\Yb,\Xb_{-j})$. Therefore, $e_j$ is a valid e-value and we conclude the proof.
\end{proof}

\section{Intuitions of Conjecture \ref{conj:Posterior_probability_curve}}\label{sec:justification_conj_1}

Conjecture \ref{conj:Posterior_probability_curve} is based on the following heuristic formalism which we derive using the replica method in statistical physics. This formalism gives the asymptotic joint empirical distributions of $\{ (\what \beta_{ j}, P_j(\cD)) \}_{j \in [d]}$ and $\{ (\what \beta_{ j}, \hat \beta_j(\cD)) \}_{j \in [d]}$. These results are not rigorous proofs, and we leave the proofs for future work. We will provide numerical verifications of the formalism in Appendix \ref{sec:experiment-verification}. 

\begin{formalism}\label{fm:post_mean_zero}
Let $\cD = (\Xb,\Yb)$ be generated from the Bayesian linear model (Assumption \ref{ass:Bayesian_linear_model}). Let $\what \beta_{ j}$ be the posterior expectation of $\beta_{0,j}$ and let $P_{ j}$ be the posterior probability of $\beta_{0,j} = 0$, i.e., 
\[
\what \beta_{j}(\cD) = \E[\beta_{0,j} \vert \cD], ~~~  P_{j}(\cD) = \P(\beta_{0,j} = 0 \vert \cD). 
\]
Then for any sufficiently smooth function $\psi: \R \times \R \mapsto \R$, we have
%\[
%\begin{aligned}
%\{ (\beta_{0, j}, \what \beta_{\Bayes, j}) \}_{j \in [d]} \cong&~ \{ (\beta_{0, j}, \cE(\beta_{0, j} + \tau_\star Z_j)) \}_{j \in [d]}, ~~ d \to \infty, \\
%\{ (\beta_{0, j}, \hat P_{\Bayes, j}) \}_{j \in [d]} \cong&~ \{ (\beta_{0, j}, \PostProb(\beta_{0, j} + \tau_\star Z_j)) \}_{j \in [d]}, ~~ d \to \infty, 
%\end{aligned}
%\]
\begin{align}
\lim_{d\to \infty, n/d \to \delta} \frac{1}{d}\sum_{j = 1}^d \psi(\beta_{0, j}, P_j(\cD))  = &~ \E_{(\beta_0,Z)\sim\Pi\times\cN(0,1)}[ \psi(\beta_0, \PostProb(\beta_{0} + \tau_\star Z) )], \label{eqn:limit-local-fdr-formalism}\\
\lim_{d\to \infty, n/d \to \delta} \frac{1}{d}\sum_{j = 1}^d \psi(\beta_{0, j}, \what \beta_j(\cD))  = &~ \E_{(\beta_0,Z)\sim\Pi\times\cN(0,1)}[ \psi(\beta_0,  \cE(\beta_{0} + \tau_\star Z) )],\label{eqn:limit-PoE-formalism} 
\end{align}
where $\tau_\star$ is the unique minimizer to the potential $\phi$ in Eq.~\eqref{eqn:potential_effective_noise}, and $\cE$ and $\cP$ are given by Eq. \eqref{eqn:posterior_expectation_low} and \eqref{eqn:posterior_probability_low} respectively. 
%where $Z_j \sim_{i.i.d.} \cN(0, 1)$. 
%Then for any sufficiently nice function $\psi: \R^2 \to \R$, we have 
%\[
%\begin{aligned}
%\lim_{n, d \to \infty, n / d \to \delta} \frac{1}{d}\sum_{j = 1}^d \psi(\what \beta_{\Bayes, j}, \beta_{0, j}) =&~ \E_{(\beta_0, Z) \sim \Pi \times \cN(0, 1)}[\psi(\cE(\beta_0 + \tau_\star Z), \beta_0)], \\
%\lim_{n, d \to \infty, n / d \to \delta} \frac{1}{d}\sum_{j = 1}^d \psi(\hat P_{\Bayes, j}, \beta_{0, j}) =&~ \E_{(\beta_0, Z) \sim \Pi \times \cN(0, 1)}[\psi(\PostProb(\beta_0 + \tau_\star Z), \beta_0)]. \\
%\end{aligned}
%\]
\end{formalism}

We first use this formalism to give the intuitions of Conjecture \ref{conj:Posterior_probability_curve}. 

\noindent
{\bf Part (1). Limiting {\FDP} and {\TPP} of {\TPoP}. } Given this formalism, we first derive the limiting formula for {\FDP} and {\TPP} of {\TPoP} as given in Eq. (\ref{eqn:lim_tpop}). Recall the definition of $\FD$, $\TD$, $\Rej$ and $S$ as in Eq. (\ref{eqn:FD_TD}), rewritten here for convenience: 
\[
\begin{aligned}
\FD(\bT_P(t;\cdot)) = &~\sum_{j=1}^d I(\beta_{0,j}=0,P_j(\cD)<t),\quad\Rej(\bT_P(t;\cdot)) =\sum_{j=1}^d I(P_j(\cD) <t),\\
 \TD(\bT_P(t;\cdot)) = &~\sum_{j=1}^d I(\beta_{0,j}\ne0,P_j(\cD) <t),\quad S =\sum_{j=1}^d I(\beta_{0,j}\ne0).
\end{aligned}
\]
Note that these quantities are functions of the joint empirical distribution $\{(\beta_{0, j}, P_j(\cD))\}_{j \in [d]}$. Then applying Eq. \eqref{eqn:limit-local-fdr-formalism} in Formalism \ref{fm:post_mean_zero} gives
\[
\begin{aligned}
&\lim_{d\to \infty, n/d \to \delta} \frac1d\FD(\bT_P(t; \cdot)) = \P( \beta_0 = 0 , \Phi  < t),&
&\lim_{d\to \infty, n/d \to \delta} \frac1d\Rej(\bT_P(t;\cdot))= \P( \Phi  < t),\\
&\lim_{d\to \infty, n/d \to \delta} \frac1d\TD(\bT_P(t;\cdot)) = \P( \beta_0 \ne 0 , \Phi  < t),&
&\lim_{d\to \infty, n/d \to \delta} \frac1dS= \P( \beta_0 \ne 0),
\end{aligned}
\]
where $\Phi = \PostProb(\beta_{0} + \tau_\star Z)$ is as defined in Eq. \eqref{eqn:posterior_probability_low}, and the probabilities to the right hand side of the equations are taken with respect to $(\beta_0,Z)\sim\Pi\times\cN(0,1)$. 
Finally, by the definitions of {\FDP} and {\TPP} as in Eq. (\ref{eqn:FDP_TPP}), we have 
\begin{equation}
\begin{aligned}
\lim_{d \to \infty} \FDP(\bT; \cD, P) = \lim_{d \to \infty}\frac{ \FD(\bT; \cD, P) }{ \Rej(\bT; \cD) \vee 1} = \frac{\P( \beta_0 = 0 , \Phi  < t)}{\P( \Phi  < t)} = \P( \beta_0 = 0 | \Phi  < t),\\
\lim_{d \to \infty} \TPP(\bT; \cD, P) = \lim_{d \to \infty}\frac{ \TD(\bT; \cD, P) }{ |S(P) |\vee 1 } = \frac{\P( \beta_0 \ne 0 , \Phi  < t)}{\P( \beta_0 \ne 0)} = \P( \Phi  < t | \beta_0 \ne 0).
\end{aligned}
\end{equation}
This justifies Eq. (\ref{eqn:lim_tpop}). 

\noindent
{\bf Part (2). Limiting {\FDP} and {\TPP} of {\CPoP}. } We next provide the intuitions for the limiting formula for {\FDP} and {\TPP} of the {\CPoP} procedure as in Eq. (\ref{eqn:lim_cpop}). To show this, note that {\CPoP} (Eq. \eqref{eqn:def-Cpop}) can be viewed as {\TPoP} (Eq. \eqref{eqn:def_T_P}) with data dependant rejection threshold $t(\cD) = P_{\what K(\lambda, \cD)}(\cD)$. Since we already have the limiting {\FDP} and {\TPP} for {\TPoP}, we just need to show that the data-dependent rejection threshold $P_{\what K(\lambda, \cD)}(\cD)$ converges to $t_\star(\lambda)$, and then apply the limiting formula for {\TPoP}. That is, it suffices to show that
\begin{equation}\label{eqn:Cpop_formalism1-suffice}
\lim_{d \to \infty} P_{\what K(\lambda, \cD)}(\cD) = t_\star(\lambda),
\end{equation}
where $t_\star(\lambda)$ is given by Eq. \eqref{eqn:lambda_to_tstar}. 

Note that $\what K(\lambda, \cD)$ is given by the solution of an optimization problem as in Eq. \eqref{eqn:def-K}, and $t_\star(\lambda)$ is given by the solution of another optimization problem as in Eq. \eqref{eqn:lambda_to_tstar}. To show the asymptotic correspondence of $P_{\what K(\lambda, \cD)}(\cD)$ and $t_\star(\lambda)$, we just need to build connections between the objective functions of these two optimization problems. 

In order to build the connection between these two optimization problems, we define
$$h_d(t):=  \P(\Phi<t)  - \big( 1 - \lambda (N/d) / (\P(\Phi<t)  ) \big) \frac1d\sum_{j = 1}^{[d\cdot \P(\Phi<t)]} P_{(j)}(\cD) ,$$
$$h(t):= \P(\Phi < t) - \big( 1 - \lambda (1 - \pi_0) / \P(\Phi < t) \big) \E[ \Phi \cdot \ones\{ \Phi < t\} ].$$
Note that $h_d(t)$ depends on the empirical distribution of $\{ P_{j}(\cD)\}_{j \in [d]}$, so Eq. \eqref{eqn:limit-local-fdr-formalism} in Formalism \ref{fm:post_mean_zero} shows that
$$\lim_{d\to \infty, n/d \to \delta}h_d(t)=h(t).$$ 
Define $K(t) \equiv d\cdot F_\Phi(t)$ where $F_\Phi$ is cumulative distribution function of $\Phi$ as defined in Eq. (\ref{eqn:posterior_probability_low}), we have
$$ h_d(t)=K(t)  - \big( 1 - \lambda N / (K(t) \vee 1) \big) \sum_{j = 1}^{K(t)} P_{(j)}(\cD),$$
which has the same form as the objective function as in Eq. \eqref{eqn:def-K}. Furthermore, by Formalism \ref{fm:post_mean_zero} again, for any $t \in [0, 1]$, the $t$ quantile of $\{ P_j(\cD) \}_{j \in [d]}$ should converge to $F_\Phi^{-1}(t)$, i.e.,  we have
$\lim_{d\to \infty, n/d \to \delta} P_{d\cdot t}(\cD) = \lim_{d\to \infty, n/d \to \delta} F_\Phi^{-1}(t)$. Combining the arguments above, we have
$$\lim_{d\to \infty, n/d \to \delta} P_{\what K(\lambda, \cD)}(\cD) = \lim_{d\to \infty, n/d \to \delta} F_\Phi^{-1}(\what K(\lambda, \cD)/d) = \lim_{d\to \infty, n/d \to \delta}\argmax_{t \in [0, 1]}h_d(t) =\argmax_{t \in [0, 1]}h(t)=t_\star(\lambda).$$
This gives Eq. \eqref{eqn:Cpop_formalism1-suffice}, which gives the desired result Eq. \eqref{eqn:lim_cpop}. 
\qedsymbol

% Thus $\bC_P(\lambda;\cdot)$, which rejects smallest $\what K(\lambda, \cD)$ local fdrs, converges to $\bT_P(t_\star(\lambda);\cdot)$ as $d\to\infty$, i.e., \sm{Explain in details. }
% \[
% \forall \eps>0, \quad\lim_{d\to\infty} \P\Big( \frac{1}{d} \sum_{j = 1}^d \ones\{ T_{P, j}(\lambda ; \cD) \neq C_{P, j}(t_\star(\lambda); \cD) \} \ge \eps \Big) = 0, 
% \]
% which justifies Eq. (\ref{eqn:lim_cpop}).

\vskip0.5cm
\subsection{Intuitions of Formalism \ref{fm:post_mean_zero}}

We next provide the intuitions for Formalism \ref{fm:post_mean_zero}. Throughout this section, we let $\lim_{d\rightarrow\infty}$ or $\lim_{n\rightarrow\infty}$ both mean that $n,d\rightarrow\infty$ and $n/d\rightarrow\delta$. %Recall that Assumption \ref{ass:Bayesian_linear_model} gives $\beta_{0, j} \sim_{i.i.d.} \Pi$, $x_{ij}\sim_{i.i.d.}\cN(0,1/n)$, and $\eps_i\sim_{i.i.d.}\cN(0,\sigma^2)$, for $i \in [n]$ and $j\in [d]$. 
%\begin{itemize}
%\item $\beta_{j}\sim_{i.i.d.}\Pi$, $j=1,\ldots,d$,
%\item $X_{ij}\sim_{i.i.d.}\cN(0,1/n)$, $i=1,\ldots,n$, $j=1,\ldots,d$,
%\item $\eps_i\sim_{i.i.d.}\cN(0,\sigma^2)$,
%\end{itemize}

We use the free energy trick and the \textit{replica method} to calculate the asymptotic empirical distribution of $\{(\beta_{0, j}, \what \beta_{j}(\Yb,\Xb))\}_{j \in [d]}$ and $\{(\beta_{0, j}, P_{j}(\Yb,\Xb))\}_{j \in [d]}$.  Taking a test function $g:\R\rightarrow\R$ (eventually we will take $g$ to be $g(x) = x$ and $g(x) = \ones\{ x = 0\}$), a test function $\psi:\R\times\R \rightarrow\R$, and a scalar parameter $\lambda$, we define a perturbed Hamiltonion $H_{\lambda, N}$ which is a function of $\bar\bbeta=(\bbeta^1,\ldots,\bbeta^N)\in\R^{N\times d}$
\begin{equation}\label{eqn:hamiltonian}
\begin{aligned}
H_{\lambda,N}(\bar\bbeta):=-\frac{1}{2\sigma^2} \sum_{b=1}^N \lVert \Yb-\Xb\bbeta^b\rVert_2^2 + \lambda\sum_{i=1}^d \psi\Big(\frac1N\sum_{b=1}^N g(\beta_i^b),\beta_{0, i}\Big).
\end{aligned}
\end{equation}
We further define $Z_n$ to be the perturbed partition function
\begin{equation}\label{eqn:normalizer}
\begin{aligned}
%Z_n(\lambda,N):=\E_{\bX,\beps}\Big[\int_{\R^{N d}} \exp\{H_\lambda(\bar\bbeta) \} \Pi(\de\bar\bbeta)\Big],
Z_n(\lambda,N):=\int_{\R^{N d}} \exp\{H_\lambda(\bar\bbeta) \} \Pi(\de\bar\bbeta),
\end{aligned}
\end{equation}
where $\Pi(\de \bar \bbeta)$ stands for $\prod_{b = 1}^N \Pi(\de \bbeta^b)$ with some abuse of notations. We then define the free energy density 
\begin{equation}\label{eqn:define_phi_in_replica}
\begin{aligned}
\phi(\lambda,N):=\lim_{d\rightarrow\infty}\frac1d \E_{\Xb,\beps}\big[\log Z_n(\lambda,N)\big].
\end{aligned}
\end{equation}
Here the expectation is with respect to the covariate matrix $\Xb$, and the noise vector $\beps$ (recall that $\Yb = \Xb \bbeta_0 + \beps$ as in Assumption \ref{ass:Bayesian_linear_model}). Taking the derivative of the free energy density and using a heuristic change of limit with derivative, we have
\begin{equation}
\begin{aligned}
\partial_\lambda \phi(\lambda,N)= \lim_{d\rightarrow\infty}\frac1d\E_{\Xb,\beps}\Big[\Big\< \sum_{i=1}^d \psi\Big(\frac1N\sum_{b=1}^N g(\beta_i^b),\beta_{0,i}\Big) \Big\>_{H_{\lambda, N}}  \Big],
\end{aligned}
\end{equation}
where $\< \cdot \>_{H_{\lambda, N}}$ stands for the expectation with respect to $\bar \bbeta \sim Z_n(\lambda, N)^{-1} \exp\{ H_{\lambda, N}(\bar \bbeta)\}$. We then take the $N \to \infty$ limit and set $\lambda = 0$. Using again a heuristic change of limits, and by the law of large numbers, we have
\begin{equation}\label{eqn:formalism1_justify_eqn1}
\begin{aligned}
\lim_{N\rightarrow\infty}\partial_\lambda \phi(\lambda,N)\big|_{\lambda=0}= \lim_{d\rightarrow\infty}\E_{\Xb,\beps}\Big[\frac1d\sum_{i=1}^d \psi(\langle g(\beta_i) \rangle_\mu,\beta_{0,i}) \Big],
\end{aligned}
\end{equation}
where we used $\lim_{N \to \infty}\< N^{-1}\sum_{b=1}^N g(\beta_i^b) \>_{H_{\lambda = 0, N}} = \langle g(\beta_i) \rangle_\mu$, and $\langle\cdot\rangle_\mu$ stands for the expectation with respect to $\mu \in \ProbFam(\R^d)$, where
$$\mu(\de\bbeta)\propto \exp\Big\{ -\frac{1}{2\sigma^2} \lVert\Yb-\Xb\bbeta\rVert_2^2 \Big\} \Pi(\de \bbeta).$$
We would expect that the right hand side of Eq. (\ref{eqn:formalism1_justify_eqn1}) concentrates well around its expectation, so that as $d\to\infty$, we can remove the expectation operator, i.e.
\begin{equation}\label{eqn:formalism1_justify_eqn2}
\begin{aligned}
\lim_{N\rightarrow\infty}\partial_\lambda \phi(\lambda,N)\big|_{\lambda=0}= \lim_{d\rightarrow\infty}\frac1d\sum_{i=1}^d \psi(\langle g(\beta_i) \rangle_\mu,\beta_{0,i}). 
\end{aligned}
\end{equation}
Note that the right hand side of Eq. \eqref{eqn:formalism1_justify_eqn2} above is what we are interested in. Our goal is thus to calculate the left hand side of  Eq. \eqref{eqn:formalism1_justify_eqn2}, and we claim that the following equation holds. 
\begin{claim}\label{clm:replica}
Under the same setup as Formalism \ref{fm:post_mean_zero}, we have
\begin{equation}\label{eqn:replica_goal}
\lim_{N\rightarrow\infty}\partial_\lambda \phi(\lambda,N)\big|_{\lambda=0} =\E_{(\beta_0,G)\sim\Pi\times\cN(0,1)}\big[ \psi(\E[g(\beta_0)|\beta_0+\tau_\star G],\beta_0)\big]. 
\end{equation}
\end{claim}
Combining Eq. \eqref{eqn:formalism1_justify_eqn2} and \eqref{eqn:replica_goal}, taking $g(x)=x$, we get
\[
\lim_{d\rightarrow\infty}\frac1d \sum_{i=1}^d \psi(\langle \beta_{i} \rangle_\mu,\beta_{0,i}) = \E_{(\beta_0,G)\sim\Pi\times\cN(0,1)}\big[ \psi(\cE(\beta_0+\tau_\star G),\beta_0)\big],
\]
and taking $g(x)=\ones\{ x = 0\}$, we get
\[
\lim_{d\rightarrow\infty}\frac1d \sum_{i=1}^d \psi(\langle \ones\{ \beta_{i} = 0\} \rangle_\mu,\beta_{0,i}) = \E_{(\beta_0,G)\sim\Pi\times\cN(0,1)}\big[\psi( \PostProb(\beta_0+\tau_\star G),\beta_0)\big].
\]
These are the desired equations \eqref{eqn:limit-local-fdr-formalism} and \eqref{eqn:limit-PoE-formalism} in Formalism \ref{fm:post_mean_zero}.

\vskip0.5cm
\subsection{Intuitions of Claim \ref{clm:replica}}
 
We are thus left to give the intuitions for Claim \ref{clm:replica}. To calculate the left hand side of Eq. \eqref{eqn:replica_goal}, we need to first calculate $\phi(\lambda, N) = \lim_{d \to \infty} \E[\log Z_n(\lambda, N)]$. We calculate $\phi(\lambda, N)$ using the replica trick $\E[\log Z]=\lim_{k\rightarrow0}\log\E[Z^k]/k$ \cite{mezard2009information}, and using a heuristic exchange of limits $d \to \infty$ and $k \to 0$. The calculation of $\lim_{N \to \infty} \partial_\lambda \phi(\lambda,N) |_{\lambda=0}$ is thus divided into three steps as below. 
\begin{enumerate}
\item[S1. ] The $d\rightarrow\infty$ limit. For fixed integer $k$, $N$, and scalar $\lambda \in \R$, we calculate
\begin{equation}\label{eqn:define_S}
S(k,\lambda,N)\equiv \lim_{d\rightarrow\infty} \frac1d \log\E_{\Xb,\beps}\big[ Z_n(\lambda,N)^k\big].
\end{equation}
\item[S2. ] The $k\rightarrow0$ limit. For fixed integer $N$ and scalar $\lambda \in \R$, we calculate
\begin{equation}\label{eqn:formula_phi_in_replica}
\phi(\lambda,N)=\lim_{k\rightarrow0}\Big[\frac1k S(k,\lambda,N)\Big].
\end{equation}
\item[S3. ] The $\lambda$ differentiation. We calculate the derivative with respect to $\lambda$, and take $N \to \infty$
\begin{equation}\label{eqn:define_s_star_in_replica}
\psi_*\equiv \lim_{N\rightarrow\infty}\partial_\lambda\phi(\lambda,N)\big|_{\lambda=0}.
\end{equation}

\end{enumerate}

\noindent
\textbf{Step S1. The $d\rightarrow\infty$ limit.} Throughout the rest of this section, the indices $a,b,i$ under the summation or product operators run over $a\in[k]$, $b\in[N]$, and $i\in[d]$. For example, we write $\sum_{a,b}=\sum_{a \in [k],b \in [N]}$ in short. We start with calculating $\E_{\Xb,\beps}[ Z_n(\lambda,N)^k ]$. Recall that Assumption \ref{ass:Bayesian_linear_model} gives $\beta_{0, j} \sim_{i.i.d.} \Pi$, $x_{ij}\sim_{i.i.d.}\cN(0,1/n)$, and $\eps_i\sim_{i.i.d.}\cN(0,\sigma^2)$, for $i \in [n]$ and $j\in [d]$. Using the fact that $(\int_{\R^{d N}} f(\bar \bbeta) \Pi(\bar \bbeta) )^k = \int_{\R^{d N k}} \prod_{a} f(\{ \bbeta^{(a | b)} \}_b) \prod_{a,b} \Pi(\de \bbeta^{(a|b)})$, we obtain
\begin{equation}\label{eqn:EZk_proof}
\begin{aligned}
\E_{\Xb,\beps}[Z_n^k]=&\E_{\Xb,\beps}\int_{\R^{d\times k\times N}} \exp\Bigg\{-\frac1{\sigma^2}\sum_{a,b}\lVert \Yb-\Xb\bbeta^{(a|b)}\rVert_2^2/2+\lambda \sum_{i,a}\psi\bigg(\Big(\frac1N\sum_{b}\beta_i^{(a|b)}\Big),\beta_{0,i}\bigg)\Bigg\}\prod_{a,b} \Pi\big(\de\bbeta^{(a|b)}\big)\\
=&\int_{\R^{d\times k\times N}}\underbrace{\E_{\Xb,\beps}\Big[\exp\big\{-\frac1{\sigma^2}\sum_{a,b}\lVert \Yb-\Xb\bbeta^{(a|b)}\rVert_2^2/2\big\}\Big]}_{E(\bbeta_\square)}\times\exp\Big\{\lambda \sum_{i,a}\psi\big(\frac1N\sum_{b}\beta_i^{(a|b)},\beta_{0,i}\big)\Big\}\prod_{a,b} \Pi\big(\de\bbeta^{(a|b)}\big), 
\end{aligned}
\end{equation}
where we denote $\bbeta_\square = \{ \{ \bbeta^{(a | b)}\}_{a, b}, \bbeta_0 \}$,
We simplify $E(\bbeta_\square)$ as follows: 
\begin{align*}
E(\bbeta_\square)=&~\E_{\Xb,\beps}\Big[\exp\big\{-\frac1{2\sigma^2}\sum_{a,b}\lVert \Yb-\Xb\bbeta^{(a|b)}\rVert_2^2 \big\}\Big]\\
%=&~\E_{\Xb,\beps}\Big[\exp\big\{-\frac1{2\sigma^2}\sum_{a,b}\lVert \beps-\Xb(\bbeta_0-\bbeta^{(a|b)}) \rVert_2^2 \big\}\Big]\\
=&~ \E_{\Xb}\int_{\R^n}\frac1{(2\pi)^{n/2}\sigma^n}\exp\Big\{-\frac1{2\sigma^2} kN\lVert\bepsilon\rVert_2^2 -\frac{\lVert\bepsilon\rVert_2^2}{2\sigma^2}+\frac1{\sigma^2} \Big\langle\bepsilon,\sum_{a,b}\Xb(\bbeta_0-\bbeta^{(a|b)})\Big\rangle\Big\}\de\bepsilon\\
&~ \times \exp\Big\{-\frac1 {2\sigma^2}\sum_{a,b}\lVert\Xb(\bbeta_0-\bbeta^{(a|b)})\rVert_2^2\Big\}\\=&~
\underbrace{(kN+1)^{-\frac n2}}_{c_n}\cdot\E_\Xb\Big[ \exp\Big\{\frac1{2\sigma^4}\Big\lVert\sum_{a,b} \Xb(\bbeta_0-\bbeta^{(a|b)})\Big\rVert_2^2\bar\sigma^2-\frac1{2\sigma^2}\sum_{a,b}\Big\lVert\Xb(\bbeta_0-\bbeta^{(a|b)})\Big\rVert_2^2\Big\} \Big]\\
=&
~ c_n \E_{\bx \sim \cN(\bzero, \frac1n \id_d)} \Big[ \exp\Big\{\frac1{2\sigma^4} \Big(\sum_{a,b} \bx^\top(\bbeta_0-\bbeta^{(a|b)}) \Big)^2\bar\sigma^2-\frac1{ 2\sigma^2}\sum_{a,b}\Big(\bx^\top(\bbeta_0-\bbeta^{(a|b)})\Big)^2 \Big\}   \Big]^n,
\end{align*}
where we denote $\bar\sigma^2= \sigma^2/(1+Nk)$ and $c_n = (kN+1)^{-\frac n2}$.

To further simplify the expression above, we define the overlaps of $\bbeta_\square$ given by ($\ones_{kN} \in \R^{kN}$ is the all one vector)
\begin{align*}
\bar\bQ(\bbeta_\square)=&~\Big(\langle \bbeta^{(a|b)},\bbeta^{(a'|b')}\rangle/d\Big)_{(a|b),(a'|b')\in[k]\times[N]} \in \R^{kN \times kN},\\
\bar\bmu(\bbeta_\square)=&~\Big(\langle\bbeta^{(a|b)},\bbeta_0\rangle/d\Big)_{(a|b)\in[k]\times[N]} \in \R^{k N},\\
p(\bbeta_0) =&~\| \bbeta_0 \|_2^2/d \to p \equiv \E_{\beta \sim \Pi}[ \beta^2]  \in \R, \\
\Sigma(\bar\bQ(\bbeta_\square),\bar\bmu(\bbeta_\square))=&~ \Big(\bar\bQ(\bbeta_\square)-\bar\bmu(\bbeta_\square)\ones_{kN}^\top-\ones_{kN}\bar\bmu(\bbeta_\square)^\top+p(\bbeta_0) \ones_{kN} \ones_{kN}^\top\Big)/\delta \in \R^{kN \times kN}. 
\end{align*}
Now for a Gaussian random vector $\bx\sim\cN(0,(1/n) \id_d)$, we define $G_{(a|b)}=\bx^T(\bbeta_0-\bbeta^{(a|b)})$. Then these $\{ G_{(a|b)} \}_{a, b}$ are multi-variate Gaussian variables with mean $0$ and covariance $\E[G_{(a|b)} G_{(a'|b')}]=(\bbeta_0-\bbeta^{(a|b)})^\top(\bbeta_0-\bbeta^{(a'|b')})/n = \Sigma(\bar\bQ(\bbeta_\square),\bar\bmu(\bbeta_\square))_{(a|b),(a'|b')}$. 

Then we can further simplify $E(\bbeta_\square)$ as follows
\begin{align*}
E(\bbeta_\square)=&
c_n \bigg[ \int \frac{1}{(2\pi)^{kN/2} \det(\Sigma)^{kN/2}}\exp\Big\{\frac1{2\sigma^4}\Big(\sum_{a,b} G_{(a|b)} \Big)^2\bar\sigma^2-\frac1{2\sigma^2}\sum_{a,b} G_{(a|b)}^2- \bG^T\Sigma^{-1}\bG/2\Big\}\de\bG  \bigg]^n\\=&
c_n\big[ \det(\Sigma^{-1}+\id_{kN}/\sigma^2 -\bar\sigma^2 \ones_{kN} \ones_{kN}^\top)\det(\Sigma) /\sigma^4 \big]^{-\frac n2}. 
\end{align*}
Note that $E(\bbeta_\square)$ depend on $\bbeta_\square$ only through $\Sigma(\bar\bQ(\bbeta_\square),\bar\bmu(\bbeta_\square))$. Thus defining 
\begin{equation}\label{eqn:def_Eng}
    \begin{aligned}
       \bar E (\bQ, \bmu) = c_n\big[ \det(\Sigma(\bQ,\bmu)^{-1}+\id_{kN}/\sigma^2 -\bar\sigma^2 \ones_{kN} \ones_{kN}^\top)\det(\Sigma(\bQ,\bmu)) /\sigma^4 \big]^{-\frac n2} , 
    \end{aligned}
\end{equation}
we have $E(\bbeta_\square)=\bar E (\bar\bQ(\bbeta_0),\bar\bmu(\bbeta_0))$. 

Now plugging the expression of $E(\bbeta_\square)$ into Eq. \eqref{eqn:EZk_proof} and using the delta identity formula $1=\int \delta(\bar\bQ-\bQ)\delta(\bar\bmu-\bmu)\de\bQ \de\bmu$, we have
\begin{equation}
    \begin{aligned}
\E_{\Xb,\beps}[Z_n^k]=&\int_{\R^{d\times k\times N}} \bar E(\bar\bQ(\bbeta_0),\bar\bmu(\bbeta_0))\times\exp\Big\{\lambda \sum_{i,a}\psi\Big(\frac1N\sum_b g( \beta_i^{(a|b)}),\beta_{0,i}\Big)\Big\}\prod_{a,b}\Pi\big(\de\bbeta^{(a|b)}\big)\\=&
\int \de\bQ \de\bmu \bar E(\bQ,\bmu)\times\text{Ent}(\bQ,\bmu),
\end{aligned}
\end{equation}
where 
$$\text{Ent}(\bQ,\bmu)=\int \delta(\bar\bQ-\bQ)\delta(\bar\bmu-\bmu)\times \exp\Big\{\lambda\sum_{i,a}\psi\Big(\frac1N \sum_b g(\beta_i^{(a|b)}),\beta_{0,i}\Big)\Big\}\prod_{a,b}\de\Pi\big(\de\bbeta^{(a|b)}\big). $$
Using the Laplace method $\lim_{n\to\infty}\frac1n\log\int \exp\{n f_n(\bQ, \bmu)\}\de \bQ \de \bmu =\sup_{\bQ, \bmu} \lim_{n\to\infty}f_n(\bQ, \bmu)$, we have
\begin{equation}\label{eqn:Z_to_E_Ent}
S(k,\lambda,N)\equiv \lim_{d\rightarrow\infty} \frac1d \log\E_{\Xb,\beps}\big[ Z_n(\lambda,N)^k\big] = \sup_{\bQ,\bmu} \Big[\lim_{d \to \infty} \frac{1}{d}\log \bar E(\bQ,\bmu) + \frac{1}{d} \log \text{Ent}(\bQ,\bmu)\Big]. 
\end{equation}

It is straightforward to calculate $\lim_{d \to \infty} d^{-1} \log \bar E(\bQ, \bmu)$. Denoting
\begin{equation}\label{eqn:eng_expand}
    \begin{aligned}
e(\bQ,\bmu):=-\frac\delta 2\log(1+kN)-\frac\delta 2 \log\det[\id_{kN}+(\bQ-\bmu\ones^\top-\ones\bmu^\top+p\ones \ones^\top)( \id_{kN}/\sigma^2-\bar\sigma^2 \ones \ones^\top/\sigma^4)].
\end{aligned}
\end{equation}
Then it is straightforward to see that
\begin{equation}\label{eqn:def_ent_e}
    \begin{aligned}
        e(\bQ,\bmu):=\lim_{d \to \infty} \frac{1}{d} \log \bar E(\bQ, \bmu). 
    \end{aligned}
\end{equation}

To calculate $\lim_{d \to \infty} d^{-1} \log \text{Ent}(\bQ,\bmu)$, using the delta identity formula
\[
\delta(\bar\bQ-\bQ)\delta(\bar\bmu-\bmu) =\int \exp\{i \cdot d \cdot(\langle \br ,\bar\bQ-\bQ\rangle + \< \bxi, \bar \bmu - \bmu\> ) \}\de (\br/(2\pi)) \de (\bxi/(2\pi)),
\]
and saddlepoint approximation (here $\text{ext}$ stands for the extremum)
\[
\lim_{n\to\infty}\frac1n\log\int_\R\exp\{n f_n(i\br, i \bxi)\}\de\lambda=\text{ext}_{\br \in \C^{kN \times kN}, \bxi \in \R^{kN}}\lim_{n\to\infty}f_n(\br, \bxi),
\]
we have
\begin{align*}
\text{Ent}(\bQ,\bmu) \asymp &~\text{ext}_{\br, \bxi}\int_{\R^{d\times k\times N}}\exp\Big\{ -\sum_{a,b,a',b'}r_{(a|b),(a'|b')}(d \cdot Q_{(a|b),(a'|b')}-\langle\bbeta^{(a|b)},\bbeta^{(a'|b')}\rangle)/2-\\&\sum_{a,b}\xi_{(a|b)}(d\cdot\mu_{(a|b)}-\langle\bbeta^{(a|b)},\bbeta_0\rangle)+\lambda \sum_{i,a}\psi\Big(\frac1N\sum_b g(\beta_i^{(a|b)}),\beta_{0,i}\Big)\Big\}\prod_{a,b}\Pi\big(\de\bbeta^{(a|b)}\big)\\\asymp &~ 
\text{ext}_{\br, \bxi}\Big[ \prod_{i=1}^d\Big( \int_{\R^{k\times N}}\exp\Big\{-\sum_{a,b,a',b'}r_{(a|b),(a'|b')}(Q_{(a|b),(a'|b')}-\beta^{(a|b)}\beta^{(a'|b')})/2\\&-\sum_{a,b}\xi_{(a|b)}(\mu_{(a|b)}-\beta^{(a|b)}\beta_{0, i})+
\lambda\sum_{a}\psi\Big(\frac1N\sum_b g(\beta^{(a|b)}),\beta_{0, i} \Big)\Big\}\prod_{a,b}\Pi\big(\de \beta^{(a|b)}\big) \Big) \Big].
\end{align*}
Then by the law of large numbers uniform in $(\br, \bxi)$, and recall that $(\beta_{0, i}) \sim_{i.i.d.} \Pi$, we have 
\begin{equation}\label{eqn:def_ent}
\begin{aligned}
\lim_{d\rightarrow\infty}\frac 1d \log\text{Ent}(\bQ,\bmu) = &~ \text{ext}_{\br, \bxi}  \text{ent}(\bQ,\bmu,\br,\bxi), 
\end{aligned}
\end{equation}
where
\begin{equation}\label{eqn:ent_expand}
\begin{aligned}
 \text{ent}(\bQ,\bmu,\br,\bxi) =&~ \E_{\beta_0 \sim \Pi}\Big[\log\int_{\R^{k\times N}}\exp\Big\{-\sum_{a,b,a',b'}r_{(a|b),(a'|b')}(Q_{(a|b),(a'|b')}-\beta^{(a|b)}\beta^{(a'|b')})/2\\-&\sum_{a,b}\xi_{(a|b)}(\mu_{(a|b)}-\beta^{(a|b)}\beta_0)+\lambda\sum_{a}\psi\Big(\frac1N\sum_{b}g(\beta^{(a|b)}),\beta_0\Big) \Big\}\prod_{a,b}\Pi\big(\de\beta^{(a|b)}\big) \Big]. 
\end{aligned}    
\end{equation}
Therefore, combining Eq. (\ref{eqn:Z_to_E_Ent}), (\ref{eqn:def_ent_e}), (\ref{eqn:def_ent}), we conclude that $S(k,\lambda,N)$ as defined in Eq. (\ref{eqn:define_S}) gives
\begin{equation}\label{eqn:S_to_E_Ent}
\begin{aligned}
S(k,\lambda,N) =&~ \text{ext}_{\bQ,\bmu} \Big[ \lim_{d \to \infty}\frac{1}{d}\log \bar E(\bQ,\bmu) + \lim_{d \to \infty}\frac{1}{d} \log \text{Ent}(\bQ,\bmu)\Big]\\
=&~ \text{ext}_{\bQ,\bmu,\br,\bxi}[e(\bQ,\bmu)+\text{ent}(\bQ,\bmu,\br,\bxi)],
\end{aligned}
\end{equation}
where $e$ and $\text{ent}$ are as defined in Eq. \eqref{eqn:eng_expand} and \eqref{eqn:ent_expand} respectively. 
%where
%\begin{align*}
%\text{ent}(\bQ,\bmu,\br,\bxi):=& \E_{\beta_0}\Big[\log\int_{\R^{k\times N}}\exp\big\{-\sum_{a,b,a',b'}r_{(a|b),(a'|b')}(Q_{(a|b),(a'|b')}-\beta^{(a|b)}\beta^{(a'|b')})/2-\sum_{a,b}\xi_{(a|b)}(\mu_{(a|b)}-\beta^{(a|b)}\beta_0)\\+&\lambda\sum_{a}\psi\big(\frac1N\sum_{b}g(\beta^{(a|b)}),\beta_0\big) \big\}\prod_{a,b}\Pi\big(\de\beta^{(a|b)}\big) \Big],
%\end{align*}

\noindent\textbf{Step S2. The $k\rightarrow0$ limit. } We next calculate the $k \to 0$ limit in Eq. (\ref{eqn:formula_phi_in_replica}). The difficulty lies in that the dimension of $(\bQ,\bmu,\br,\bxi)$ depends on $k$. Following the replica trick in statistical physics, we use the replica symmetric ansatz to simplify the expression of $S$, and then calculate the $k \to 0$ limit. Using the replica symmetric ansatzs, we assume that the variables $(\bQ, \bmu, \br, \bxi)$ achieving the extremum of Eq. (\ref{eqn:S_to_E_Ent}) are replica symmetric in the following sense: there exists variables $(q_0, q_1, q_2, r_0, r_1, r_2, \mu, \xi)$ such that $\bQ$ and $\br$ have the block form (where each block is of size $k \times k$)
\begin{equation}\label{eqn:matrix_Q_r}
\bQ=\begin{pmatrix}
 \begin{matrix}q_1&&q_0\\&\ddots&\\q_0&&q_1 \end{matrix} & \rvline  &q_2&\rvline&\cdots&\rvline&q_2\\
 \hline
 q_2 &\rvline&\begin{matrix}q_1&&q_0\\&\ddots&\\q_0&&q_1 \end{matrix} & \rvline &\cdots&\rvline&q_2\\
 \hline
 \vdots&\rvline&\vdots&\rvline&\ddots&\rvline&\vdots\\
 \hline
 q_2 & \rvline  &q_2&\rvline&\cdots&\rvline&\begin{matrix}q_1&&q_0\\&\ddots&\\q_0&&q_1 \end{matrix}
\end{pmatrix},\br=\begin{pmatrix}
 \begin{matrix}r_1&&r_0\\&\ddots&\\r_0&&r_1 \end{matrix} & \rvline  &r_2&\rvline&\cdots&\rvline&r_2\\
 \hline
 r_2 &\rvline&\begin{matrix}r_1&&r_0\\&\ddots&\\r_0&&r_1 \end{matrix} & \rvline &\cdots&\rvline&r_2\\
 \hline
 \vdots&\rvline&\vdots&\rvline&\ddots&\rvline&\vdots\\
 \hline
 r_2 & \rvline  &r_2&\rvline&\cdots&\rvline&\begin{matrix}r_1&&r_0\\&\ddots&\\r_0&&r_1 \end{matrix}
\end{pmatrix},
\end{equation}
and $\bmu$ and $\bxi$ have the following form
$$ \bmu= [\mu, \ldots, \mu]^\sT,~~~~~~~\bxi=[\xi, \ldots, \xi]^\sT. 
$$
We further reparametrize these variables and introduce $(q, w_1, w_2, \rho_1, \rho_2, \nu, \zeta)$ satisfying 
\begin{equation}
\begin{aligned}
&~q_2=q, ~~~~~~q_0=q_2+\sigma^2{w_2},~~~~~~ q_1=q_0+\sigma^2 {w_1}, \\
&~r_2=\frac1{\sigma^4}\rho_1,~~~~~~ r_0=r_2+\frac1{\sigma^4}\rho_2,~~~~~~ r_1-r_0=-\nu/\sigma^2,~~~~~~ \xi=\zeta/\sigma^2.
\end{aligned}
\end{equation}
Using this parametrization, take the $e$ function defined as in Eq. (\ref{eqn:eng_expand}), and take $k\to 0$ limit, we have
\begin{align}\label{eqn:def_ebar}
\lim_{k\rightarrow 0}\frac1k e(\bQ,\bmu) = \bar e(q,w_1,w_2,\mu) \equiv -\frac{\delta N}2\big[\log(1+\frac {w_1}\delta)+\frac{w_2}{\delta +w_1}+\frac{1}{\sigma^2(\delta+w_1)}(p-2\mu+q+\delta\sigma^2)\big].
\end{align}
Moreover, using this parameterization, the $\text{ent}$ function as defined in Eq. (\ref{eqn:def_ent}) gives 
\begin{align*}
\text{ent}(\bQ, \bmu, \br, \bxi)=& -\frac12\Big(N^2 k^2q\rho/\sigma^4 + Nk^2(\rho_1w_2/\sigma^2+\rho_2w_2/\sigma^2+\rho_2q/\sigma^4)\\
&+Nk(-\nu w_1-\nu w_2-\nu q/\sigma^2+ w_1\rho_1/\sigma^2+ w_1\rho_2/\sigma^2)\Big)-Nk\mu\zeta/\sigma^2
\\&+\E_{\beta_0}\bigg[ \log \int_{\R^{k\times N}}\exp\Big\{ \frac{\rho_1}{2\sigma^4}(\sum_{a,b}\beta^{(a|b)})^2+\frac{\rho_2}{2\sigma^4}\sum_b(\sum_{a}\beta^{(a|b)})^2-\frac{\nu}{2\sigma^2}\sum_{a,b}(\beta^{(a|b)})^2\\&+\frac\zeta{\sigma^2}\sum_{a,b}\beta^{(a|b)}\beta_0 +\lambda\sum_{a}\psi\Big(\frac1N\sum_b g(\beta^{(a|b)}),\beta_0\Big) \Big\}\prod_{a,b} \Pi\big(\de\beta^{(a|b)}\big)\bigg].
\end{align*}
To further simplify the equation above, we use the fact that for $G\sim\cN(0,1)$, $\E[\exp(\lambda G)]=\exp(\lambda^2/2)$. So that we introduce Gaussian random variables $G_0,\ldots,G_n \sim_{i.i.d.} \cN(0,1)$, then the $\E_{\beta_0}[\cdot]$ part of the equation above becomes
\begin{align*}
&\E_{\beta_0}\bigg[\log\E_{G_0,\ldots,G_N} \Big[\int_{\R^{k\times N}} \exp\big\{ \frac{\sqrt{\rho_1}}{\sigma^2}\sum_{a,b}\beta^{(a|b)}G_0+\frac{\sqrt{\rho_2}}{\sigma^2}\sum_b(\sum_a \beta^{(a|b)})G_b
\\&- \frac\nu{2\sigma^2}\sum_{ab}(\beta^{(a|b)})^2+\frac\zeta{\sigma^2}\sum_{a,b}\beta^{(a|b)}\beta_0+\lambda\sum_a\psi\big(\frac1N\sum_b g(\beta^{(a|b)}),\beta_0\big)\big\}\prod_{a,b} \Pi\big(\de\beta^{(a|b)}\big)\Big] \bigg]
\\=&\E_{\beta_0}\bigg[\log  \E_{G_0,\ldots,G_N}  \Big\{ \Big[\int_{\R^N}\exp\big\{ \frac1{\sigma^2}\sum_b \beta^b(\sqrt{\rho_1}G_0+\sqrt{\rho_2}G_b)-\frac{\nu}{2\sigma^2} \sum_{b}(\beta^b)^2
\\&+\frac\zeta{\sigma^2}\sum_b \beta^b \beta_0+\lambda \psi\big(\frac1N\sum_b g(\beta^b),\beta_0\big)\big\} \prod_{b} \Pi\big(\de\beta^{b}\big) \Big]^k \Big\}  \bigg]. 
\end{align*}
To take the $k \to 0$ limit, using the replica formula $\E[\log Z] = \lim_{k \to 0} (1/k) \log \E[Z^k]$ in a reverse way, we obtain
\begin{equation}\label{eqn:def_entbar}
\begin{aligned}
\lim_{k\rightarrow0}\frac1k\text{ent}=&~ \overline{\text{ent}}(q,w_1,w_2,\mu,\rho_1,\rho_2,\nu,\zeta)\\
\equiv&~ N[-\frac12(-\nu w_1-\nu w_2-\nu/\sigma^2 q+w_1\rho_1/\sigma^2+ w_1\rho_2/\sigma^2)-\mu\zeta/\sigma^2]\\&+\E_{\beta_0,G_0,\ldots,G_N}\bigg[\log\int_{\R^N}\exp\big\{ \frac1{\sigma^2}\sum_b \beta^b(\sqrt{\rho_1}G_0+\sqrt{\rho_2}G_b)-\frac{\nu}{2\sigma^2} \sum_{b}(\beta^b)^2\\&+\frac\zeta{\sigma^2}\sum_b \beta^b\beta_0+\lambda \psi\big(\frac1N\sum_b g(\beta^b),\beta_0\big) \big\} \prod_{b} \Pi\big(\de\beta^{b}\big)  \bigg].
\end{aligned}
\end{equation}
Then, combining Eq. \eqref{eqn:formula_phi_in_replica}, \eqref{eqn:S_to_E_Ent}, \eqref{eqn:def_ebar}, and \eqref{eqn:def_entbar}, we have
\[
\begin{aligned}
\phi(\lambda,N)=&~ \text{ext}_{q,w_1,w_2,\mu,\rho_1,\rho_2,\nu,\zeta}[\overline{e}(q,w_1,w_2,\mu)+\overline{\text{ent}}(q,w_1,w_2,\mu,\rho_1,\rho_2,\nu,\zeta) ] \\
=&~ \text{ext}_{q,w_1,w_2,\mu,\rho_1,\rho_2,\nu,\zeta}\Big\{  -\frac N{\sigma^2}\Big[\frac{\delta\sigma^2}{2}\big(\log(1+\frac{w_1}{\delta})+\frac{w_2}{\delta+w_1}\big)+\frac{\delta}{2(\delta+w_1)}(p-2\mu+q+\delta\sigma^2)\\
&~ +\frac{\sigma^2}2(-\nu w_1-\nu w_2) +\frac12(w_1\rho_1+w_1\rho_2-\nu q)+\mu\zeta \Big] \\
&~ +
\E_{\beta_0,G_0,\ldots,G_N}\Big[\log\int_{\R^N}\exp\big\{\frac1{\sigma^2}\big(\sum_b \beta^b(\sqrt{\rho_1} G_0+\sqrt{\rho_2}G_b+\zeta  \beta_0)-\frac12 \nu \sum_b(\beta^b)^2\big)\\
&~+\lambda\psi\big(\frac1N\sum_b g(\beta^b),\beta_0\big) \big\} \prod_{b} \Pi\big(\de\beta^{b}\big) \Big] \Big\}. 
\end{aligned}
\]
Taking derivatives with respect to $(q,\mu)$ in the equation above and setting them to be zero, we obtain that the extremum will take place at $\delta/(\delta+w_1)=\nu=\zeta$. Plugging in this equality, we get a simplified equation for $\phi(\lambda, N)$
\begin{equation}\label{eqn:phi_expression-in-proof}
\begin{aligned}
 \phi(\lambda,N)=&~\text{ext}_{\rho_1,\rho_2,\nu} -\frac{N}{\sigma^2}\Big[\frac\nu2(p+\delta\sigma^2)-\frac{\delta\sigma^2}{2}\log\nu +\frac\delta2(\frac1\nu-1)(\rho_1+\rho_2-\nu\sigma^2)\Big]
\\
&~+\E_{\beta_0,G_0,\ldots,G_N}\Big[\log\int_{\R^N}\exp\Big\{\frac1{\sigma^2}\big(\sum_b \beta^b(\sqrt{\rho_1} G_0+\sqrt{\rho_2}G_b+\nu \beta_0)-\frac12 \nu \sum_b(\beta^b)^2\big)\\
&~+\lambda\psi\Big(\frac1N\sum_b g(\beta^b),\beta_0\Big)\Big\} \prod_{b} \Pi\big(\de\beta^{b}\big)  \Big]. 
\end{aligned}
\end{equation}

\noindent \textbf{Step S3. The $\lambda$ differentiation.} 
We finally calculate the $\lambda$ differentiation as in Eq. (\ref{eqn:define_s_star_in_replica}). Using Danskin's theorem, we get
\begin{align}\label{eqn:psi_star_in_formalism1}
\partial_\lambda\phi(\lambda,N)\big|_{\lambda=0}=\E_{\beta_0,\what \beta^1,\ldots,\what \beta^N,G_0,\ldots,G_N} \psi\big(\frac1N\sum_b g(\what \beta^b),\beta_0\big),
\end{align}
where 
\begin{equation}
(\beta_0, G_0, \ldots, G_N) \sim \Pi \times \cN(0, 1)^{\otimes (N+1)},
\end{equation} and
\begin{equation}\label{eqn:conditional_beta_distribution_in-justrificaiton}
(\what \beta^1,\ldots,\what \beta^N)|_{\beta_0,G_0,\ldots,G_N} \sim {\bar \mu}(\de \bar \bbeta) \propto\exp\Big\{-\frac{\nu}{2\sigma^2}\sum_b\Big[\beta^b-(\beta_0+\frac{\sqrt{\rho_1}}{\nu}G_0+\frac{\sqrt{\rho_2}}{\nu}G_b)\Big]^2\Big\}\Pi\big(\de\bar\bbeta\big),
\end{equation}
with $(\rho_1,\rho_2, \nu)$ satisfying the following self-consistent equation, which is obtained by setting stationary of the objective in Eq. \eqref{eqn:phi_expression-in-proof} with respect to $(\rho_1, \rho_2, \nu)$
\begin{align}&N\Big[\frac12(p+\delta\sigma^2)-\frac{\delta\sigma^2}{2\nu}-\frac{(\rho_1+\rho_2)\delta}{2\nu^2}+\frac12\delta\sigma^2\Big]+\E\Big[\sum_b\Big(\frac12 (\what \beta^b)^2-\beta_0\what \beta^b \Big) \Big]=0, \label{eqn:scf1-in-proof}\\
&\frac N2\delta(\frac1\nu-1)=\frac1{2\sqrt{\rho_2}}\E\Big[\sum_b \what \beta^b G_b\Big],\label{eqn:scf2-in-proof}\\
&\frac N2\delta(\frac1\nu-1)=\frac1{2\sqrt{\rho_1}}\E\Big[\sum_b \what \beta^b G_0\Big].\label{eqn:scf3-in-proof}
\end{align}
Here the expectations in these equations are taken with respect to $(\beta_0,\what \beta^1,\ldots,\what \beta^N,G_0,\ldots,G_N)$.
Note that we have $p = \E_{\beta \sim \Pi}[\beta^2]$. Thus Eq. \eqref{eqn:scf1-in-proof} can be simplified as \begin{equation}\label{eqn:scf-in-proof-reform1}
\delta(\frac{\rho_1+\rho_2}{\nu^2}-\sigma^2)+\delta(\frac1\nu-1)\sigma^2=\E_{\beta_0,G_0,\ldots,G_N}\Big[\frac1N\sum_b \langle(\beta^b-\beta_0)^2 \rangle_{{\bar \mu}} \Big].
\end{equation}
To simplify Eq. \eqref{eqn:scf2-in-proof} and \eqref{eqn:scf3-in-proof}, using Gaussian integration by parts, we have
$$\E\Big[\sum_b \what \beta^b G_b\Big]=\frac{\sqrt{\rho_2}}{\sigma^2}\E_{\beta_0,G_0,\ldots,G_N}\Big[\sum_b\langle (\beta^b)^2\rangle_{\bar \mu}-\langle \beta^b\rangle_{\bar \mu}^2 \Big],$$
$$\E\Big[\sum_b \what \beta^b G_0\Big]=\frac{\sqrt{\rho_1}}{\sigma^2}\E_{\beta_0,G_0,\ldots,G_N}\Big[\sum_b\langle (\beta^b)^2\rangle_{\bar \mu}-\langle \beta^b\rangle_{\bar \mu}^2 \Big].$$
Plugging this into Eq. \eqref{eqn:scf2-in-proof} and \eqref{eqn:scf3-in-proof} gives
\begin{equation}\label{eqn:scf-in-proof-reform2}
\delta(\frac1\nu -1) \sigma^2 =\E_{\beta_0,G_0,\ldots,G_N}\Big[\frac1N\sum_b\langle (\beta^b)^2\rangle_{\bar \mu}-\langle \beta^b\rangle_{\bar \mu}^2 \Big].
\end{equation}
Define $\tau^2 = (\rho_1 + \rho_2)/\nu^2$. Combining Eq. \eqref{eqn:scf-in-proof-reform1} and \eqref{eqn:scf-in-proof-reform2} gives 
$$
\delta(\tau^2-\sigma^2)=\E_{(\beta_0,G)\sim\Pi\times\cN(0,1)}[(\E[\beta_0|\beta_0+\tau G]-\beta_0)^2]. 
$$
This gives the same equation as Eq. (\ref{eqn:self_consistent}), and we denote its fixed point to be $\tau_\star^2$. Furthermore, when $\lambda = 0$, we can see that $(\rho_2)_\star = 0$ leads to a solution of the fixed point equation. Then by Eq. \eqref{eqn:scf-in-proof-reform2} we further get $\nu_\star/\sigma^2 = \tau_\star^2$. Therefore, the distribution of $(\what \beta^1, \ldots, \what \beta^N) |_{\beta_0, G_0, G_1, \ldots, G_N}$ as in Eq. \eqref{eqn:conditional_beta_distribution_in-justrificaiton} becomes conditionally independent
\begin{equation}\label{eqn:conditional_beta_distribution_in-justrificaiton2}
(\what \beta^1,\ldots,\what \beta^N)|_{\beta_0,G_0} \sim {\bar \mu}(\de \bar \bbeta) \propto \prod_b \exp\Big\{-\frac{1}{2\tau_\star^2}\Big[\beta^b-(\beta_0+ \tau_\star  G_0) \Big]^2\Big\}\Pi\big(\de \beta^b \big). 
\end{equation}
That is, we have $(\what \beta^1,\ldots,\what \beta^N) |_{\beta_0, G_0} \sim_{i.i.d.} \cL_{(\beta, Z) \sim \Pi \times \cN(0, 1)}(\beta| \beta+\tau_\star Z = \beta_0 + \tau_\star G_0)$. 

% \sm{I am here} 

% Choose $\rho_1=0$ \tj{I checked both derivation and gaussian integration by parts, and they both look correct...} and make a change of variable $\tau^2=\frac{\rho_2}{\nu^2}$, then the two self-consistent equations can be combined into
% $$\delta(\tau^2-\sigma^2)=\E_{(\beta_0,G)\sim\Pi\times\cN(0,1)}[(\E[\beta_0|\beta_0+\tau G]-\beta_0)^2],$$
% which gives the same equation as Eq. (\ref{eqn:self_consistent}). We denote $\tau_\star$ to be the solution of this equation. 

We finally calculate the $N$ limit of Eq. (\ref{eqn:psi_star_in_formalism1}). By the conditional independence and identical distribution of $(\what \beta^1,\ldots,\what \beta^N) |_{\beta_0, G_0}$, using law of large numbers, taking $N \to \infty$ in Eq. (\ref{eqn:psi_star_in_formalism1}) gives
$$\psi_\star =\lim_{N\rightarrow\infty}\partial_\lambda \phi(\lambda, N) |_{\lambda = 0} = \E_{(\beta_0,G_0)\sim\Pi\times\cN(0,1)}\big[ \psi(\E[g(\beta_0)|\beta_0+\tau_\star G_0],\beta_0)\big].$$
This gives Eq. \eqref{eqn:replica_goal} of Claim \ref{clm:replica}. \qed
% Thus we conclude
% $$\lim_{d\rightarrow\infty}\frac1d\E_{\bX,\beps}\Big[\sum_{i=1}^d \psi(\langle g(\beta_{i}) \rangle_\mu,\beta_{0,i}) \Big] =\E_{(\beta_0,G)\sim\Pi\times\cN(0,1)}\big[ \psi(\E[g(\beta_0)|\beta_0+\tau G],\beta_0)\big]. $$
% Finally, taking $g(x)=x$, we have 
% \[
% \lim_{d\rightarrow\infty}\frac1d\E_{\bX,\beps}\Big[\sum_{i=1}^d \psi(\langle \beta_{i} \rangle_\mu,\beta_{0,i}) \Big] = \E_{(\beta_0,G)\sim\Pi\times\cN(0,1)}\big[ \psi(\cE(\beta_0+\tau G),\beta_0)\big].
% \]
% Taking $g(x)=\ones\{ x = 0\}$, we have 
% \[
% \lim_{d\rightarrow\infty}\frac1d\E_{\bX,\beps}\Big[\sum_{i=1}^d \psi(\langle \ones\{ \beta_{i} = 0\} \rangle_\mu,\beta_{0,i}) \Big] = \E_{(\beta_0,G)\sim\Pi\times\cN(0,1)}\big[ \PostProb(\beta_0+\tau G),\beta_0)\big].
% \]
% These are the desired results. 

\section{Intuitions of Conjecture \ref{conj:PoPCe_PoEdCe_power}}\label{sec:PoPCe_PoEdCe_power}

%We justify Conjecture \ref{conj:PoPCe_PoEdCe_power} by deriving Formalism \ref{fm:pvalue_permu},  \ref{fm:distill_s}, \ref{conj:marginal_nu}, \ref{fm:poedce2} which will be stated later. 

In this section, we provide the intuitions of Conjecture \ref{conj:PoPCe_PoEdCe_power}. We first present Formalism \ref{fm:pvalue_permu} and \ref{fm:poedce2} below, whose intuitions are contained in Section~\ref{sec:justification-crt-pvalues} and \ref{sec:justification-distilled} respectively. We remark that the intuitions of these formalisms are not rigorous proofs, and we leave rigorous proofs as future work. We will provide numerical verifications of these formalisms in Appendix \ref{sec:experiment-verification}. 

\begin{formalism}
\label{fm:pvalue_permu}
Let $(\Xb,\Yb)$ be generated from the Bayesian linear model (Assumption \ref{ass:Bayesian_linear_model}). For $j \in [d]$, let $ P_{ j}$ be the posterior probability of $\beta_{0,j} = 0$ given $(\Xb, \Yb)$, i.e., 
\begin{equation}
 P_{ j}(\Yb, \Xb) = \P(\beta_{0,j} = 0 \vert \Yb, \Xb).
\end{equation}
Let $\hat p_j$ be the {\CRT} p-value corresponding to $P_{j}(\Yb, \Xb)$ in the $K \to \infty$ limit (c.f. Line \ref{line_1_9} of Algorithm \ref{alg:popce}), i.e., 
\begin{equation}\label{eqn:definition-hat-pj-formalism2}
\hat p_j(\Yb, \Xb) = \P_{\txb_j \sim \cN(\bzero, (1/n)\id_n)}( P_{ j}(\Yb, \Xb)\ge P_{ j}(\Yb,  \Xb_{-j}, \txb_j) ).
\end{equation}
 Then for any sufficiently smooth function $\psi: \R \times \R \mapsto \R$, we have %\sm{The right hand side calculate the limit directly? }
% \begin{align*}
% \lim_{d\to \infty, n/d \to \delta} \frac{1}{d}\sum_{j = 1}^d \psi(\beta_{0, j}, \hat p_j) 
% =
% \lim_{d\to \infty, n/d \to \delta} \frac{1}{d}\sum_{j = 1}^d \psi\big(\beta_{0, j},\Psi[\PostProb(\beta_{0, j} + \tau_\star Z_j)]\big) ,
% \end{align*}
\begin{align}\label{eqn:formalism2equation}
\lim_{d\to \infty, n/d \to \delta} \frac{1}{d}\sum_{j = 1}^d \psi(\beta_{0, j}, \hat p_j(\Yb, \Xb)) 
=
\E_{(\beta_0, Z) \sim \Pi \times \cN(0, 1)}[ \psi(\beta_{0},\Psi[\PostProb(\beta_{0} + \tau_\star Z)] )] ,
\end{align}
where $\cP(\,\cdot \,) = \cP(\, \cdot \,; \Pi, \tau_\star)$ is as defined in Eq. \eqref{eqn:posterior_probability_low}, $\tau_\star$ is the unique minimizer to the potential $\phi$ in Eq.~\eqref{eqn:potential_effective_noise}, and $\Psi$ is the CDF of  $\PostProb(\tau_\star Z)$ when $Z \sim \cN(0, 1)$. 
%Then for any sufficiently nice function $\psi: \R^2 \to \R$, we have 
%\[
%\begin{aligned}
%\lim_{n, d \to \infty, n / d \to \delta} \frac{1}{d}\sum_{j = 1}^d \psi(\what \beta_{\Bayes, j}, \beta_{0, j}) =&~ \E_{(\beta_0, Z) \sim \Pi \times \cN(0, 1)}[\psi(\cE(\beta_0 + \tau_\star Z), \beta_0)], \\
%\lim_{n, d \to \infty, n / d \to \delta} \frac{1}{d}\sum_{j = 1}^d \psi(\hat P_{\Bayes, j}, \beta_{0, j}) =&~ \E_{(\beta_0, Z) \sim \Pi \times \cN(0, 1)}[\psi(\PostProb(\beta_0 + \tau_\star Z), \beta_0)]. \\
%\end{aligned}
%\]
\end{formalism}

\begin{formalism}\label{fm:poedce2}
Let $(\Xb,\Yb)$ be generated from the Bayesian linear model (Assumption \ref{ass:Bayesian_linear_model}). For $k \in [d]$, define $\< \cdot \>_{-k}$ to be the ensemble average over the leave-one-out distribution
\[
\mu_{-k}( \de \bbeta_{-k}) \propto \exp\Big\{ - \| \Yb - \Xb_{-k} \bbeta_{-k} \|_2^2 / (2 \sigma^2) \Big\} \prod_{j \neq k} \Pi(\de \beta_j). 
\]
We further define base statistics
\begin{equation}\label{eqn:definition_uk-formalism}
U_k(\Yb,  \Xb) = \cP[(\tau_\star^2 / \sigma^2 ) \< \Yb - \Xb_{-k} \< \bbeta_{-k} \>_{-k}, \xb_k\>].
\end{equation}
Then we let $\tp_j$ be the distilled {\CRT} p-value corresponding to $U_j$ in the $K \to \infty$ limit in Algorithm \ref{alg:poedce}, i.e., 
\begin{equation}\label{eqn:definition-tilde-p-formalism}
\tp_j(\Yb,  \Xb) = \P_{\tilde \xb_j \sim \cN(\bzero, (1/n)\id_n)}( U_j(\Yb,  \Xb) \ge U_j(\Yb,  \Xb_{-j}, \txb_j) ). 
\end{equation}
Then for any sufficiently smooth function $\psi: \R \times \R \mapsto \R$, we have %\sm{The right hand side calculate the limit directly? }
\begin{align}\label{eqn:formalism3equation}
\lim_{d\to \infty, n/d \to \delta} \frac{1}{d}\sum_{j = 1}^d \psi(\beta_{0, j}, \tp_j(\Yb, \Xb)) 
=
\E_{(\beta_0, Z) \sim \Pi \times \cN(0, 1)}[ \psi(\beta_{0},\Psi[\PostProb(\beta_{0} + \tau_\star Z)] )] ,%\lim_{d\to \infty, n/d \to \delta} \frac{1}{d}\sum_{j = 1}^d \psi\big(\beta_{0, j},\Psi[\PostProb(\beta_{0, j} + \tau_\star Z_j)]\big) ,
\end{align}
where $\cP(\,\cdot \,) = \cP(\, \cdot \,; \Pi, \tau_\star)$ is as defined in Eq. \eqref{eqn:posterior_probability_low}, $\tau_\star$ is the unique minimizer to the potential $\phi$ in Eq.~\eqref{eqn:potential_effective_noise}, and $\Psi$ is the CDF of $\PostProb(\tau_\star Z)$ when $Z \sim \cN(0, 1)$.
\end{formalism}

% The justifications of Formalism~\ref{fm:pvalue_permu} and \ref{fm:poedce2} are contained in Section~\ref{sec:justification-crt-pvalues} and \ref{sec:justification-distilled} respectively. We remark that these justifications are non-rigorous since we utilized several conjectured concentration results, which are verified experimentally (c.f. Appendix \ref{sec:experiment-verification}) but hard to be proved rigorously. We leave the rigorous proofs for future work.  %We hope that the justification and formalisms here provide the reader a roadmap towards a rigorous proof of Conjecture \ref{conj:PoPCe_PoEdCe_power}.

Now we use Formalism \ref{fm:post_mean_zero}, \ref{fm:pvalue_permu} and \ref{fm:poedce2} to show Conjecture \ref{conj:PoPCe_PoEdCe_power}. Here we focus on showing the asymptotic optimality of the {\PoPCe} procedure. The intuition for the {\PoEdCe} procedure is the same.

Recall that we have shown that {\CPoP} gives the largest {\mTPR} given {\BFDR} controlled at level $\alpha$ as in Proposition \ref{prop:Bayes_optimality}. We have also derived the limiting {\FDP} and {\TPP} curve of {\TPoP} and {\CPoP} as in Conjecture \ref{conj:Posterior_probability_curve}, and showed that {\TPoP} and {\CPoP} have the same asymptotic {\TPP} and {\FDP} with proper choice of parameters. Therefore, in order to show that {\PoPCe} asymptotically achieves the optimal {\mTPR} given {\BFDR} controlled at level $\alpha$, it suffices to show that {\PoPCe} has the same asymptotic {\mTPR} as the level-$\alpha$ {\TPoP} procedure (so that it has the same asymptotic {\mTPR} as the level-$\alpha$ {\CPoP} procedure and thus it approximately gives the largest {\mTPR} given {\BFDR} controlled at level $\alpha$). 

To show the asymptotic equivalence of {\PoPCe} and {\TPoP}, note that when $K_n \to \infty$, {\PoPCe} (Algorithm \ref{alg:popce}) is equivalent to the following procedure: if $\Psi(t_{\PoPCe}(\alpha - \eps)) \pi_0 d /  |\{j: \Psi^{-1}(\hat p_j) \le t_{\PoPCe}(\alpha - \eps) \}| < \alpha$, we reject the hypotheses $\{ j: \Psi^{-1}(\hat p_j) \le t_{\PoPCe}(\alpha - \eps) \}$, and otherwise we reject nothing. Here the truncation threshold $t_{\PoPCe}$ gives
\begin{equation}\label{eqn:tpopcealpha}
t_{\PoPCe}(\alpha) = \max\Big\{  s \in [0, 1]: \lim_{d\to \infty, n/d \to \delta}\FDP( \bT_P(s;\cdot) ; \Pi) \le \alpha \Big\}. 
\end{equation}
On the other hand, the {\TPoP} procedure (Eq. \eqref{eqn:def_T_P}) rejects $\{ j: P_j(\cD) \le t_{\TPoP}(\alpha)\}$, where 
\begin{equation}\label{eqn:ttpopalpha}
t_{\TPoP}(\alpha) \text{ is such that } \mFDR(\bT_P(t_{\TPoP}(\alpha); \cdot), \Pi) = \alpha. 
\end{equation}
By the concentration property of $\mFDR$ and by Eq. \eqref{eqn:tpopcealpha} and \eqref{eqn:ttpopalpha}, we have that $t_{\TPoP}(\alpha)$ and $t_{\PoPCe}(\alpha)$ are asymptotically the same. Furthermore, by Formalism \ref{fm:pvalue_permu}, we have that %\sm{Explain more this equation}
\[
\lim_{d \to \infty} \Psi(t_{\PoPCe}(\alpha - \eps)) \pi_0 d /  |\{j: \Psi^{-1}(\hat p_j) \le t_{\PoPCe}(\alpha - \eps) \}| = \alpha - \eps. 
\]
Since $\eps = \eps_n$ goes to $0$ slow enough, the inequality $\Psi(t_{\PoPCe}(\alpha - \eps)) \pi_0 d /  |\{j: \Psi^{-1}(\hat p_j) \le t_{\PoPCe}(\alpha - \eps) \}| < \alpha$ can be satisfied with high probability. Therefore, with high probability, {\PoPCe}  rejects the hypotheses $\{ j: \Psi^{-1}(\hat p_j) \le t_{\PoPCe}(\alpha - \eps) \}$. Finally, by Formalism \ref{fm:post_mean_zero} and \ref{fm:pvalue_permu}, for any test function $\psi$, we have
\[
\lim_{d\to \infty, n/d \to \delta} \frac{1}{d}\sum_{j = 1}^d \psi(\beta_{0, j}, P_j(\cD))  = \lim_{d\to \infty, n/d \to \delta} \frac{1}{d}\sum_{j = 1}^d \psi(\beta_{0, j}, \Psi^{-1}(\hat p_j)). 
\]
 The equality above implies that, in terms of limiting {\TPP} and {\FDP}, rejecting $P_j(\cD)$ below any threshold $t$ is equivalent to rejecting $\Psi^{-1}(\hat p_j)$ below the same threshold $t$. Since the rejection threshold $t_{\TPoP}(\alpha)$ of $P_j(\cD)$ in {\TPoP} are asymptotically the same as the rejection threshold $t_{\PoPCe}(\alpha - \eps)$ of $\Psi^{-1}(\hat p_j)$ in {\PoPCe} (module a small $\eps = \eps_n$ that goes to zero sufficiently slow), it follows that {\PoPCe} and {\TPoP} have asymptotically the same {\mTPR}. Finally, notice that Conjecture \ref{conj:Posterior_probability_curve} implies that {\CPoP} and {\TPoP} have asymptotically the same {\mTPR}, we have that {\PoPCe} and {\CPoP} also have asymptotically the same {\mTPR}. This gives Eq. (\ref{eqn:mTPR-T-C}) of Conjecture \ref{conj:PoPCe_PoEdCe_power}. The later statements of Conjecture \ref{conj:PoPCe_PoEdCe_power} follow immediately from Proposition \ref{prop:Bayes_optimality} and Lemma \ref{lem:bound}. 

We next provide the intuitions of Formalism~\ref{fm:pvalue_permu} and \ref{fm:poedce2}.

\subsection{Intuitions of Formalism \ref{fm:pvalue_permu}: Distribution of the CRT p-values}\label{sec:justification-crt-pvalues}~

Recall the definition of the {\CRT} p-value $\hat p_j$ as in Eq. \eqref{eqn:definition-hat-pj-formalism2}, and recall that $P_j(\Yb, \Xb) = \P(\beta_{0, j} = 0 | \Yb, \Xb)$ is the local fdr of the $j$-th hypothesis. Based on heuristic derivations and numerical simulations, we claim that the {\CRT} p-values $( \hat p_j(\Yb, \Xb) )_{j \in [d]}$ and the $\Psi$-transformed local fdrs $(\Psi( P_{ j}(\Yb,  \Xb) ))_{j \in [d]}$ are very close:
\begin{align}
&~\lim_{d\to\infty,n/d\to\delta} \frac{1}{d}\sum_{j = 1}^d \Big( \hat p_j(\Yb, \Xb)- \Psi( P_{ j}(\Yb,  \Xb) ) \Big)^2  = 0.   \label{eq:permu_distribution_1}
\end{align} 
Let us admit this claim for now. Moreover,  Formalism~\ref{fm:post_mean_zero} gives that, for any sufficiently smooth function $\psi: \R \times \R \mapsto \R$, 
\begin{align}
\lim_{d\to \infty, n/d \to \delta} \frac{1}{d}\sum_{j = 1}^d \psi(\beta_{0, j}, P_j(\Yb,\Xb))  = &~ \E_{(\beta_0,Z)\sim\Pi\times\cN(0,1)}[ \psi(\beta_0, \PostProb(\beta_{0} + \tau_\star Z) )]. 
\label{eq:formalism_1_revisit}
\end{align}
Therefore, for any sufficiently smooth function $\tilde\psi: \R \times \R \mapsto \R$,  taking $\psi(x,y)=\tilde\psi(x,\Psi(y))$ in Eq.~\eqref{eq:formalism_1_revisit} and combining it with the claimed Eq.~\eqref{eq:permu_distribution_1}, we get
 \begin{align*}
 &~\lim_{d\to \infty, n/d \to \delta} \frac{1}{d}\sum_{j = 1}^d \tilde\psi(\beta_{0, j}, \hat p_j(\Yb,\Xb))\\ 
=&~
\lim_{d\to \infty, n/d \to \delta} \frac{1}{d}\sum_{j = 1}^d \tilde\psi(\beta_{0, j}, \Psi(P_j(\Yb,\Xb))) 
= \E_{(\beta_0, Z) \sim \Pi \times \cN(0, 1)}[ \tilde\psi(\beta_{0},\Psi[\PostProb(\beta_{0} + \tau_\star Z)] )].
\end{align*}
This is the conclusion of Formalism \ref{fm:pvalue_permu}. 

We are thus left to provide intuitions for the claim as in Eq.~\eqref{eq:permu_distribution_1}. Note that $\hat p_j$ is given by 
\[
\hat p_j(\Yb, \Xb) = \P_{\txb_j \sim \cN(\bzero, (1/n)\id_n)}( P_{ j}(\Yb, \Xb)\ge P_{ j}(\Yb,  \Xb_{-j}, \txb_j) ).
\]
Therefore, to show that $\hat p_j(\Yb, \Xb) \approx \Psi(P_j(\Yb, \Xb))$, we just need to show that, the distribution of $P_{j}(\Yb,  \Xb_{-j}, \txb_j)$, conditional on $(\Yb,  \Xb)$, is approximately the same as $\cP(\tau_\star Z)$ when $Z \sim \cN(0, 1)$. 

By definition, we can rewrite $P_{ j}(\Yb,  \Xb_{-j}, \txb_j)$ as following: $$ P_{ j}(\Yb,  \Xb_{-j}, \txb_j)= \P(\beta_{0,j}=0|\Yb,\Xb_{-j},\txb_j)=\la \ones\{\beta_{j}=0 \}\ra_{\mu},$$ where 
\begin{align}
\mu\left(\mathrm{d} \beta_j, \mathrm{d} \bbeta_{-j}\right) 
&\propto   
\exp \left\{-\left\|\bY-\Xb_{-j}\bbeta_{-j}-\txb_j\beta_j \right\|_{2}^{2} /\left(2 \sigma^{2}\right)\right\} \Pi(\mathrm{d} \beta_j) \prod_{i\neq j} \Pi\left(\mathrm{d} \beta_{i}\right)\notag \\
&=\exp \left\{-\left\|\Xb_{-j}(\bbeta_{0,-j}-\bbeta_{-j})+\beps+\xb_j\beta_{0,j}-\txb_j\beta_j \right\|_{2}^{2} /\left(2 \sigma^{2}\right)\right\} \Pi(\mathrm{d} \beta_j) \prod_{i\neq j} \Pi\left(\mathrm{d} \beta_{i}\right).\label{eq:formalism_post_prob_measure}
\end{align}
Applying Formalism~\ref{conj:marginal_nu} below, we have
\begin{align}\lim _{d \rightarrow \infty, n / d \rightarrow \delta} \E_{\Xb,\txb_j,\beps,\bbeta_{0}}\Big(\la \ones\{\beta_j = 0\}\ra_\mu-\la \ones\{\beta_j = 0\} \ra_\nu \Big)^2=0, \label{eq:2_norm_conv}
\end{align}
where $\nu$ is a distribution over $\beta_j$, defined as 
\begin{align}
\nu(\mathrm{d} \beta_j)  &\propto \exp \left\{\frac{\beta_j \txb_j^{\top}\left[\Xb_{-j}\left(\bbeta_{0,-j}-\left\langle{\bbeta}_{-j}\right\rangle_{\mu_{-}}\right)+\tbeps\right]}{\sigma^2}-\frac{\left(\|\txb_j\|_{2}^{2}-\zeta_{n}^{2}/\sigma^2\right)}{2\sigma^2}\beta_j^{2}\right\} \Pi(\mathrm{d} \beta_j) \\
\propto&~ \P_{(\beta_j, Z) \sim \Pi \times \cN(0, 1)}\left(\beta_j\middle|\beta_j+\tau_j Z = \tau_j G_j \right). \label{eq:nu_one_dist}
\end{align}
In the above equation, we have
\begin{align}
\tbeps\equiv&~ \beps+\xb_j\beta_{0,j},\nonumber\\
\mu_{-}\left(\mathrm{d} {\bbeta}_{-j}\right) \propto&~ \exp \left\{-\left\|\Xb_{-j}{\bbeta}_{0,-j}+\tbeps-\Xb_{-j}{\bbeta}_{-j}\right\|_{2}^{2} /\left(2 \sigma^{2}\right)\right\} \prod_{i\neq j} \Pi\left(\mathrm{d} \beta_{-j, i}\right),\label{eq:mu_minus_dist}
\end{align}
$\< \cdot \>_{\mu_-}$ is the ensemble average with respect to $\mu_-$, and
\begin{align}
\zeta_n^2 \equiv&~ \|\Xb_{-j} (\bbeta_{0,-j}-\langle{\bbeta}_{-j}\rangle_{\mu_{-}})\|_{2}^{2} / n,\\
\tau_j \equiv&~ 1/\sqrt{({\|\txb_j\|_{2}^{2}-\zeta_n^2/\sigma^2})/\sigma^2},\nonumber\\
G_j \equiv &~ \frac{\txb_j^{\top}\left[\Xb_{-j}\left(\bbeta_{0,-j}-\langle{\bbeta}_{-j}\rangle_{\mu_{-}}\right)+\tbeps\right]}{\sigma \sqrt{\|\txb_j\|_{2}^{2}-\zeta_n^2/\sigma^2}}. \nonumber
\end{align}

By Eq.~(4) in \cite{barbier2016mutual} and by Eq.~\eqref{eqn:equi_risk}, we have 
\begin{align}\zeta_n^2&\overset{p}{\to} \lim_{n,d\to\infty,n/d\to\delta} \E\zeta_n^2
=\frac{\sigma^2\E_{(\beta_0, G) \sim \Pi \times \cN(0, 1)}[(\beta_0 - \cE(\beta_0 + \tau G; \Pi, \tau))^2]}{\sigma^2\delta+\E_{(\beta_0, G) \sim \Pi \times \cN(0, 1)}[(\beta_0 - \cE(\beta_0 + \tau G; \Pi, \tau))^2]} =
\frac{\sigma^2(\tau_\star^{2}-\sigma^2)}{\tau_\star^{2}}.
\label{zeta_n_limit}\end{align}
Furthermore, by Eq.~(26) in \cite{barbier2016mutual}, we obtain
\begin{align}\frac{\| \Xb_{-j}(\bbeta_{0,-j}-\langle{\bbeta}_{-j}\rangle_{\mu_{-}})+\tbeps\|_2^2}{n}=\frac{\|\tbeps\|^2}{n}+\zeta_n^2+\frac{2\big\la \tbeps, \Xb_{-j}\left(\bbeta_{0,-j}-\langle{\bbeta}_{-j}\rangle_{\mu_-}\right) \big\ra}{n}\overset{p}{\to}\sigma^2-\lim_{n,d\to\infty,n/d\to\delta} \E\zeta_n^2 =\frac{\sigma^4}{\tau_\star^{2}}. \label{ymmse_and_noise}
\end{align}
Moreover, we have $\|\txb_j\|_2^2 \overset{p}{\to} 1$. Combining the convergence results above implies that
\[
\begin{aligned}
\lim_{j \to \infty} \tau_j = &~ 1/\sqrt{(1 - \sigma^2(\tau_\star^2 - \sigma^2) / \tau_\star^2)/\sigma^2} = \tau_\star, \\
G_j \stackrel{d}{\to}&~ \cN\Big(0, \|\Xb_{-j}(\bbeta_{0,-j}-\langle{\bbeta}_{-j}\rangle_{\mu_{-}} )+\tbeps\|_2^2 / (n \sigma \sqrt{\|\txb_j\|_{2}^{2}-\zeta_n^2/\sigma^2}) \Big) \stackrel{d}{\to} \cN(0, 1). 
\end{aligned}
\]
As a consequence, by Eq. \eqref{eq:2_norm_conv} and \eqref{eq:nu_one_dist} and by the definition of $\cP(\,\cdot \,) = \cP(\, \cdot \,; \Pi, \tau_\star)$  as in Eq. \eqref{eqn:posterior_probability_low}, for any set $S \subseteq [0, 1]$, we have
\[
\begin{aligned}
\P_{\tilde \xb_j}(\< \ones\{ \beta_j = 0 \} \>_\mu \in S | \Xb, \beps, \bbeta_0) \approx \P(\< \ones\{ \beta_j = 0\}\>_{\nu} \in S | \Xb, \beps, \bbeta_0) \approx \P_{G \sim \cN(0, 1)}( \cP(\tau_\star G ) \in S  ). 
\end{aligned}
\]
Taking $S = [0, P_j(\Xb, \Yb)] $, we obtain for any fixed $j \in [d]$ that
\[
\begin{aligned}
\hat p_j(\Xb, \Yb) =&~ \P_{\tilde \xb_j}(\< \ones\{ \beta_j = 0 \} \>_\mu \le P_j(\Xb, \Yb) | \Xb, \beps, \bbeta_0)\\ \approx&~ \P(\< \ones\{ \beta_j = 0\}_{\nu} \le P_j(\Xb, \Yb) | \Xb, \beps, \bbeta_0) \approx \P_{G \sim \cN(0, 1)}( \cP(\tau_\star G ) \le P_j(\Xb, \Yb)  ) = \Psi(P_j(\Xb, \Yb)). 
\end{aligned}
\]
Averaging this approximation over $j \in [d]$ gives Eq.~\eqref{eq:permu_distribution_1}.

We finally present Formalism \ref{conj:marginal_nu} and provide its intuitions. 

\begin{formalism}[Marginal distribution of $\mu$ in Formalism~\ref{fm:pvalue_permu}]\label{conj:marginal_nu}

Assume that $\operatorname{supp}\{\Pi\} \subseteq[-B, B]$ for some fixed $0<B<\infty .$ Let $\boldsymbol{Z} \in \mathbb{R}^{n \times(d-1)}$ with $Z_{i j} \sim_{i i d} \mathcal{N}(0,1 / n), \boldsymbol{z} \in \mathbb{R}^{n}$ with $z_{i} \sim_{i i d} \mathcal{N}(0,1 / n), \boldsymbol{\xi}_{-} \in \mathbb{R}^{d-1}$ with $\xi_{-, j} \sim_{\text {i.i.d. }} \Pi\in \ProbFam(\R), \beps\in \mathbb{R}^{n}$ with $\varepsilon_{i} \sim_{i i d} \mathcal{N}\left(0, \sigma^{2}\right)$, and $\xi\in\R$. Let $\boldsymbol{y}=\boldsymbol{Z} \boldsymbol{\xi}_{-}+\boldsymbol{z} \xi+\boldsymbol{\varepsilon}.$ Define measures $\mu \in \ProbFam\left(\mathbb{R}^{d}\right), \mu_{-} \in \ProbFam\left(\mathbb{R}^{d-1}\right), \nu \in \ProbFam(\mathbb{R})$ by (with ensemble average $\langle\cdot\rangle_{\mu},\langle\cdot\rangle_{\mu_{-}}$, and $\left.\langle\cdot\rangle_{\nu}\right)$
$$
\begin{aligned}
\mu\left(\mathrm{d} \beta, \mathrm{d} \boldsymbol{\beta}_{-}\right) & \propto \exp \left\{-\left\|\boldsymbol{y}-\boldsymbol{Z} \boldsymbol{\beta}_{-}-\boldsymbol{z} \beta\right\|_{2}^{2} /\left(2 \sigma^{2}\right)\right\} \Pi(\mathrm{d} \beta) \prod_{i=1}^{d-1} \Pi\left(\mathrm{d} \beta_{-, i}\right), \\
\mu_{-}\left(\mathrm{d} \boldsymbol{\beta}_{-}\right) & \propto \exp \left\{-\left\|\boldsymbol{Z} \boldsymbol{\xi}_{-}+\boldsymbol{\varepsilon}-\boldsymbol{Z} \boldsymbol{\beta}_{-}\right\|_{2}^{2} /\left(2 \sigma^{2}\right)\right\} \prod_{i=1}^{d-1} \Pi\left(\mathrm{d} \beta_{-, i}\right), \\
\nu(\mathrm{d} \beta) & \propto \exp \left\{\frac{(\beta-\boldsymbol{\xi}) \boldsymbol{z}^{\top}\left[\boldsymbol{Z}\left(\boldsymbol{\xi}_{-}-\left\langle\boldsymbol{\beta}_{-}\right\rangle_{\mu_{-}}\right)+\boldsymbol{\varepsilon}\right]}{\sigma^2}-\frac{\left(\|\boldsymbol{z}\|_{2}^{2}-\zeta_{n}^{2}/\sigma^2\right)}{2\sigma^2}(\beta-\xi)^{2}\right\} \Pi(\mathrm{d} \beta),
\end{aligned}
$$
where
$$
\zeta_{n}^{2}=\left\|\boldsymbol{Z}\left(\boldsymbol{\xi}_{-}-\left\langle\boldsymbol{\beta}_{-}\right\rangle_{\mu_{-}}\right)\right\|_{2}^{2} / n.
$$
Then, fix any bounded function $f: \mathbb{R} \rightarrow \mathbb{R}$ and $\xi\in\R$, we have
$$
\lim _{d \rightarrow \infty, n / d \rightarrow \delta} \E_{\bZ,\bz,\beps,\bxi_-}\left(\la f(\beta)\ra_\mu-\la f(\beta)\ra_\nu\right)^2=0.
$$

\end{formalism}
\begin{proof}[Intuitions of Formalism \ref{conj:marginal_nu}]
Without loss of generality, we assume $\sigma^2 = 1$. Define
\[
\begin{aligned}
M(\beta, \bbeta_-) \equiv&~ \exp\Big\{ - (\beta - \xi) \bz^\sT \bZ (\bbeta_- - \< \bbeta_- \>_{\mu_-}) -  \frac{ \zeta_n^2}{2}(\beta - \xi)^2 \Big\}, \\
h(\beta,\bbeta_-) \equiv&~ (\beta - \xi) \cdot \bz^\sT (\bZ \bxi_- + \beps - \bZ \bbeta_-) - \frac{1}{2} \| \bz \|_2^2 (\beta - \xi)^2, \\
u(\beta) \equiv &~ (\beta - \xi) \bz^\sT [\bZ (\bxi_- - \< \bbeta_- \>_{\mu_-}) + \beps ] - \frac{(\| \bz \|_2^2 - \zeta_n^2)}{2}(\beta - \xi)^2. 
\end{aligned}
\]
Then we have $\nu(\de \beta) = \exp\{ u(\beta)\} \Pi(\de \beta)$, and
\[
\begin{aligned}
\exp\{ h(\beta,\bbeta_-) \} = M(\beta, \bbeta_-) \exp\{ u(\beta)\}.  
\end{aligned}
\]
For any $f: \R \to \R$ (as a function of $\beta$), we have 
\[
\begin{aligned}
\< f(\beta) \>_\mu = \frac{\int f (\beta) \< \exp\{ h(\beta,\bbeta_-)\} \>_{\mu_-}  \Pi(\de \beta)}{\int \< \exp\{ h(\beta,\bbeta_-) \} \>_{\mu_-} \Pi(\de \beta)} = \frac{\< f (\beta) \< M(\beta, \bbeta_-) \>_{\mu_-}  \>_\nu}{\< \<  M(\beta, \bbeta_-) \>_{\mu_-} \>_\nu}. 
\end{aligned}
\]
This gives 
\[
\begin{aligned}
\< f(\beta) \>_\mu - \< f(\beta) \>_\nu =&~ \frac{\< f (\beta) \< M \>_{\mu_-}  \>_\nu}{\< \<  M \>_{\mu_-} \>_\nu}  \< \< 1 - M \>_{\mu_-} \>_\nu + \< f (\beta) \>_\nu \< \<  M - 1 \>_{\mu_-} \>_\nu, 
\end{aligned}
\]
so that 
\begin{equation}\label{eqn:lem:square_difference}
\Big\vert \< f(\beta) \>_\mu - \< f(\beta) \>_\nu \Big\vert  \le  2 \| f \|_\infty \Big\vert \< \<  M -1 \>_{\mu_-} \>_\nu \Big\vert. 
\end{equation}

Next, to upper bound $\E(\< f(\beta) \>_\mu - \< f(\beta) \>_\nu )^2$, we define the event 
\begin{equation}\label{eqn:Def_cE_eps_in_pproof}
\cE_\eps = \Big\{ \sup_{\beta \in [-B, B]} \<  M -1 \>_{\mu_-}^2 \le \eps \Big\}
\end{equation}
and let $F =\sup_{\beta \in [-B, B]}\vert f(\beta) \vert<\infty$. Then 
\begin{equation}\label{eqn:important_equation_f_difference_in_formalism_proof}
\begin{aligned}
 \E_{\bz, \bZ, \bxi_-, \beps}\Big(\< f(\beta) \>_\mu - \< f(\beta) \>_\nu \Big)^2
=&\E_{\bz, \bZ, \bxi_-, \beps}\Big[\Big(\< f(\beta) \>_\mu - \< f(\beta) \>_\nu \Big)^2 (\ones(\cE_\eps^c) + \ones(\cE_\eps))\Big]
\\
\le&~  4 F^2 \Big[ \P_{\bz, \bZ, \bxi_-, \beps}(\< \<  M -1 \>_{\mu_-}^2 \>_{\nu} \ge \eps) + \eps \Big]\\
\le&~ 4 F^2 \Big[ \P_{\bz, \bZ, \bxi_-, \beps}(\cE_\eps^c) + \eps \Big], \\
\end{aligned}
\end{equation}
where the first inequality uses Eq.~\eqref{eqn:lem:square_difference}. As a consequence, Formalism \ref{conj:marginal_nu} holds as long as $\lim_{d \to \infty, n/d \to \delta} \P_{\bz, \bZ, \bxi_-, \beps}(\cE_\eps^c) = 0$. 

% Moreover,  this is true if Claim~\ref{lem:M_concentration} holds. \sm{modify}

% We now claim that, for any $\xi \in [-B, B]$
% \begin{equation}
% \lim_{d \to \infty, n/d \to \delta} \P_{\bz, \bZ, \bxi_-, \beps}\Big( \sup_{\beta \in [-B, B]} \<  M -1 \>_{\mu_-}^2 \ge \eps \Big) = 0. 
% \end{equation}
We show $\lim_{d \to \infty, n/d \to \delta} \P_{\bz, \bZ, \bxi_-, \beps}(\cE_\eps^c) = 0$ using an interpolation method. Denote
\[
V(x) = \exp \Big\{ - (\beta - \xi) x -  \frac{ \zeta_n^2}{2}(\beta - \xi)^2 \Big\} - 1. 
\]
Let $G_1, G_2 \sim_{i.i.d.} \cN(0, 1)$, and $S_j(t) = \sqrt{t}  \bz^\sT (\bZ (\bbeta_-^{(j)} - \< \bbeta_- \>_{\mu_-})) + \sqrt{1 - t} \zeta_n G_j$. Define 
\[
\ell(t) = \E_{\bz, G_1, G_2}[\< V(S_1(t)) V(S_2(t)) \>_{\mu_-^{\otimes 2}} ], 
\]
where $\< O(\bbeta_-^{(1)}, \bbeta_-^{(2)})\>_{\mu_-^{\otimes 2}}$ stands for the expectation of $O(\bbeta_-^{(1)}, \bbeta_-^{(2)})$ with respect to $(\bbeta_-^{(1)}, \bbeta_-^{(2)}) \sim \mu_- \times \mu_-$. Then we have $\ell(0) = 0$, and 
\begin{equation}\label{eqn:ell_M_relation_in_proof}
\ell(1) = \E_{\bz} \Big[ \Big\< \exp \Big\{ - (\beta - \xi)  \bz^\sT \bZ (\bbeta_- - \< \bbeta_- \>_{\mu_-})  -  \frac{ \zeta_n^2}{2}(\beta - \xi)^2 \Big\} - 1  \Big\>_{\mu_-}^2 \Big] = \< M - 1 \>_{\mu_-}^2. 
\end{equation}
Furthermore, calculating the derivative of $\ell$ and using the Stein's formula, we have
\begin{equation}\label{eqn:ell_prime}
\begin{aligned}
\ell'(t) =&~ \Big\< \frac{1}{2} \sum_{ij=1}^2 \Big( \frac{1}{n}\< \bZ(\bbeta_-^{(1)} - \< \bbeta_- \>_{\mu_-}), \bZ(\bbeta_-^{(2)} - \< \bbeta_- \>_{\mu_-})\> - \zeta_n^2 \ones\{i = j\} \Big) \times \E_{\bz, G_1, G_2}\Big[  \partial^2_{ij} [V(S_1(t))  V(S_2(t))] \Big] \Big\>_{\mu_-^{\otimes 2}} \\
\le&~ 2 \Big( \Big\< \Big( \frac{1}{n} \| \bZ(\bbeta_- - \< \bbeta_- \>_{\mu_-}) \|_2^2 - \zeta_n^2 \Big)^2 \Big\>_{\mu_-} +  \Big\< \frac{1}{n^2} \< \bZ(\bbeta_-^{(1)} - \< \bbeta_- \>_{\mu_-}), \bZ(\bbeta_-^{(2)} - \< \bbeta_- \>_{\mu_-})\>^2 \Big\>_{\mu_-^{\otimes 2}} \Big)^{1/2} \\
&~ \times \Big( \< \E[\partial_x V(S_1(t))^4] \>_{\mu_-} + \< \E[\partial_x^2 V(S_1(t))^2] \>_{\mu_-} \< \E[ V(S_1(t))^2] \>_{\mu_-} \Big)^{1/2}. 
\end{aligned}
\end{equation}
To upper bound $\ell'(t)$, note that when $\Pi$ is supported in $[-B, B]$ and $\beta, \xi \in [-B, B]$, we have (for some universal constant $K$)
\[
\begin{aligned}
&~\< \E[\partial_x V(S_1(t))^4] \>_{\mu_-} \\
=&~ (\beta - \xi)^4\exp \Big\{ - 2 \zeta_n^2(\beta - \xi)^2 \Big\} \< \E[  \exp\{ - 4 (\sqrt{t}  \bz^\sT (\bZ (\bbeta_-^{(j)} - \< \bbeta_- \>_{\mu_-})) + \sqrt{1 - t} \zeta_n G ) (\beta - \xi)\}] \>_{\mu_-} \\
\le&~(\beta - \xi)^4\exp \Big\{ 6 \zeta_n^2 (\beta - \xi)^2 \Big\} \<  \exp\{ 8  (\| \bZ (\bbeta_- - \< \bbeta_- \>_{\mu_-}) \|_2^2 /n) (\beta - \xi)^2\} \>_{\mu_-}\\
\le&~ K B^4 \exp\{ K B^4 \| \bZ \|_{\op}^2 \}. 
\end{aligned}
\]
Similarly, we have 
\[
\< \E[\partial_x^2 V(S_1(t))^2] \>_{\mu_-} \le K B^4 \exp\{ K B^4  \| \bZ \|_{\op}^2 \}, ~~ \< \E[ V(S_1(t))^2] \>_{\mu_-} \le K  \Big( \exp\{ K B^4  \| \bZ \|_{\op}^2 \} + 1\Big). 
\]
Denote $\Gamma(\bZ) = K [(B^2 + 1) \exp\{ K B^4 \| \bZ \|_{\op}^2 \} + 1]$. Combining the above bounds with Eq.~\eqref{eqn:ell_prime} and \eqref{eqn:ell_M_relation_in_proof}, we have 
\begin{align}
\E_{\bxi, \beps}[\< M - 1 \>_{\mu_-}^2] = \E_{\bxi, \beps}[ \ell(1)] \le \int_0^1 \E[| \ell'(t) |] \de t \le \Gamma(\bZ) \cdot (E_1 + E_2)^{1/2}, \label{eq:M_concentrate}
\end{align}
where
\[
\begin{aligned}
E_1 \equiv&~ \E_{\bxi, \beps} \Big\< \Big( \frac{1}{n} \| \bZ(\bbeta_- - \< \bbeta_- \>_{\mu_-}) \|_2^2 - \zeta_n^2 \Big)^2 \Big\>_{\mu_-} =  \E_{\bxi, \beps} \Big( \frac{1}{n} \| \bZ(\bxi_{-} - \< \bbeta_- \>_{\mu_-}) \|_2^2 - \zeta_n^2 \Big)^2 , \\
E_2 \equiv&~ \E_{\bxi, \beps} \Big\< \frac{1}{n^2} \< \bZ(\bbeta_-^{(1)} - \< \bbeta_- \>_{\mu_-}), \bZ(\bbeta_-^{(2)} - \< \bbeta_- \>_{\mu_-})\>^2 \Big\>_{\mu_-^{\otimes 2}}. 
\end{aligned}
\]
We believe that $\frac{1}{n} \| \bZ(\bxi_{-} - \< \bbeta_- \>_{\mu_-}) \|_2^2$ will concentrate around its expectation $\zeta_n^2$, and $\frac{1}{n^2} \< \bZ(\bbeta_-^{(1)} - \< \bbeta_- \>_{\mu_-}), \bZ(\bbeta_-^{(2)} - \< \bbeta_- \>_{\mu_-})\>$ will concentrate around its expectation $0$, so that $E_1,E_2$ converge to zero uniformly over $\beta \in [-B, B]$ as $n\to\infty$. However, due to technical difficulties, we make this as a conjecture and leave it open for future work. Then by Eq. \eqref{eq:M_concentrate} and by that $E_1, E_2 \to 0$, we have $\E_{\bxi, \beps}[\< M - 1 \>_{\mu_-}^2] \to 0$, so that by Eq. \eqref{eqn:Def_cE_eps_in_pproof}, we have
\[
\lim_{d \to \infty, n/d \to \delta} \P_{\bz, \bZ, \bxi_-, \beps}(\cE_\eps^c) = 0. 
\]
Combining with Eq. \eqref{eqn:important_equation_f_difference_in_formalism_proof} gives the conclusion of the formalism. \end{proof}

\subsection{Intuitions of Formalism \ref{fm:poedce2}: Distribution of the distilled statistics}\label{sec:justification-distilled}

~

% In the definition of the limiting distilled {\CRT} p-value $\tilde p_j$ as in Eq. (\ref{eqn:definition-tilde-p-formalism}), 
Since $\txb_k \sim \cN(\bzero, (1/n) \id_n )$ is independent of $(\Xb, \Yb)$ and note that $\< \cdot \>_{-k}$ does not depend on $\txb_k$, it follows that 
\[
\< \Yb - \Xb_{-k} \< \bbeta_{-k} \>_{-k}, \txb_k\>\overset{d}{\to} \cN(0, \| \Yb - \Xb_{-k} \< \bbeta_{-k} \>_{-k} \|_2^2/n ) \overset{d}{\to} \cN(0,\sigma^4/\tau_{\star}^{2}), 
\]
where the last convergence is by Eq.~\eqref{ymmse_and_noise}. Therefore, $U_j$ as defined in Eq. \eqref{eqn:definition_uk-formalism} satisfies $U_j(\Yb,  \Xb_{-j}, \txb_j) \overset{d}{\to}\PostProb(\tau_\star Z)$, where $Z \sim\cN(0,1)$. Since $\Psi$ is the CDF of $\PostProb(\tau_\star Z)$ when $Z \sim\cN(0,1)$, we have that for fixed $j \in [d]$, 
\[
\tp_j(\Yb,  \Xb) = \P_{\tilde \xb_j \sim \cN(\bzero, (1/n)\id_n)}( U_j(\Yb,  \Xb) \ge U_j(\Yb,  \Xb_{-j}, \txb_j) ) \stackrel{d}{\approx} \Psi(U_j(\Yb,  \Xb)). 
\]
Due to this approximation, we also expect that 
\begin{align}\label{eqn:tildep_Psi_formalism_proof}
\lim_{d\to \infty, n/d \to \delta} \frac{1}{d}\sum_{j = 1}^d \psi(\beta_{0, j}, \tp_j(\Yb, \Xb)) 
=
\lim_{d\to \infty, n/d \to \delta} \frac{1}{d}\sum_{j = 1}^d \psi(\beta_{0, j}, \Psi(U_j(\Yb,  \Xb))). 
\end{align}
Therefore, to show Formalism \ref{fm:poedce2}, we just need to derive the asymptotic empirical distribution of $\{ (\beta_{0, j}, \Psi[U_j(\Yb, \Xb))]  \}_{j \in [d]}$, which is given by the following formalism. 
\begin{formalism}\label{fm:distill_s}
Let $(\Yb, \Xb)$ be generated from the Bayesian linear model (Assumption \ref{ass:Bayesian_linear_model}). Define $\< \cdot \>_{-k}$ to be the ensemble average over 
\[
\mu_{-k}( \de \bbeta_{-k}) \propto \exp\Big\{ - \| \Yb - \Xb_{-k} \bbeta_{-k} \|_2^2 / (2 \sigma^2) \Big\} \prod_{j \neq k} \Pi(\de \beta_j). 
\]
Define 
\[
S_k(\Yb,  \Xb) = \< \Yb - \Xb_{-k} \< \bbeta_{-k} \>_{-k}, \xb_k\>.
\]
Then in the $n, d \to \infty$ and $n / d \to \delta$ asymptotics, we have for sufficiently smooth $\psi:\R\times\R\mapsto\R$  that
\begin{align}\label{eqn:limiting-empirical-formalim5}
\lim_{d\to \infty, n/d \to \delta} \frac{1}{d}\sum_{j = 1}^d \psi(\beta_{0, j},  S_j(\Yb, \Xb))
&= \E_{(\beta_0, Z) \sim \Pi \times \cN(0, 1)}[\psi(\beta_{0}, (\sigma^2 / \tau_\star^2 ) \beta_{0} + (\sigma^2 / \tau_\star) Z)]. 
% \lim_{d\to \infty, n/d \to \delta} \frac{1}{d}\sum_{j = 1}^d \psi\big(\beta_{0, j}, (\sigma^2 / \tau_\star^2 ) \beta_{0, j} + (\sigma^2 / \tau_\star) Z_j\big),
\end{align}
\end{formalism}
Using Formalism~\ref{fm:distill_s} and Eq. \eqref{eqn:tildep_Psi_formalism_proof}, we immediately obtain
\begin{align*}
&~\lim_{d\to \infty, n/d \to \delta} \frac{1}{d}\sum_{j = 1}^d \psi(\beta_{0, j}, \tilde p_j(\Yb, \Xb))= \lim_{d\to \infty, n/d \to \delta} \frac{1}{d}\sum_{j = 1}^d \psi(\beta_{0, j}, \Psi(U_j(\bY, \bX)))\\
=&~
\lim_{d\to \infty, n/d \to \delta} \frac{1}{d}\sum_{j = 1}^d \psi(\beta_{0, j}, \Psi[\PostProb( (\tau_\star^2 / \sigma^2) S_j(\Yb, \Xb))]) 
=\E_{(\beta_0, Z) \sim \Pi \times \cN(0, 1)}\psi(\beta_{0},\Psi[\PostProb(\beta_{0} + \tau_\star Z)])].
\end{align*}
This gives Formalism \ref{fm:poedce2}.

\begin{proof}[Intuitions of Formalism \ref{fm:distill_s}]

Here we fix a coordinate $k \in [d]$, and provide the intuition that $(\beta_{0, k}, S_k(\Yb, \Xb))$ has asymptotically the same distribution as $(\beta_0, (\sigma^2 / \tau_\star^2 ) \beta_{0} + (\sigma^2 / \tau_\star) Z))$ where $(\beta_0, Z) \sim \Pi \times \cN(0, 1)$.

We define $\tilde \Yb = \Yb - \xb_k \cdot \beta_{0, k}$. That is, $\tilde \Yb$ is a leave-one-out model
\begin{align*}
\tilde \Yb=\Xb_{-k}\bbeta_{0,-k}+\xb_k\cdot0+\beps. 
\end{align*} 
We further define 
$$
\tilde\mu_{-k}( \de \tilde\bbeta_{-k}) \propto \exp\Big\{ - \| \tilde\Yb - \Xb_{-k} \tilde\bbeta_{-k} \|_2^2 / (2 \sigma^2) \Big\} \prod_{j \neq k} \Pi(\de \tilde\beta_j)
$$
with a shorthand notation $\la\cdot\ra_{-\tilde k}$ denoting the ensemble average over $\tilde\mu_{-k}$. We next define an intermediate quantity 
\begin{align*}
\tilde S_k(\tilde\Yb,  \Xb) \equiv \< \tilde\Yb - \Xb_{-k} \< \tilde\bbeta_{-k} \>_{{ -\tilde k}}, \xb_k\>. 
\end{align*}
By the independence between $\tilde\Yb - \Xb_{-k} \< \tilde\bbeta_{-k} \>_{{ -\tilde k}}$ and $\xb_k$, and by Eq.~\eqref{ymmse_and_noise}, we have
\begin{align}\label{eqn:tildeS_k-limit-formalims5}
\tilde S_k(\tilde\Yb,  \Xb)
\overset{d}{=} \cN(0, \|\tilde\Yb - \Xb_{-k} \< \tilde\bbeta_{-k} \>_{{ -\tilde k}}\|_2^2 / n )\overset{d}{\to}\cN(0,\sigma^4 / \tau_\star^2). 
\end{align}

Furthermore, we define 
\begin{equation}\label{eqn:Delta_k-definition-formalism5}
\begin{aligned}
\Delta_k(\Yb, \Xb) \equiv&~ S_k(\Yb,  \Xb) - \tilde S_k(\tilde\Yb,  \Xb) - (\sigma^2/\tau_\star^{2}) \beta_{0,k}\\
=&~ \< (\Yb - \Xb_{-k} \< \bbeta_{-k} \>_{-k}) - (\tilde \Yb - \Xb_{-k} \< \tilde \bbeta_{-k} \>_{- \tilde k}), \xb_k\> - (\sigma^2/\tau_\star^{2}) \beta_{0,k}\\
=&~ \beta_{0,k} (\|\xb_k\|_2^2 - \sigma^2/\tau_\star^{2} ) +\<  \Xb_{-k} (\< \tilde \bbeta_{-k} \>_{- \tilde k} - \< \bbeta_{-k} \>_{-k}), \xb_k\>.
\end{aligned}
\end{equation}
Let $\beta_{0,k}$ to be fixed and taking expectation over remaining quantities, we obtain
\begin{align}
\lim_{n\to\infty}\E \Delta_k(\Yb,  \Xb) &=\lim_{n\to\infty} \E \< - \Xb_{-k} \< \bbeta_{-k} \>_{-k}, \xb_k\>+ \beta_{0,k} (1 - \sigma^2/\tau_\star^{2} ) \nonumber\\
&= \lim_{n\to\infty}\frac{1}{n}\E \tr(\nabla_{\xb_k}\< \Xb_{-k}(\bbeta_{0,-k}-\bbeta_{-k}) \>_{-k})+ \beta_{0,k} (1 - \sigma^2/\tau_\star^{2} ) \nonumber\\
% &= \lim_{n\to\infty}-\frac{\beta_{0,k}}{n\sigma^2}\left(\E  \< \|\Xb_{-k}(\bbeta_{0,-k}-\bbeta_{-k})\|_2^2 \>_{-k}-\E  \|\< \Xb_{-k}(\bbeta_{0,-k}-\bbeta_{-k})\>_{-k}\|_2^2 \right)+ \beta_{0,k} (1 - \sigma^2/\tau_\star^{2} )\nonumber\\
&=\lim_{n\to\infty}-\frac{\beta_{0,k}}{n\sigma^2}\E  \|\< \Xb_{-k}(\bbeta_{0,-k}-\bbeta_{-k})\>_{-k}\|_2^2+ \beta_{0,k} (1 - \sigma^2/\tau_\star^{2} ) = 0 \label{eq_first_moment_t_k}.
\end{align} 
Here, the first equality uses the fact that $\E[\| \xb_k \|_2^2] = 1$ and $\E[\< \Xb_k \< \tilde \bbeta_{-k} \>_{-\tilde k}, \xb_k\>] = 0$. The second equality uses Stein's lemma. The third equality follows from derivative calculations and the Nishimori's identity. The last equality is by Eq.~\eqref{zeta_n_limit}. 

Furthermore, we believe that $\Delta_k(\Yb, \Xb)$ will concentrate around its mean, i.e.,
\begin{equation}\label{eqn:Delta_k-concentration-formalims5}
\Delta_k(\Yb, \Xb) \stackrel{p}{\to} \E[\Delta_k(\Yb, \Xb)] \to 0.
\end{equation}
This concentration phenomenon is not a simple consequence of any standard concentration inequality. We leave the rigorous proof to future work. 
As a consequence, by the definition of $\Delta_k$ as in Eq. \eqref{eqn:Delta_k-definition-formalism5}, and by Eq. \eqref{eqn:tildeS_k-limit-formalims5} and \eqref{eqn:Delta_k-concentration-formalims5}, this shows that $S_k(\Yb, \Xb) \stackrel{d}{\to} \cN((\sigma^2/\tau_\star^{2}) \beta_{0,k}, \sigma^4 / \tau_\star^2)$, and hence  $(\beta_{0, k}, S_k(\Yb, \Xb))$ has asymptotically the same distribution as $(\beta_0, (\sigma^2 / \tau_\star^2 ) \beta_{0} + (\sigma^2 / \tau_\star) Z))$ where $(\beta_0, Z) \sim \Pi \times \cN(0, 1)$. Hence we expect that Eq. \eqref{eqn:limiting-empirical-formalim5} holds. 
\end{proof}

\section{Verification of formalisms through numerical simulations}\label{sec:experiment-verification}

In this section, we numerically verify Formalism \ref{fm:post_mean_zero}, \ref{fm:pvalue_permu} and \ref{fm:poedce2}. Note that these formalisms have a common form  
\begin{equation}\label{eqn:formalism-verification}
\lim_{d\to \infty, n/d \to \delta} \frac{1}{d}\sum_{j = 1}^d \psi(\beta_{0, j},f_j(\cD)) 
=
\E_{(\beta_0, Z) \sim \Pi \times \cN(0, 1)}[ \psi(\beta_{0},f(\beta_0,Z )],
\end{equation}
where $\{ f_j(\cD) \}_{j \in [d]}$ is the object of interest, and $f(\beta_0,Z )$ is the limiting version of $\{ f_j(\cD) \}_{j \in [d]}$. To numerically verify equations like \eqref{eqn:formalism-verification} hold, we compute the $1$-Wasserstein distance between the empirical distribution of $\{ f_j(\cD) \}_{j \in [d]}$ and the empirical distribution of $\{ f(\beta_{j},Z_j) \}_{j\in[d]}$, where $(\beta_j,Z_j)\simiid \Pi\times \cN(0,1)$. Eq. \eqref{eqn:formalism-verification} implies that the $1$-Wasserstein distance should converge to $0$ as $n, d \to \infty, n / d \to \delta$. To show this holds, in the following figures, we plot the $1$-Wasserstein distance versus the sample size $n$ (in the regime $n / d \to \delta$).

\begin{figure}[h]\centering
\includegraphics[width=0.95\textwidth]{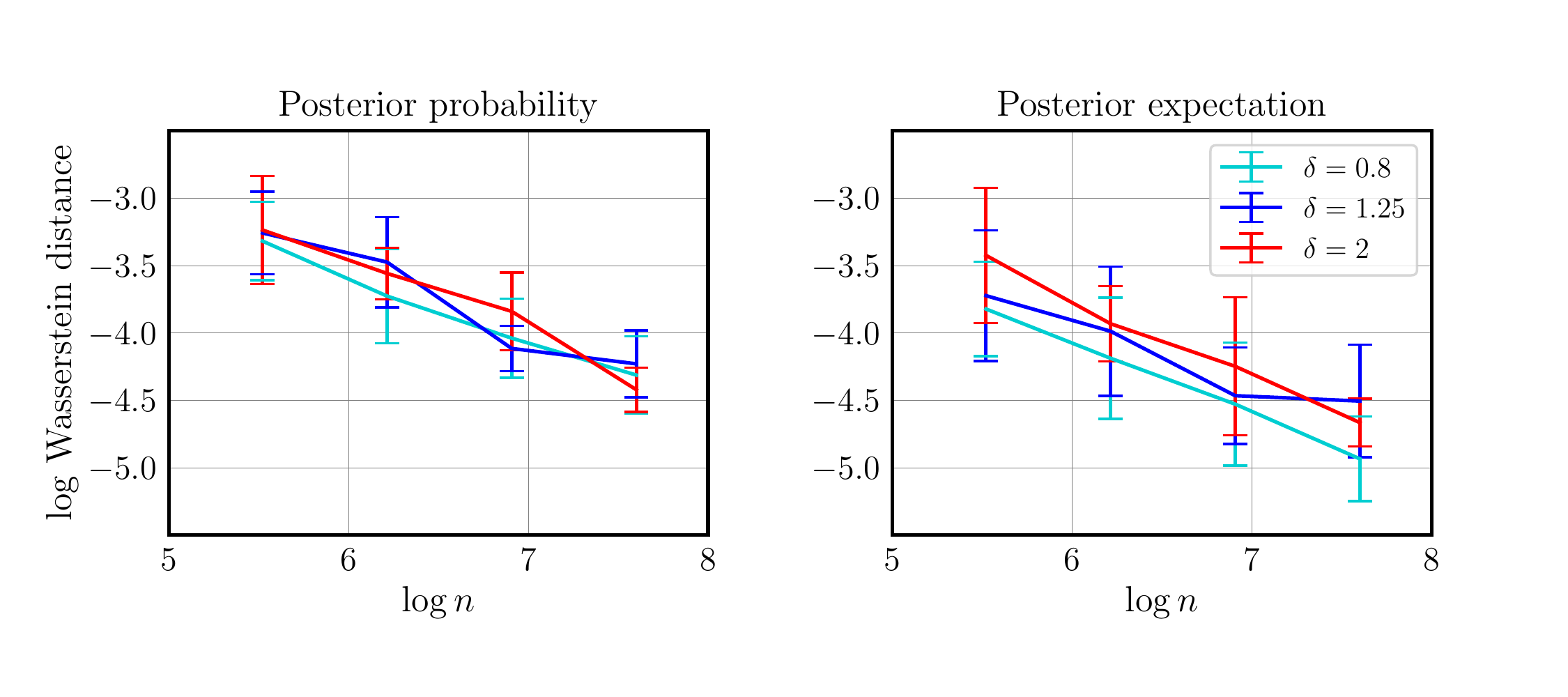}
\caption{Log-log plot of the Wasserstein distance versus the sample size $n$ (for $n=250,500,100,2000$) for Formalism \ref{fm:post_mean_zero}. Left panel: the distance between the local fdrs and the samples from the predicted limiting distribution (Eq. \eqref{eqn:limit-local-fdr-formalism} in Formalism \ref{fm:post_mean_zero}). Right panel: the distance between the posterior expectations and samples from the predicted limiting distribution (Eq. \eqref{eqn:limit-PoE-formalism} in Formalism \ref{fm:post_mean_zero}). The mean curve is averaged over 10 independent instances, and the error bars report the standard deviation across instances. }
\label{fig:formalism1}
\end{figure}

\begin{figure}[h]\centering
\includegraphics[width=0.95\textwidth]{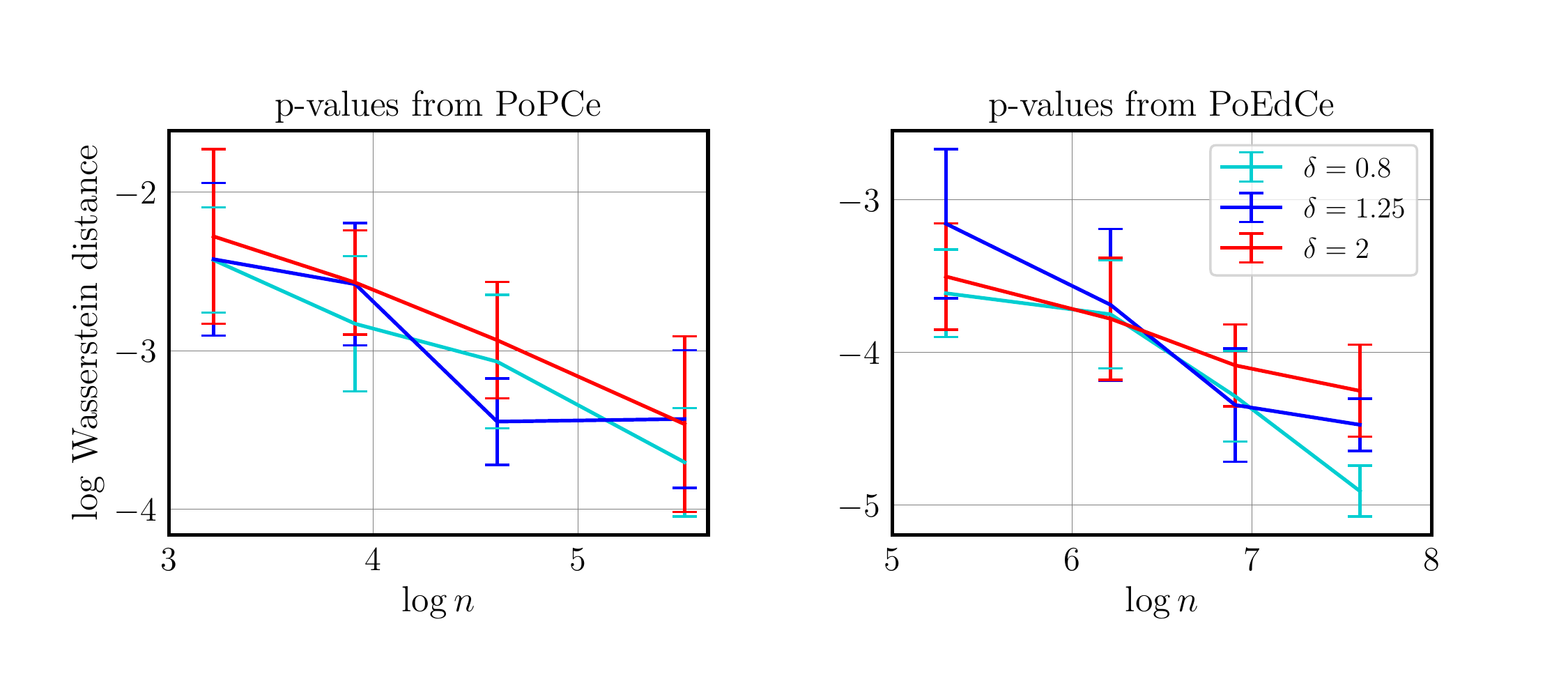}
\caption{Log-log plot of the Wasserstein distance versus the sample size $n$ for Formalism \ref{fm:pvalue_permu} and \ref{fm:poedce2}. 
Left panel: the distance between the {\PoPCe} p-values and the samples from the predicted limiting distribution (Eq. \eqref{eqn:formalism2equation} in Formalism \ref{fm:pvalue_permu}). The sample size grid is chosen to be $n=25,50,100,250$. Right panel: the distance between the {\PoEdCe} p-values and the samples from the predicted limiting distribution (Eq. \eqref{eqn:formalism3equation} in Formalism \ref{fm:poedce2}). The sample size grid is chosen to be $n=250,500,100, 2000$. The mean curve is averaged over 10 independent instances, and the error bars report the standard deviation across instances. }
\label{fig:formalism23}
\end{figure}

Figure \ref{fig:formalism1} is log-log plots of Wasserstein distances against sample sizes $n$ for verifying Formalism \ref{fm:post_mean_zero}. For each choice of $\delta=0.8,1.25, 2$ and each $n=250,500,1000, 2000$, we generate $10$ instances of $(\Yb,\Xb)$ from the Bayesian linear model with model parameters $\sigma=0.25$ and $\Pi=0.6\cdot\delta_0+0.2\cdot\delta_1+0.2\cdot\delta_{-1}$. We also sample $(\beta_j,Z_j) \simiid \Pi\times \cN(0,1)$ for $j\in[d]$. In the left panel, we plot the Wasserstein distance between the posterior means $\{\E[\beta_{0,j}|\cD]\}_{j\in[d]}$ and $\{\cE(\beta_j+\tau_\star Z_j)\}_{j\in[d]}$. In the right panel, we plot the Wasserstein distance between the local fdrs $\{\P(\beta_{0,j}=0|\cD) \}_{j\in[d]}$ and $\{\PostProb(\beta_j+\tau_\star Z_j)\}_{j\in[d]}$. Figure \ref{fig:formalism1} shows that the Wasserstein distances decay to zero as $n \to \infty$, which coincides with the predictions \eqref{eqn:limit-local-fdr-formalism} and \eqref{eqn:limit-PoE-formalism} in Formalism \ref{fm:post_mean_zero}. %The left and right panels of Figure \ref{fig:formalism1} are for posterior probability and posterior mean, respectively. 

Figure \ref{fig:formalism23} is log-log plots of Wasserstein distances against sample sizes $n$ for verifying Formalism \ref{fm:pvalue_permu} and \ref{fm:poedce2}. Similar to the experiments in Figure \ref{fig:formalism1}, we generate $10$ instances of $(\Yb, \Xb)$ with the same aspect ratios $\delta$, noise level $\sigma$, and prior $\Pi$. In the left panel, we plot the Wasserstein distance between the {\CRT} p-values in {\PoPCe} $\{ \hat p_j(\cD) \}_{j \in [d]}$ and $\{ \Psi(\cP(\beta_j + \tau_\star Z_j))\}_{j \in [d]}$, with sample size  $n=25,50,100,250$. In the right panel, we plot the Wasserstein distance between the {\dCRT} p-values in {\PoEdCe} $\{ \tilde p_j(\cD) \}_{j \in [d]}$ and $\{ \Psi(\cP(\beta_j + \tau_\star Z_j))\}_{j \in [d]}$, with sample size $n=250,500,1000,2000$. We choose smaller $n$ in the left panel since it is computationally heavier to calculate {\CRT} p-values.
Figure \ref{fig:formalism1} shows that the Wasserstein distances decay to zero as $n \to \infty$, which coincides with the prediction \eqref{eqn:formalism2equation} in Formalism \ref{fm:pvalue_permu}  and prediction \eqref{eqn:formalism3equation} in Formalism \ref{fm:poedce2}.

% Here we calculate the Wasserstein distance between the empirical distributions of $(\Psi[\PostProb(\beta_{j} + \tau_\star Z_j)])_{j\in[d]}$ and p-values obtained by {\PoPCe} and {\PoEdCe} calculated from the data. In both figures, the Wasserstein distance of empirical distributions from high dimensional model and one dimensional model decreases as $n$ increases. As shown in both figures, their $\log$-$\log$ plot shows linear decay, so that we can expect convergence of the Wasserstein distance to be decaying in exponential rate towards $0$. 

%\bf{If you include a photo:}\vspace{-33pt}
%\begin{IEEEbiography}[{\includegraphics[width=1in,height=1.25in,clip,keepaspectratio]{fig1}}]{Michael Shell}
%Use $\backslash${\tt{begin\{IEEEbiography\}}} and then for the 1st argument use $\backslash${\tt{includegraphics}} to declare and link the author photo.
%Use the author name as the 3rd argument followed by the biography text.
%\end{IEEEbiography}



\end{document}